\documentclass[12pt,reqno]{amsart}
\def\draft{n}
\usepackage{
  fullpage,graphicx,dbnsymb,amsmath,amssymb,picins,mathtools,stmaryrd,color,
  hyperref,multicol
}
\usepackage[all]{xy}
\usepackage[overload]{abraces} % http://ctan.org/pkg/abraces

\if\draft y
  \usepackage[color]{showkeys}
  \usepackage{everypage,draftwatermark}
  \SetWatermarkScale{4}
  \SetWatermarkLightness{0.9}
\fi

\theoremstyle{plain}
\newtheorem{theorem}{Theorem}[section]

\newtheorem{proposition}[theorem]{Proposition}

\newtheorem{lemma}[theorem]{Lemma}

\newtheorem{conjecture}[theorem]{Conjecture}

\newtheorem{property}[theorem]{Property}

\theoremstyle{definition}
\newtheorem{definition}[theorem]{Definition}

\newtheorem{lemmadefinition}[theorem]{Lemma-Definition}

\newtheorem{riddle}[theorem]{Riddle}

\theoremstyle{remark}
\newtheorem{comment}[theorem]{Comment}

\newtheorem{example}[theorem]{Example}
\newtheorem{exercise}[theorem]{Exercise}

\newtheorem{remark}[theorem]{Remark}
\newtheorem{warning}[theorem]{Warning}

\newlength{\standardunitlength}
\catcode`\@=11
\long\def\@makecaption#1#2{%
    \vskip 10pt
    \setbox\@tempboxa\hbox{%\ifvoid\tinybox\else\box\tinybox\fi
      \small\sf{\bfcaptionfont #1. }\ignorespaces #2}%
    \ifdim \wd\@tempboxa >\captionwidth {%
        \rightskip=\@captionmargin\leftskip=\@captionmargin
        \unhbox\@tempboxa\par}%
      \else
        \hbox to\hsize{\hfil\box\@tempboxa\hfil}%
    \fi}
\font\bfcaptionfont=cmssbx10 scaled \magstephalf
\newdimen\@captionmargin\@captionmargin=2\parindent
\newdimen\captionwidth\captionwidth=\hsize
\catcode`\@=12

\newcommand{\ad}{\operatorname{ad}}

\newcommand{\Span}{\operatorname{span}}

\newcommand{\tr}{\operatorname{tr}}

\def\qed{{\hfill\text{$\Box$}}}

\newlength{\globalparindent}
\newenvironment{myitemize}{
  \begin{list}{$\bullet$}{\setlength{\leftmargin}{16pt}
  \setlength{\labelwidth}{12pt}
  \setlength{\labelsep}{4pt}}
}{
  \end{list}
}

\def\arXiv#1{{\href{http://front.math.ucdavis.edu/#1}{arXiv:#1}}}

\def\web#1{{[\href{http://www.math.toronto.edu/~drorbn/papers/KBH/#1}{\tt Web/#1}]}}

\def\baru{{\bar u}}
\def\barv{{\bar v}}
\def\barw{{\bar w}}
\def\barbaru{{\bar\baru}}
\def\bbN{{\mathbb N}}
\def\bbQ{{\mathbb Q}}
\def\bbR{{\mathbb R}}
\def\calA{{\mathcal A}}
\def\calK{{\mathcal K}}
\def\calM{{\mathcal M}}
\def\calO{{\mathcal O}}
\def\calP{{\mathcal P}}

\def\calT{{\mathcal T}}
\def\calU{{\mathcal U}}
\def\calW{{\mathcal W}}
\def\frakg{{\mathfrak g}}

\def\uT{{u\calT}}
\def\vT{{v\calT}}
\def\wT{{w\calT}}

\def\Abh{{\calA^{bh}}}
\def\Kbh{{\calK^{rbh}}}
\def\Kbhd{{\calK^{b=h}}}
\def\Kbhz{{\calK^{rbh}_0}}
\def\Pbh{{\calP^{bh}}}
\def\Zbh{{Z^{bh}}}

\def\remove{\!\setminus\!}

\def\aft{{\overrightarrow{4T}}}
\def\aAS{{\overrightarrow{AS}}}
\def\aSTU{{\overrightarrow{STU}}}
\def\aIHX{{\overrightarrow{IHX}}}
\def\bch{\operatorname{bch}}

\newcommand{\diver}{\operatorname{div}}
\def\CW{\text{\it CW}}
\def\CWr{\text{\it CW}^{\,r}}
\def\FA{\text{\it FA}}
\def\FL{\text{\it FL}}
\newcommand{\Fun}{\operatorname{Fun}}
\def\CA{\operatorname{CA}}

\def\human{\begin{picture}(0,0)%
\includegraphics{figs/human.pstex}%
\end{picture}%
\setlength{\unitlength}{789sp}%
\begingroup\makeatletter\ifx\SetFigFont\undefined%
\gdef\SetFigFont#1#2#3#4#5{%
  \reset@font\fontsize{#1}{#2pt}%
  \fontfamily{#3}\fontseries{#4}\fontshape{#5}%
  \selectfont}%
\fi\endgroup%
\begin{picture}(1226,2425)(-12,-1573)
\end{picture}%
}
\def\machine{\begin{picture}(0,0)%
\includegraphics{figs/machine.pstex}%
\end{picture}%
\setlength{\unitlength}{789sp}%
\begingroup\makeatletter\ifx\SetFigFont\undefined%
\gdef\SetFigFont#1#2#3#4#5{%
  \reset@font\fontsize{#1}{#2pt}%
  \fontfamily{#3}\fontseries{#4}\fontshape{#5}%
  \selectfont}%
\fi\endgroup%
\begin{picture}(1224,1374)(-11,-598)
\end{picture}%
}
\def\mathinclude#1{{%
  \noindent
  \imagetop{\begin{picture}(0,0)%
\includegraphics{figs/human.pstex}%
\end{picture}%
\setlength{\unitlength}{789sp}%
\begingroup\makeatletter\ifx\SetFigFont\undefined%
\gdef\SetFigFont#1#2#3#4#5{%
  \reset@font\fontsize{#1}{#2pt}%
  \fontfamily{#3}\fontseries{#4}\fontshape{#5}%
  \selectfont}%
\fi\endgroup%
\begin{picture}(1226,2425)(-12,-1573)
\end{picture}%
}\
  \imagetop{\includegraphics[scale=0.18]{figs/#1.eps}}
  \newline
}}
\def\dialoginclude#1{{%
  \noindent
  \imagetop{\begin{picture}(0,0)%
\includegraphics{figs/human.pstex}%
\end{picture}%
\setlength{\unitlength}{789sp}%
\begingroup\makeatletter\ifx\SetFigFont\undefined%
\gdef\SetFigFont#1#2#3#4#5{%
  \reset@font\fontsize{#1}{#2pt}%
  \fontfamily{#3}\fontseries{#4}\fontshape{#5}%
  \selectfont}%
\fi\endgroup%
\begin{picture}(1226,2425)(-12,-1573)
\end{picture}%
}\ \imagetop{\includegraphics[scale=0.18]{figs/Human-#1.eps}}
  \newline
  \vskip 1mm\noindent
  \imagetop{\begin{picture}(0,0)%
\includegraphics{figs/machine.pstex}%
\end{picture}%
\setlength{\unitlength}{789sp}%
\begingroup\makeatletter\ifx\SetFigFont\undefined%
\gdef\SetFigFont#1#2#3#4#5{%
  \reset@font\fontsize{#1}{#2pt}%
  \fontfamily{#3}\fontseries{#4}\fontshape{#5}%
  \selectfont}%
\fi\endgroup%
\begin{picture}(1224,1374)(-11,-598)
\end{picture}%
}\ \imagetop{\includegraphics[scale=0.18]{figs/Machine-#1.eps}}
  \newline
}}

\def\imagetop#1{\vtop{\null\hbox{#1}}}

\begin{document}
\newdimen\captionwidth\captionwidth=\hsize
\setcounter{secnumdepth}{4}

\title[Balloons and Hoops]{Balloons and Hoops and their Universal Finite
Type Invariant, BF Theory, and an Ultimate Alexander Invariant}

\author{Dror~Bar-Natan}
\address{
  Department of Mathematics\\
  University of Toronto\\
  Toronto Ontario M5S 2E4\\
  Canada
}
\email{drorbn@math.toronto.edu}
\urladdr{\url{http://www.math.toronto.edu/~drorbn}}

\date{this edition: Aug.~7,~2013, first edition: Aug.~7, 2013}


\begin{abstract}
Balloons are two-dimensional spheres. Hoops are one dimensional loops.
Knotted Balloons and Hoops (KBH) in 4-space behave much like the first and
second homotopy groups of a topological space --- hoops can be composed
as in $\pi_1$, balloons as in $\pi_2$, and hoops ``act'' on balloons as
$\pi_1$ acts on $\pi_2$. We observe that ordinary knots and tangles
in 3-space map into KBH in 4-space and become amalgams of both balloons
and hoops.

We give an ansatz for a tree and wheel (that is, free-Lie and cyclic
word) -valued invariant $\zeta$ of (ribbon) KBHs in terms of the said
compositions and action and we explain its relationship with finite type
invariants. We speculate that $\zeta$ is a complete evaluation of the BF
topological quantum field theory in 4D. We show that a certain ``reduction
and repackaging'' of $\zeta$ is an ``ultimate Alexander invariant''
that contains the Alexander polynomial (multivariable, if you wish), has
extremely good composition properties, is evaluated in a topologically
meaningful way, and is least-wasteful in a computational sense. If you
believe in categorification, that should be a wonderful playground.

\vskip 3mm

\parpic[l]{\includegraphics[height=1in]{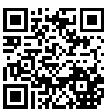}}
\noindent Web resources for this paper are available at
\newline\null\hfill
  \web{}$:=$\url{http://www.math.toronto.edu/~drorbn/papers/KBH/},
\hfill\null\newline
including an electronic version, source files, computer programs, lecture
handouts and lecture videos; the latest of the handouts is attached at
the end of this paper. {\em Throughout this paper we follow the notational
conventions and notations outlined in Section~\ref{subsec:Conventions}.}

\end{abstract}

\maketitle

\tableofcontents

\section{Introduction}

\parpic[r]{\input{figs/TubeT.pstex_t}}
\begin{riddle} \label{riddle:piT} The set of homotopy classes of maps of
a tube $T=S^1\times[0,1]$ into a based topological space $(X,b)$ which
map the rim $\partial T=S^1\times\{0,1\}$ of $T$ to the basepoint $b$
is a group with an obvious ``stacking'' composition; we call that
group $\pi_T(X)$. Homotopy theorists often study $\pi_1(X)=[S^1,X]$ and
$\pi_2(X)=[S^2,X]$ but seldom if ever do they study $\pi_T(X)=[T,X]$. Why?
\end{riddle}

The solution of this riddle is on Page~\pageref{sol:piT}. Whatever it
may be, the moral is that it is better to study the homotopy classes of
circles and spheres in $X$ rather than the homotopy classes of tubes in
$X$, and by morphological transfer, it is better to study isotopy classes
of embeddings of circles and spheres into some ambient space than isotopy
classes of embeddings of tubes into the same space.

In~\cite{WKO}, Zsuzsanna Dancso and I studied the finite-type knot theory
of ribbon tubes in $\bbR^4$ and found it to be closely related to deep
results by Alekseev and Torossian~\cite{AlekseevTorossian:KashiwaraVergne}
on the Kashiwara-Vergne conjecture and Drinfel'd's associators. At some
point we needed a computational tool with which to make and to verify
conjectures.

This paper started in being that computational tool. After a lengthy
search I found a language in which all the operations and equations needed
for~\cite{WKO} could be expressed and computed. Upon reflection, it turned
out that the key to that language was to work with knotted balloons and
hoops, meaning spheres and circles, rather than with knotted tubes.

Then I realized that there may be independent interest in that
computational tool. For (ribbon) knotted balloons and hoops in $\bbR^4$
($\Kbh$, Section~\ref{sec:objects}) in themselves form a lovely
algebraic structure (an MMA, Section~\ref{sec:Operations}), and the
``tool'' is really a well-behaved invariant $\zeta$. More precisely,
$\zeta$ is a ``homomorphism $\zeta$ of the MMA $\Kbhz$ to the MMA $M$
of trees and wheels'' (trees in Section~\ref{sec:trees} and wheels in
Section~\ref{sec:zeta}). Here $\Kbhz$ is a variant of $\Kbh$ defined
using generators and relations (Definition~\ref{def:Kbhz}). Assuming a
sorely missing Reidemeister theory for ribbon-knotted tubes in $\bbR^4$
(Conjecture~\ref{conj:Kbhz}), $\Kbhz$ is actually equal to $\Kbh$.

The invariant $\zeta$ has a rather concise definition that uses only
basic operations written in the language of free Lie algebras. In fact,
a nearly complete definition appears within Figure~\ref{fig:mgaops},
with lesser extras in Figures~\ref{fig:ConnectedSum}
and~\ref{fig:Generators}. These definitions are relatively easy
to implement on a computer, and as that was my original goal, the
implementation along with some computational examples is described
in Section~\ref{sec:Computations}. Further computations, more closely
related to~\cite{AlekseevTorossian:KashiwaraVergne} and to~\cite{WKO},
will be described in a future publication.

In Section~\ref{sec:ft} we sketch a conceptual interpretation of $\zeta$.
Namely, we sketch the statement and the proof of the following theorem:

\begin{theorem} The invariant $\zeta$ is (the logarithm of) a universal
finite type invariant of the objects in $\Kbhz$ (assuming
Conjecture~\ref{conj:Kbhz}, of ribbon-knotted balloons and hoops in
$\bbR^4$).
\end{theorem}

While the formulae defining $\zeta$ are reasonably simple, the proof
that they work using only notions from the language of free Lie
algebras involves some painful computations --- the more reasonable
parts of the proof are embedded within Sections~\ref{sec:trees}
and~\ref{sec:zeta}, and the less reasonable parts are postponed to
Section~\ref{subsec:JProperties}. An added benefit of the results of
Section~\ref{sec:ft} is that they constitute an alternative construction of
$\zeta$ and an alternative proof of its invariance --- the construction
requires more words than the free-Lie construction, yet the proof of
invariance becomes simpler and more conceptual.

In Section~\ref{sec:BF} we discuss the relationship of $\zeta$
with the BF topological quantum field theory, and in Section~\ref{sec:uA}
we explain how a certain reduction of $\zeta$ becomes a system of
formulae for the (multivariable) Alexander polynomial which, in some
senses, is better than any previously available formula.

Section~\ref{sec:odds} is for ``odds and ends'' --- things worth saying,
yet those that are better postponed to the end. This includes the details
of some definitions and proofs, some words about our conventions, and an
attempt at explaining how I think about ``meta'' structures.

\section{The Objects} \label{sec:objects}

\subsection{Ribbon Knotted Balloons and Hoops} \label{subsec:kbh} This
paper is about ribbon-knotted balloons ($S^2$'s) and hoops (or loops, or
$S^1$'s) in $\bbR^4$ or, equivalently, in $S^4$.  Throughout this paper
$T$ and $H$ will denote finite\footnote{The bulk of the paper easily
generalizes to the case where $H$ (not $T$!) is infinite, though nothing
is gained by allowing $H$ to be infinite.} (not necessarily disjoint)
sets of ``labels'', where the labels in $T$ label the balloons (though
for reasons that will become clear later, they are also called ``tail
labels'' and the things they label are sometimes called ``tails''),
and the labels in $H$ label the hoops (though they are sometimes called
``head labels'' and they sometimes label ``heads'').

\parpic[r]{\def\ori{$\circlearrowleft$}\input{figs/modelKBH.pstex_t}}
\begin{definition} \label{def:KBH}
A $(T;H)$-labelled rKBH (ribbon-Knotted Balloons and Hoops)
is a ribbon\footnote{The adjective ``ribbon'' will be explained in
Definition~\ref{def:ribbon} below.}
up-to-isotopy embedding into $\bbR^4$ or into $S^4$ of $|T|$ oriented 2-spheres
labelled by the elements of $T$ (the ``balloons''), of $|H|$ oriented
circles labelled by the elements of $H$ (the ``hoops''), and of $|T|+|H|$
strings (namely, intervals) connecting the $|T|$ balloons and the
$|H|$ hoops to some fixed base point, often denoted $\infty$. Thus a
$(\underline{2};\underline{3})$-labelled\footnote{See ``notational
conventions'', Section~\ref{subsec:Conventions}.} rKBH, for example, is a
ribbon up-to-isotopy embedding into $\bbR^4$ or into $S^4$ of the space
drawn on the right. Let $\Kbh(T;H)$ denote the set of all $(T;H)$-labelled
rKBHs.
\end{definition}

Recall that 1-dimensional objects cannot be knotted in 4-dimensional
space. Hence the hoops in an rKBH are not in themselves knotted, and
hence an rKBH may be viewed as a knotting of the $|T|$ balloons in it,
along with a choice of $|H|$ elements of the fundamental group of the
complement of the balloons. Likewise, the $|T|+|H|$ strings in an rKBH only
matter as homotopy classes of paths in the complement of the balloons. In
particular, they can be modified arbitrarily in the vicinity of $\infty$,
and hence they don't even need to be specified near $\infty$ --- it is
enough that we know that they emerge from a small neighbourhood of $\infty$
(small enough so as to not intersect the balloons) and that each reaches
its balloon or its hoop.

Conveniently we often pick our base point to be the point $\infty$ of
the formula $S^4=\bbR^4\cup\{\infty\}$ and hence we can draw rKBHs in
$\bbR^4$ (meaning, of course, that we draw in $\bbR^2$ and adopt
conventions on how to lift these drawings to $\bbR^4$).

We will usually reserve the labels $x$, $y$, and $z$ for hoops, the labels
$u$, $v$, and $w$ for balloons, and the labels $a$, $b$, and $c$ for things
that could be either balloons or hoops. With almost no risk of ambiguity,
we also use $x$, $y$, $z$, along also with $t$, to denote the coordinates
of $\bbR^4$. Thus $\bbR^2_{xy}$ is the $xy$ plane within $\bbR^4$,
$\bbR^3_{txy}$ is the hyperplane perpendicular to the $z$ axis, and
$\bbR^4_{txyz}$ is just another name for $\bbR^4$.

\begin{figure}
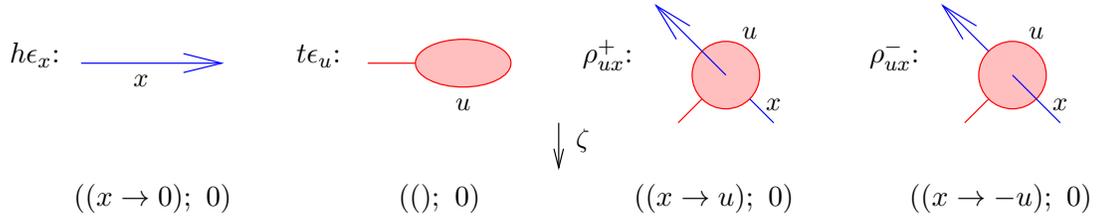

\[
  \def\heps{{$((x\to 0);\ 0)$}}
  \def\teps{{$(();\ 0)$}}
  \def\rhop{{$((x\to u);\ 0)$}}
  \def\rhom{{$((x\to -u);\ 0)$}}
  \input figs/Generators.pstex_t
\]
\caption{
  The four generators $h\epsilon_x$, $t\epsilon_u$, and
  $\rho^\pm_{ux}$, drawn in $\bbR^3_{xyz}$ ($\rho^\pm_{ux}$ differ in
  the direction in which $x$ pierces $u$ --- from below at $\rho^+_{ux}$
  and from above at $\rho^-_{ux}$). The lower part of the figure previews
  the values of the main invariant $\zeta$ discussed in this paper on
  these generators. More later, in Section~\ref{sec:zeta}.
} \label{fig:Generators}
\end{figure}

Examples~\ref{exa:generators} and~\ref{exa:Sophisticated} below are more
than just examples, for they introduce much notation that we use later on.

\begin{example} \label{exa:generators} The first four examples of rKBHs
are the ``four generators'' shown in Figure~\ref{fig:Generators}:

\begin{myitemize}
\item $h\epsilon_x$ is an
element of $\Kbh(;x)$ (more precisely, $\Kbh(\emptyset;\{x\})$). It has
a single hoop extending from near $\infty$ and back to near $\infty$,
and as indicated above, we didn't bother to indicate how it closes near
$\infty$ and how it is connected to $\infty$ with an extra piece of
string. Clearly, $h\epsilon_x$ is the ``unknotted hoop''.
\end{myitemize}
\begin{myitemize}
\parpic[r]{\parbox{1.5in}{
  \includegraphics[width=1.5in]{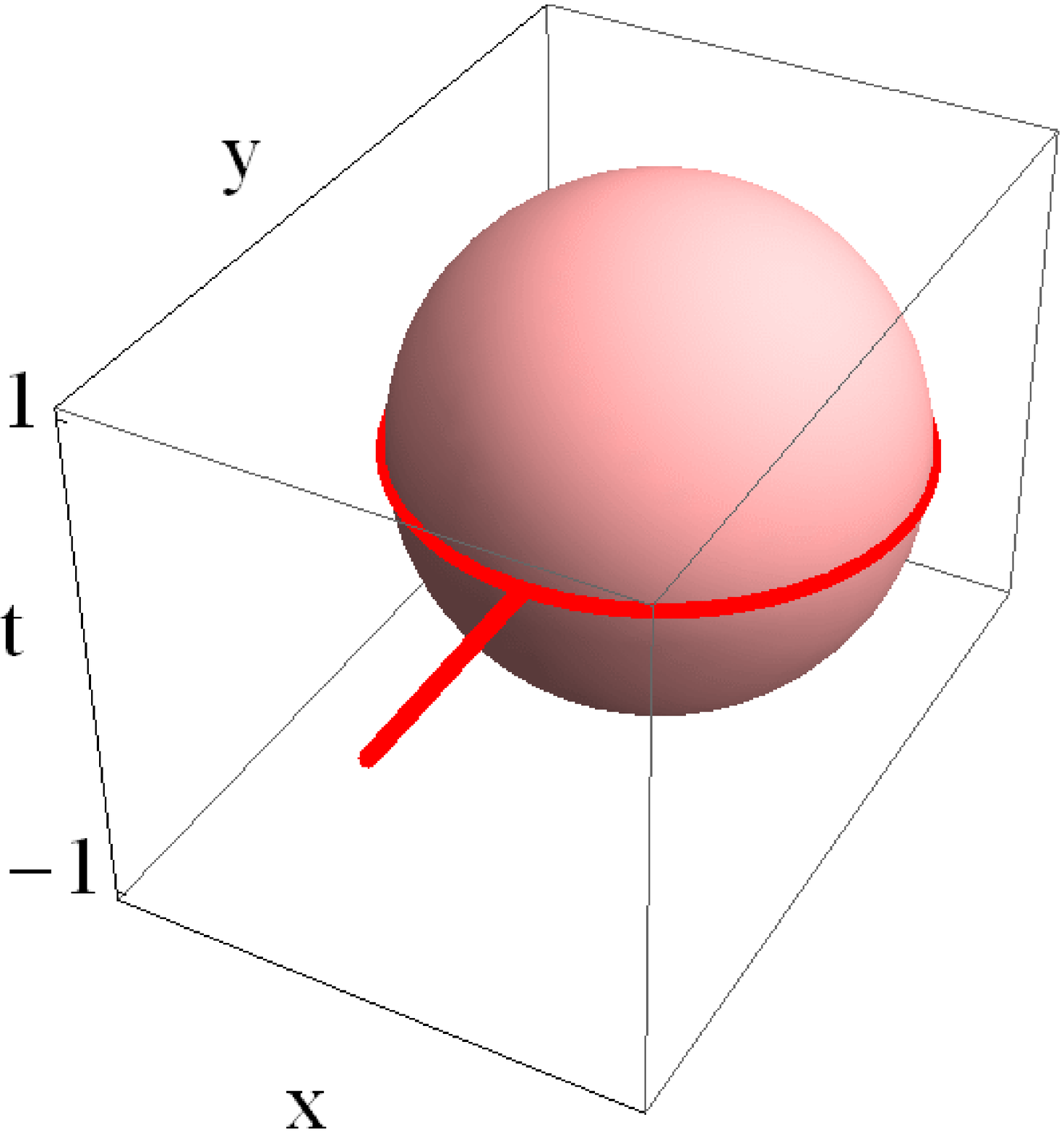}
  \footnotesize Warning: the vertical direction here is the ``time''
  coordinate $t$, so this is an $\bbR^3_{txy}$ picture.
  \newline\ 
}}
\item $t\epsilon_u$ is an
element of $\Kbh(u;)$. As a picture in $\bbR^3_{xyz}$,
it looks like a simplified tennis racket, consisting of a handle, a rim,
and a net. To interpret a tennis racket in $\bbR^4$, we embed $\bbR^3_{xyz}$
into $\bbR^4_{txyz}$ as the hyperplane $[t=0]$, and inside it we place
the handle
and the rim as they were placed in $\bbR^3_{xyz}$. We also make two copies of
the net, the ``upper'' copy and the ``lower'' copy. We place the upper
copy so that its boundary is the rim and so that its interior is pushed
into the $[t>0]$ half-space (relative to the pictured $[t=0]$ placement)
by an amount proportional to the distance from the boundary. Similarly we
place the lower copy, except we push it into the $[t<0]$ half space. Thus
the two nets along with the rim make a 2-sphere in $\bbR^4$, which is
connected to $\infty$ using the handle. Clearly, $t\epsilon_u$ is the
``unknotted balloon''. We orient $t\epsilon_u$ by adopting the conventions
that surfaces drawn in the plane are oriented counterclockwise (unless
otherwise noted) and that when pushed to $4D$, the upper copy retains
the original orientation while the lower copy reverses it.
\end{myitemize}
\begin{myitemize}
\item $\rho^+_{ux}$ is an element of $\Kbh(u;x)$. It is the 4D analog of
the ``positive Hopf link''. Its picture in Figure~\ref{fig:Generators}
should be interpreted in much the same way as the previous two ---
what is displayed should be interpreted as a 3D picture using standard
conventions (what's hidden is ``below''), and then it should be placed
within the $[t=0]$ copy of $\bbR^3_{xyz}$ in $\bbR^4$. This done, the
racket's net should be split into two copies, one to be pushed to $[t>0]$
and the other to $[t<0]$. In $\bbR^3_{xyz}$ it appears as if the hoop
$x$ intersects the balloon $u$ right in the middle. Yet in $\bbR^4$ our
picture represents a legitimate knot as the hoop is embedded in $[t=0]$,
the nets are pushed to $[t\neq 0]$, and the apparent intersection is
eliminated.
\end{myitemize}
\begin{myitemize}
\item $\rho^-_{ux}$ is the ``negative Hopf link''. It is constructed
out of its picture in exactly the same way as $\rho^+_{ux}$. We postpone to
Section~\ref{subsec:Hopf} the explanation of why $\rho^+_{ux}$ is
``positive'' and $\rho^-_{ux}$ is ``negative''.

\end{myitemize}
\end{example}

\begin{example} \label{exa:Sophisticated} Below on the right is a somewhat
more sophisticated example of an rKBH with two balloons labelled $a$ and
$b$ and two hoops labelled with the same labels (hence it is an element
of $\Kbh(a,b;a,b)$). It should be interpreted using the same conventions
as in the previous example, though some further comments are in order:

\parpic[r]{\raisebox{-45mm}{\input{figs/Sophisticated.pstex_t}}}
\noindent$\bullet$ 
The ``crossing'' marked (1) on the right is between two
hoops and in 4D it matters not if it is an overcrossing or an
undercrossing. Hence we did not bother to indicate which of the two it is.
A similar comment applies in two other places.

\noindent$\bullet$ Likewise, crossing (2) is between a 1D strand and a
thin tube, and its sense is immaterial. For no real reason we've drawn
the strand ``under'' the tube, but had we drawn it ``over'', it would
be the same rKBH.  A similar comment applies in two other places.

\noindent$\bullet$ Crossing (3) is ``real'' and is similar to $\rho^-$
in the previous example. Two other crossings in the picture are similar
to $\rho^+$.

\noindent$\bullet$ Crossing (4) was not seen before, though its 4D
meaning should be clear from our interpretation rules: nets are pushed
up (or down) along the $t$ coordinate by an amount proportional to the
distance from the boundary.  Hence the wider net in (4) gets pushed more
than the narrower one, and hence in 4D they do not intersect even though
their projections to 3D do intersect, as the figure indicates. A similar
comment applies in two other places.

\noindent$\bullet$ Our example can be simplified a bit using
isotopies. Most notably, crossing (5) can be eliminated by pulling the
narrow ``$\backslash$'' finger up and out of the wider ``$\slash$''
membrane. Yet note that a similar feat cannot be achieved near (3) and
(4). Over there the wider ``$\slash$'' finger cannot be pulled down and
away from the narrower ``$\backslash$'' membrane and strand without a
singularity along the way.
\end{example}

We can now complete Definition~\ref{def:KBH} by providing the the
definition of ``ribbon embedding''.

\parpic[r]{\input{figs/RibbonSing.pstex_t}}
\begin{definition} \label{def:ribbon} We say that an embedding
of a collection of 2-spheres $S_i$ into $\bbR^4$ (or into $S^4$)
is ``ribbon'' if it can be extended to an immersion $\iota$ of a
collection of 3-balls $B_i$ whose boundaries are the $S_i$'s, so that the
singular set $\Sigma\subset\bbR^4$ of $\iota$ consists of transverse
self-intersections, and so that each connected component $C$ of $\Sigma$
is a ``ribbon singularity'': $\iota^{-1}(C)$ consists of two closed disks
$D_1$ and $D_2$, with $D_1$ embedded in the interior of one of the $B_i$
and with $D_2$ embedded with its interior in the interior of some $B_j$
and with its boundary in $\partial B_j=S_j$. A dimensionally-reduced
illustration is on the right. The ribbon condition does not place any
restriction on the hoops of an rKBH.
\end{definition}

It is easy to verify that all the examples above are ribbon, and that all the
operations we define below preserve the ribbon condition.

There is much literature about ribbon knots in $\bbR^4$. See
e.g.\ \cite{HabiroKanenobuShima:R2K, HabiroShima:R2KII,
KanenobuShima:TwoFiltrationsR2K, Satoh:RibbonTorusKnots,
Watanabe:ClasperMoves, WKO}.

\subsection{Usual tangles and the map $\delta$} \label{subsec:delta}
For the purposes of this paper, a ``usual tangle''\footnote{Better
English would be ``ordinary tangle'', but I want the short form to be
``u-tangle'', which fits better with the ``v-tangles'' and ``w-tangles''
that arise later in this paper.}, or a ``u-tangle'', is a ``framed pure
labelled tangle in a disk''. In detail, it is a piece of an
oriented knot diagram drawn in a disk, having no closed components and
with its components labelled by the elements of some set $S$, with all
regarded modulo the Reidemeister moves R1', R2, and R3:
\[ \input{figs/RMoves.pstex_t} \]

\parpic[r]{\begin{picture}(0,0)%
\includegraphics{figs/uExample.pstex}%
\end{picture}%
%
%  pstex_opts: -m 1.2 
%
\setlength{\unitlength}{4736sp}%
\begingroup\makeatletter\ifx\SetFigFont\undefined%
\gdef\SetFigFont#1#2#3#4#5{%
  \reset@font\fontsize{#1}{#2pt}%
  \fontfamily{#3}\fontseries{#4}\fontshape{#5}%
  \selectfont}%
\fi\endgroup%
\begin{picture}(759,697)(-175,80)
\put(376,539){\makebox(0,0)[b]{\smash{{\SetFigFont{11}{13.2}{\rmdefault}{\mddefault}{\itdefault}{\color[rgb]{0,0,0}a}%
}}}}
\put(  4,223){\makebox(0,0)[b]{\smash{{\SetFigFont{11}{13.2}{\rmdefault}{\mddefault}{\itdefault}{\color[rgb]{0,0,0}b}%
}}}}
\end{picture}%
}
The set of all tangles with components labelled by $S$ is denoted
$\uT(S)$. An example of a member of $\uT(a,b)$ is on the right. Note that
our u-tangles do not have a specific ``up'' direction so they do not form
a category, and that the condition ``no closed components'' prevents them
from being a planar algebra. In fact, $\uT$ carries almost no interesting
algebraic structure. Yet it contains knots (as 1-component tangles)
and more generally, by restricting to a subset, it contains ``pure tangles''
or ``string links''~\cite{HabeggerLin:Classification}. And in the next
section $\uT$ will be generalized to $\vT$ and to $\wT$, which do carry much
interesting structure.

There is a map $\delta\colon\uT(S)\to\Kbh(S;S)$. The picture should
precede the words, and it appears as the left half of
Figure~\ref{fig:deltaExample}.

\begin{figure}
\[ \input{figs/deltaExample.pstex_t}
  \qquad \includegraphics[width=3.5in]{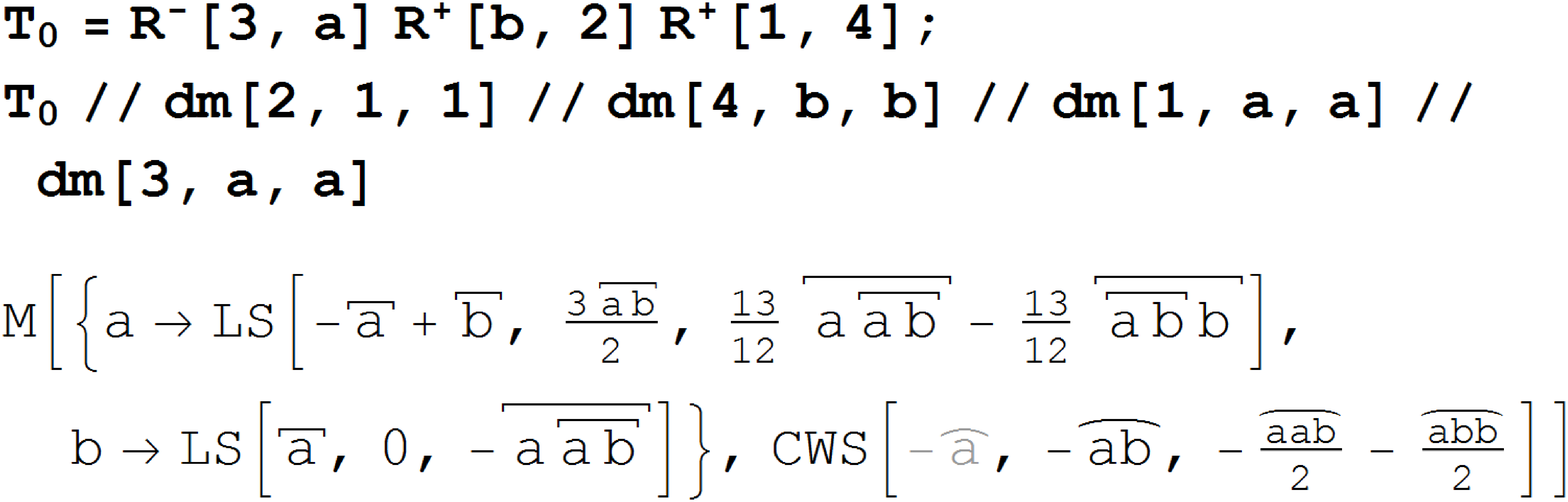}
\]
\caption{A $T_0\mapsto\delta(T_0)$ example, and its invariant $\zeta$
  of Section~\ref{sec:zeta} (computed to degree 3).
} \label{fig:deltaExample}
\end{figure}

In words, if $T\in\uT(S)$, to make $\delta(T)$ we convert each strand $s\in
S$ of $T$ into a pair of parallel entities: a copy of $s$ on the right and
a band on the left ($T$ is a planar diagram and $s$ is oriented, so ``left''
and ``right'' make sense). We cap the resulting band near its beginning and
near its end, connecting the cap at its end to $\infty$ (namely, to outside
the picture) with an extra piece of string --- so that when the bands are
pushed to 4D in the usual way, they become balloons with strings.  
Finally, near the crossings of $T$ we apply the following (sign-preserving)
local rules:
\[ \input{figs/deltaRules.pstex_t}. \]

\begin{proposition} The map $\delta$ is well defined.
\end{proposition}

\begin{proof} We need to check that the Reidemeister moves in $\uT$
are carried to isotopies in $\Kbh$. We'll only display the ``band part''
of the third Reidemeister move, as everything else is similar or easier:
\[ \begin{picture}(0,0)%
\includegraphics{figs/deltaR3.pstex}%
\end{picture}%
%
%  pstex_opts: -m 1 
%
\setlength{\unitlength}{3947sp}%
\begingroup\makeatletter\ifx\SetFigFont\undefined%
\gdef\SetFigFont#1#2#3#4#5{%
  \reset@font\fontsize{#1}{#2pt}%
  \fontfamily{#3}\fontseries{#4}\fontshape{#5}%
  \selectfont}%
\fi\endgroup%
\begin{picture}(5499,1449)(-611,-1348)
\put(1351,-511){\makebox(0,0)[b]{\smash{{\SetFigFont{11}{13.2}{\rmdefault}{\mddefault}{\updefault}{\color[rgb]{0,0,0}$\delta$}%
}}}}
\end{picture}%
 \]
The fact that the two ``band diagrams'' above are isotopic before
``inflation'' to $\bbR^4$, and hence also after, is visually obvious. \qed
\end{proof}

\subsection{The Fundamental Invariant and the Near-Injectivity of $\delta$}
\label{subsec:pi}
The ``Fundamental Invariant'' $\pi(K)$ of $K\in\Kbh(u_i;x_j)$
is the triple $(\pi_1(K^c);\, m;\, l)$, where within this triple:
\begin{itemize}
\item The first entry is the fundamental group of the complement of
the balloons of $K$, with basepoint taken to be at $\infty$. 
\item The second entry $m$ is the function $m\colon T\to\pi_1(K^c)$
which assigns to a balloon $u\in T$ its ``base meridian'' $m_u$
--- the path obtained by travelling along the string of $u$ from
$\infty$ to near the balloon, then Hopf-linking with the balloon $u$
once in the positive direction much like in the generator $\rho^+$ of
Figure~\ref{fig:Generators}, and then travelling back to the basepoint
again along the string of $u$.
\item The third entry $l$ is the function $l\colon H\to\pi_1(K^c)$ which
assigns to hoop $x\in H$ its longitude $l_x$ --- it is simply the hoop $x$
itself regarded as an element of $\pi_1(K^c)$.
\end{itemize}

Thus for example, with $\langle \alpha\rangle$ denoting the group
generated by a single element $\alpha$ and following the ``notational
conventions'' of Section~\ref{subsec:Conventions} for ``inline functions'',
\[ \pi(h\epsilon_x)=(1;\,();\,(x\to 1)),
  \qquad\pi(t\epsilon_u)=(\langle \alpha\rangle;\,(u\to \alpha);\,()),
\]
\[
  \text{and}\qquad
  \pi(\rho^\pm_{ux})
  = (\langle \alpha\rangle;\,(u\to \alpha);\,(x\to \alpha^{\pm 1})).
\]

We leave the following proposition as an exercise for the reader:

\begin{proposition} If $T$ is an $\underline{n}$-labelled u-tangle,
then $\pi(\delta(T))$ is the fundamental group of the complement of $T$
(within a 3-dimensional space!), followed by the list of meridians of $T$
(placed near the outgoing ends of the components of $T$), followed by
the list of longitudes of $T$. \qed
\end{proposition}

It is well known (e.g.~\cite[Theorem~6.1.7]{Kawauchi:Survey}) that knots
are determined by the fundamental group of their complements, along with
their ``peripheral systems'', namely their meridians and longitudes regarded
as elements of the fundamental groups of their complements. Thus we have:

\begin{theorem}
When restricted to long knots (which are the same as knots), $\delta$
is injective.~\qed
\end{theorem}

\begin{remark} A similar map studied by Winter~\cite{Winter:SpunTorusKnots}
is (sometimes) 2 to 1, as it retains less orientation information.
\end{remark}

I expect that $\delta$ is also injective on arbitrary tangles
and that experts in geometric topology would consider this trivial,
but this result would be outside of my tiny puddle.

\subsection{The Extension to v/w-Tangles and the Near-Surjectivity of $\delta$}
\label{subsec:vw}
The map $\delta$ can be extended to ``virtual
crossings''~\cite{Kauffman:VirtualKnotTheory} using the local assignment
\begin{equation} \label{eq:deltaVirtual}
  \raisebox{-6mm}{\begin{picture}(0,0)%
\includegraphics{figs/deltaVirtual.pstex}%
\end{picture}%
%
%  pstex_opts: -m 1.0 
%
\setlength{\unitlength}{3947sp}%
\begingroup\makeatletter\ifx\SetFigFont\undefined%
\gdef\SetFigFont#1#2#3#4#5{%
  \reset@font\fontsize{#1}{#2pt}%
  \fontfamily{#3}\fontseries{#4}\fontshape{#5}%
  \selectfont}%
\fi\endgroup%
\begin{picture}(2424,474)(-86,377)
\put(1651,614){\makebox(0,0)[b]{\smash{{\SetFigFont{11}{13.2}{\rmdefault}{\mddefault}{\updefault}{\color[rgb]{0,0,0}$=$}%
}}}}
\put(676,689){\makebox(0,0)[b]{\smash{{\SetFigFont{11}{13.2}{\rmdefault}{\mddefault}{\updefault}{\color[rgb]{0,0,0}$\delta$}%
}}}}
\end{picture}%
.}
\end{equation}
In a few more words, u-tangles can be extended to ``v-tangles''
by allowing ``virtual crossings'' as on the left hand side
of~\eqref{eq:deltaVirtual}, and then modding out by the ``virtual
Reidemeister moves'' and the ``mixed move'' / ``detour move''
of~\cite{Kauffman:VirtualKnotTheory}%
\footnote{
  In~\cite{Kauffman:VirtualKnotTheory} the mixed / detour move was yet
  unnamed, and was simply ``move (c) of Figure 2''.
}.
One may then observe, as in Figure~\ref{fig:OCUC}, that $\delta$ respects
those moves as well as the ``overcrossings commute'' relation (yet not
the ``undercrossings commute'' relation).  Hence $\delta$ descends to
the space $\wT$ of w-tangles, which are the quotient of v-tangles by
the overcrossings commute relation.

\begin{figure}
\[ \input{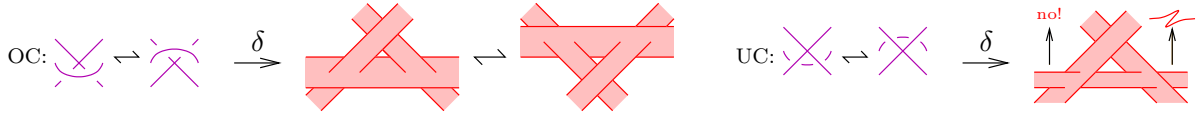} \]
\caption{
  The ``Overcrossing Commute'' (OC) relation and the gist of the proof
  that it is respected by $\delta$, and the ``Undercrossing Commute'' (UC)
  relation and the gist of the reason why it is not respected by $\delta$.
} \label{fig:OCUC}
\end{figure}

A topological-flavoured construction of $\delta$ appears in
Section~\ref{subsec:topdelta}.

The newly extended $\delta\colon\wT\to\Kbh$ cannot possibly be surjective,
for the rKBHs in its image always have an equal number of balloons as hoops,
with the same labels. Yet if we allow the deletion of components, $\delta$
becomes surjective:

\begin{theorem} \label{thm:surjective}
For any KTG $K$ there is some w-tangle $T$ so that $K$ is obtained from
$\delta(T)$ by the deletion of some of its components.
\end{theorem}

\begin{proof} (Sketch) This is a variant of Theorem~3.1 of
Satoh's~\cite{Satoh:RibbonTorusKnots}. Clearly every knotting of 2-spheres
in $\bbR^4$ can be obtained from a knotting of tubes by capping those
tubes. Satoh shows that any knotting of tubes is in the image of a map he
calls ``Tube'', which is identical to our $\delta$ except our $\delta$ also
includes the capping (good) and an extra hoop component for each balloon
(harmless as they can be deleted). Finally to get the hoops of $K$ simply
put them in as extra strands in $T$, and then delete the spurious balloons
that $\delta$ would produce next to each hoop. \qed
\end{proof}

\section{The Operations} \label{sec:Operations}

\subsection{The Meta-Monoid-Action} \label{subsec:MMAOperations}
Loosely speaking, an rKBH $K$ is a map of several $S^1$'s and several
$S^2$'s into some ambient space. The former (the hoops of $K$) resemble
elements of $\pi_1$, and the latter (the balloons of $K$) resemble
elements of $\pi_2$. In general in homotopy theory, $\pi_1$ and $\pi_2$
are groups, and further, there is an action of $\pi_1$ on $\pi_2$. Thus
we find that on $\Kbh$ there are operations that resemble the group
multiplication of $\pi_1$, and the group multiplication of $\pi_2$,
and the action of $\pi_1$ on $\pi_2$.

\begin{figure}
\[
  \def\KBox{{\parbox{1.5in}{\begin{center}
    $(T;H)=(u,v;x,y)$,
    \newline$\mu=(\lambda;\omega)$ in
    \newline$\FL(T)^H\times\CWr(T)$,
  \end{center}}}}
  \def\hmBox{{\parbox{1.5in}{
    $\lambda\mapsto\lambda\remove\{x,y\}\cup$\hfill\null
    \newline\null\hfill$(z\to\bch(\lambda_x,\lambda_y))$,
    \newline\null\hfill$\omega$ untouched,\hfill\null
  }}}
  \def\tmBox{{\parbox{1.5in}{\begin{center}
    $\mu\mapsto\mu\sslash(u,v\to w)$,
  \end{center}}}}
  \def\thaBox{{\parbox{1.5in}{\begin{center}
    $\mu\mapsto(\lambda;\,\omega+J_u(\lambda_x))\!\sslash\!RC_u^{\lambda_x}$,
  \end{center}}}}
  \def\defsBox{{
    with $C_u^\gamma=(u\mapsto e^{\ad\gamma}(u))$,
    $RC_u^\gamma=(C_u^{-\gamma})^{-1}$, and $J_u(\gamma) =
    \int_0^1ds\,\diver_u\!\left(\gamma \sslash RC_u^{s\gamma}
    \right) \sslash C_u^{-s\gamma}$.
  }}
  \input{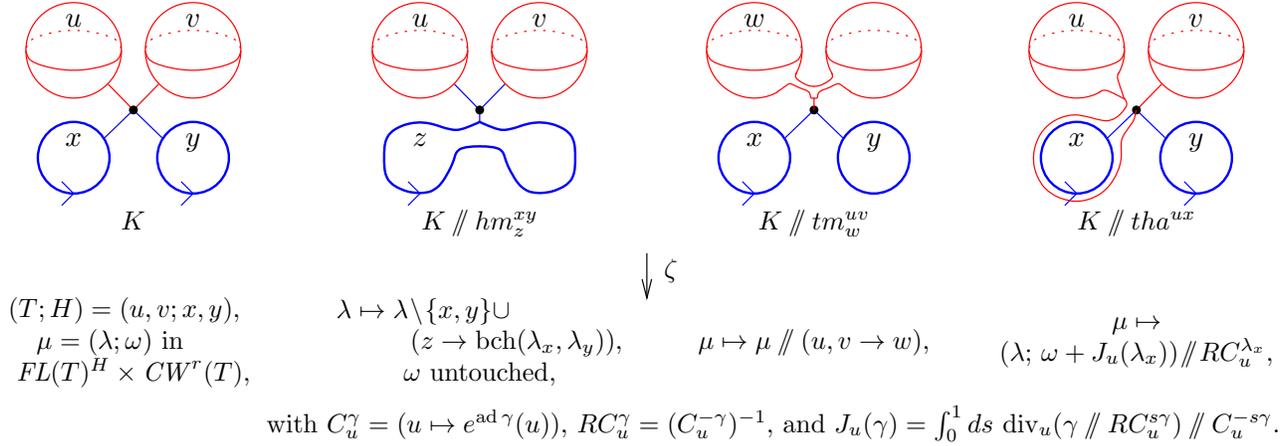}
\]
\caption{
  An rKBH $K$ and the three basic unary operators applied to it. We use
  schematic notation; $K$ may have plenty more components, and it may
  actually be knotted. The lower part of the figure is a summary of the
  main invariant $\zeta$ defined in this paper. See Section~\ref{sec:zeta}.
} \label{fig:mgaops}
\end{figure}

Let us describe these operations more carefully. Let $K\in\Kbh(T;H)$. 

\begin{myitemize}

\item Analogously to the product in $\pi_1$, there is the operation of
``concatenating two hoops''. Specifically, if $x$ and $y$ are two distinct
labels in $H$ and $z$ is a label not in $H$ (except possibly equal to $x$
or to $y$), we let%
\footnote{
  See ``notational conventions'', Section~\ref{subsec:Conventions}.
} $K\sslash hm^{xy}_z$ be $K$ with the $x$ and $y$ hoops removed and
replaced with a single hoop labelled $z$ that traces the path of them
both. See Figure~\ref{fig:mgaops}.

\parpic[r]{\includegraphics[width=0.8in]{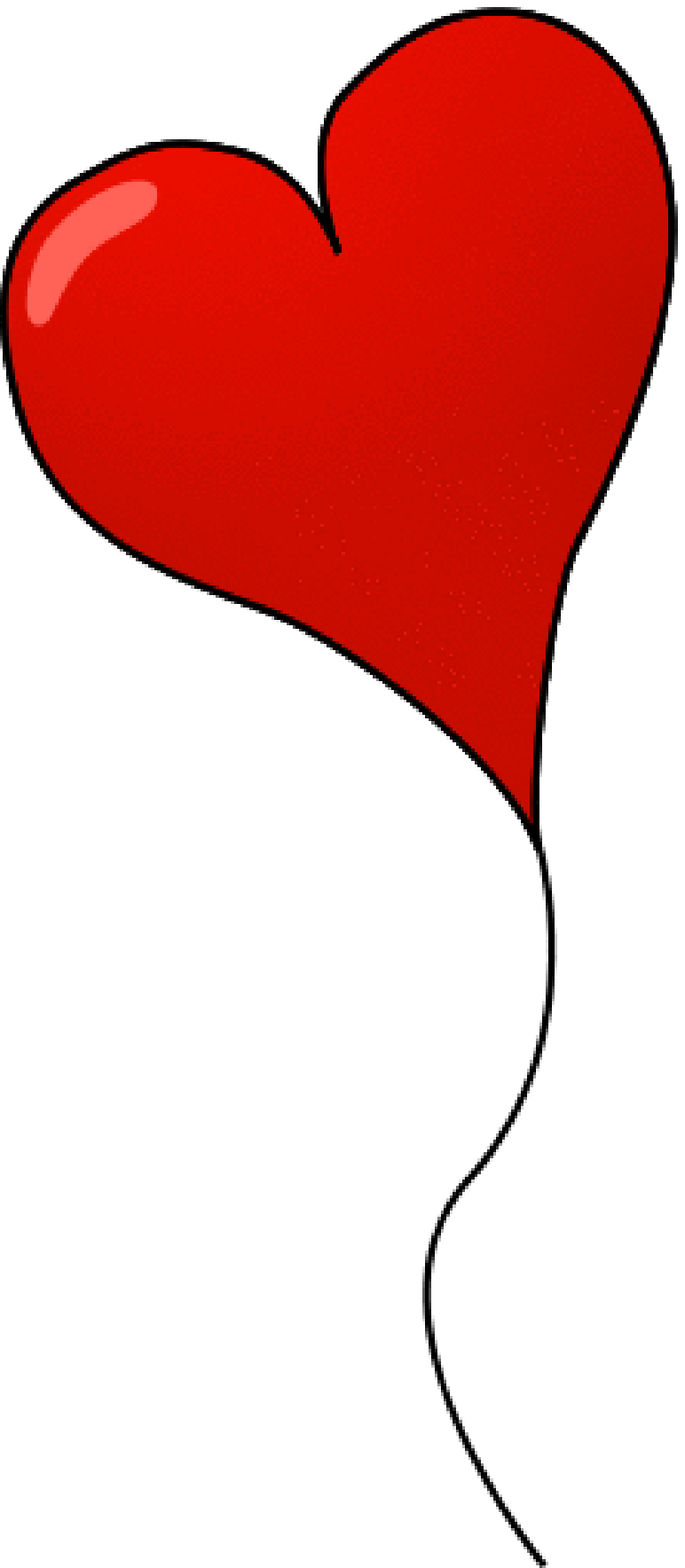}}
\item Analogously to the homotopy-theoretic product of $\pi_2$, there
is the operation of ``merging two balloons''. Specifically, if $u$ and
$v$ are two distinct labels in $T$ and $w$ is a label not in $T$ (except
possibly equal to $u$ or to $v$), we let $K\sslash tm^{uv}_w$ be $K$
with the $u$ and $v$ balloons removed and replaced by a single two-lobed
balloon (topologically, still a sphere!) labelled $w$ which spans them
both. See Figure~\ref{fig:mgaops}, or the even nicer two-lobed balloon
displayed on the right.

\item Analogously to the homotopy-theoretic action of $\pi_1$ on $\pi_2$,
there is the operation $tha^{ux}$ (``tail by head action on $u$ by $x$'')
of re-routing the string of the balloon $u$ to go along the hoop $x$,
as illustrated in Figure~\ref{fig:mgaops}. In balloon-theoretic language,
after the isotopy which pulls the neck of $u$ along its string, this is
the operation of ``tying the balloon'', commonly performed to prevent
the leakage of air (though admittedly, this will fail in 4D).

\end{myitemize}

In addition, $\Kbh$ affords the further unary operations $t\eta^u$ (when
$u\in T$) of ``puncturing'' the balloon $u$ (implying, deleting it) and
$h\eta^x$ (when $x\in H$) of ``cutting'' the hoop $x$ (implying, deleting
it). These two operations were already used in the statement and proof of
Theorem~\ref{thm:surjective}.

\parpic[r]{\parbox{2.5in}{\begin{center}
  \def\astBox{{$
    (\lambda_1;\omega_1)\!\ast\!(\lambda_2;\omega_2)
    \!=\! (\lambda_1\cup\lambda_2;\omega_1+\omega_2)
  $}}
  \input{figs/ConnectedSum.pstex_t}
  \caption{Connected sums.} \label{fig:ConnectedSum}
\end{center}}}
In addition, $\Kbh$ affords the binary operation $\ast$  of ``connected
sum'', sketched on the right (along with its $\zeta$ formulae of
Section~\ref{sec:zeta}). Whenever we have disjoint label sets
$T_1\cap T_2=\emptyset=H_1\cap H_2$, it is an operation
$\Kbh(T_1;H_1)\times\Kbh(T_2;H_2)\to\Kbh(T_1\cup T_2;H_1\cup H_2)$. We
often suppress the $\ast$ symbol and write $K_1K_2$ for $K_1\ast K_2$.

Finally, there are re-labelling operations $h\sigma^a_b$ and $t\sigma^a_b$
on $\Kbh$, which take a label $a$ (either a head or a tail) and rename
it $b$ (provided $b$ is ``new'').

\begin{proposition} \label{prop:MMAProperties}
The operations $\ast$, $t\sigma^u_v$, $h\sigma^x_y$, $t\eta^u$, $h\eta^x$,
$hm^{xy}_z$, $tm^{uv}_w$ and $tha^{ux}$ and the special elements
$t\epsilon_u$ and $h\epsilon_x$ have the following properties:
\begin{itemize}
\item If the labels involved are distinct, the unary operations all commute
  with each other.
\item The re-labelling operations have some obvious properties and
  interactions: $\sigma^a_b\sslash\sigma^b_c=\sigma^a_c$,
  $hm^{xy}_x\sslash h\sigma^x_z=hm^{xy}_z$, etc., and similarly for the
  deletion operations $\eta^a$.
\item $\ast$ is commutative and associative; where it makes sense, it
  bi-commutes with the unary operations ($(K_1\sslash hm^{xy}_z)\ast
  K_2 = (K_1\ast K_2)\sslash hm^{xy}_z$, etc.).
\item $t\epsilon_u$ and $h\epsilon_x$ are ``units'':
\[ (K\ast t\epsilon_u)\sslash tm^{uv}_w=K\sslash t\sigma^v_w,
  \qquad (K\ast t\epsilon_u)\sslash tm^{vu}_w=K\sslash t\sigma^v_w,
\]
\[ (K\ast h\epsilon_x)\sslash hm^{xy}_z=K\sslash h\sigma^y_z,
  \qquad (K\ast h\epsilon_x)\sslash hm^{yx}_z=K\sslash h\sigma^y_z.
\]
\item Meta-associativity of $hm$, similar to the associativity in $\pi_1$:
\begin{equation} \label{eq:hassoc}
  hm^{xy}_x\sslash hm^{xz}_x=hm^{yz}_y\sslash hm^{xy}_x.
\end{equation}
\item Meta-associativity of $tm$, similar to the associativity in $\pi_2$:
\begin{equation} \label{eq:tassoc} 
  tm^{uv}_u\sslash tm^{uw}_u=tm^{vw}_v\sslash tm^{uv}_u.
\end{equation}
\item Meta-actions commute. The following is a special case of the first
  property above, yet it deserves special mention because later in this 
  paper it will be the only such commutativity that is non-obvious to
  verify:
\begin{equation} \label{eq:thatha}
  tha^{ux}\sslash tha^{vy} = tha^{vy}\sslash tha^{ux}.
\end{equation}
\item Meta-action axiom $t$, similar to $(uv)^x=u^xv^x$:
\begin{equation} \label{eq:taction}
  tm^{uv}_w\sslash tha^{wx}=tha^{ux}\sslash tha^{vx}\sslash tm^{uv}_w.
\end{equation}
\item Meta-action axiom $h$, similar to $u^{xy}=(u^x)^y$:
\begin{equation} \label{eq:haction}
  hm^{xy}_z\sslash tha^{uz}=tha^{ux}\sslash tha^{uy}\sslash hm^{xy}_z.
\end{equation}
\end{itemize}
\end{proposition}

\begin{proof} The first four properties say almost nothing and we did not
even specify them in full%
\footnote{
  We feel that the clarity of this paper is enhanced by this omission.
}.
The remaining four deserve attention, especially in the light of the
fact that the verification of their analogs later in this paper will
be non-trivial. Yet in the current context, their verification is
straightforward. \qed
\end{proof}

Later we will seek to construct invariants of rKBH's by specifying their
values on some generators and by specifying their behaviour under our list
of operations. Thus it is convenient to introduce a name for the algebraic
structure of which $\Kbh$ is an instance:

\begin{definition} \label{def:MMA}
A meta-monoid-action (MMA) $M$ is a collections of sets $M(T;H)$,
one for each pair of finite sets of labels $T$ and $H$, along with
partially-defined operations%
\footnote{
  $tm^{uv}_w$, for example, is defined on $M(T;H)$ exactly when
  $u,v\in T$ yet $w\not\in T\remove\{u,v\}$. All other operations behave similarly.
}
$\ast$, $t\sigma^u_v$, $h\sigma^x_y$, $t\eta^u$, $h\eta^x$, $hm^{xy}_z$,
$tm^{uv}_w$ and $tha^{ux}$, and with special elements $t\epsilon_u\in
M(\{u\};\emptyset)$ and $h\epsilon_x\in M(\emptyset;\{x\})$, which
together satisfy the properties in Proposition~\ref{prop:MMAProperties}.
\end{definition}

For the rationale behind the name ``meta-monoid-action'' see
Section~\ref{subsec:meta-naming}. In Section~\ref{ssec:MetaHopf}
we note that $\Kbh$ in fact has the further structure making it a
meta-group-action (or more precisely, a meta-Hopf-algebra-action).

\subsection{The Meta-Monoid of Tangles and the Homomorphism $\delta$.}
\label{subsec:WisMG}
Our aim in this section is to show that the map $\delta\colon\wT\to\Kbh$
of Sections~\ref{subsec:delta} and~\ref{subsec:vw}, which maps w-tangles
to knotted balloons and hoops, is a ``homomorphism''. But first we have
to discuss the relevant algebraic structures on $\wT$ and on $\Kbh$.

\parpic[r]{\input{figs/mabc.pstex_t}}
$wT$ is a ``meta-monoid'' (see Section~\ref{subsubsec:MetaMonoids}). Namely,
for any finite set $S$ of ``strand labels'' $\wT(S)$ is a set, and whenever
we have a set $S$ of labels and three labels $a\neq b$ and $c$ not in it,
we have the operation $m^{ab}_c\colon\wT(S\cup\{a,b\})\to\wT(S\cup\{c\})$
of ``concatenating strand $a$ with strand $b$ and calling the resulting
strand $c$''. See the picture on the right, and note that while on $\uT$ the
operation $m^{ab}_c$ would be defined only if the head of $a$ happens to be
adjacent to the tail of $b$, on $\vT$ and on $\wT$ this operation is always
defined, as the head of $a$ can always be brought near the tail of $b$ by
adding some virtual crossings, if necessary. $\wT$ trivially also carries the
rest of the necessary structure to form a meta-monoid --- namely, strand
relabelling operations $\sigma^a_b$, strand deletion operations $\eta^a$, and
a disjoint union operation $\ast$, and ``units'' $\epsilon_a$ (tangles with
a single unknotted strand labelled $a$).

It is easy to verify the associativity property (compare with
Equation~\eqref{eq:assoc} of Section~\ref{subsubsec:monoids}):
\[ m^{ab}_a\sslash m^{ac}_a = m^{bc}_b\sslash m^{ab}_a:
  \qquad
  \begin{array}{c}\input{figs/massoc.pstex_t}\end{array}.
\]
It is also easy to verify that if a tangle $T\in\wT(a,b)$ is non-split, then
$T\neq(T\sslash\eta^b)\ast(T\sslash\eta^a)$, so in the sense of
Section~\ref{subsubsec:MetaMonoids}, $\wT$ is non-classical.

\parpic[r]{\parbox{2in}{\footnotesize \ \newline
  {\bf Solution of Riddle~\ref{riddle:piT}.} \label{sol:piT}
  $\pi_T\cong\pi_1\ltimes\pi_2$ (a semi-direct product!), so if you know
  all about $\pi_1$ and $\pi_2$ (and the action of $\pi_1$ on $\pi_2$),
  you know all about $\pi_T$.
}}
$\Kbh$ is an analog of both $\pi_1$ and $\pi_2$. In homotopy theory,
the group $\pi_1$ acts on $\pi_2$ so one may form the semi-direct product
$\pi_1\ltimes\pi_2$. In a similar manner, one may put a ``combined''
multiplication on that part of $\Kbh$ in which the balloons and the hoops
are matched together. More precisely, given a finite set of labels $S$, let
$\Kbhd(S):=\Kbh(S;S)$ be the set of rKBHs whose balloons and whose hoops are
both labelled with labels in $S$. Then define
$dm^{ab}_c\colon\Kbhd(S\cup\{a,b\})\to\Kbhd(S\cup\{c\})$ (the prefix $d$ is
for ``diagonal'', or ``double'') by
\begin{equation} \label{eq:dm}
  dm^{ab}_c = tha^{ab} \sslash tm^{ab}_c \sslash hm^{ab}_c.
\end{equation}
It is a routine exercise to verify that the properties
\eqref{eq:hassoc}--\eqref{eq:haction} of $hm$, $tm$, and $tha$ imply that
$dm$ is meta-associative:
\[ dm^{ab}_a\sslash dm^{ac}_a = dm^{bc}_b\sslash dm^{ab}_a. \]
Thus $dm$ (along with ``diagonal'' $\eta$'s and $\sigma$'s and an unmodified
$\ast$) puts a meta-monoid structure on $\Kbhd$.

\parpic[r]{\input{figs/deltahomo.pstex_t}}
\begin{proposition} $\delta\colon\wT\to\Kbhd$ is a meta-monoid homomorphism.
(A rough picture is on the right: in the picture $a$ and $b$
are strands within the {\em same} tangle, and they may be knotted with each
other and with possible further components of that tangle). \qed
\end{proposition}

\subsection{Generators and Relations for $\Kbh$}
\label{subsec:genrels}

It is always good to know that a certain algebraic structure is finitely
presented. If we had a complete set of generators and relations for
$\Kbh$, for example, we could define a ``homomorphic invariant'' of rKBHs
by picking some target MMA $\calM$ (Definition~\ref{def:MMA}), declaring
the values of the invariant on the generators, and verifying that the
relations are satisfied.  Hence it's good to know the following:

\begin{theorem} \label{thm:Generators} The MMA $\Kbh$ is generated (as an
MMA) by the four rKBHs $h\epsilon_x$, $t\epsilon_u$, and $\rho^{\pm}_{ux}$
of Figure~\ref{fig:Generators}.
\end{theorem}

\begin{proof} By Theorem~\ref{thm:surjective} and the fact that the MMA
operations include component deletions $t\eta^u$ and $h\eta^x$ it
follows that $\Kbh$ is generated by the image of $\delta$. By the previous
proposition and the fact~\eqref{eq:dm} that $dm$ can be written in terms
of the MMA operations of $\Kbh$, it follows that $\Kbh$ is generated by
the $\delta$-images of the generators of $\wT$. But the generators of
$\wT$ are the virtual crossing $\underset{a\ b}{\virtualcrossing}$ and the
right-handed and left-handed crossings $\underset{a\ b}{\overcrossing}$
and $\underset{a\ b}{\undercrossing}$, and so the theorem follows from the
following easily verified assertions:
$\delta\left(\underset{a\ b}{\virtualcrossing}\right) =
t\epsilon_a h\epsilon_a t\epsilon_b h\epsilon_b$,
$\delta\left(\underset{a\ b}{\overcrossing}\right) =
\rho^+_{ab}t\epsilon_b h\epsilon_a$, and
$\delta\left(\underset{a\ b}{\undercrossing}\right) =
\rho^-_{ba}t\epsilon_a h\epsilon_b$. \qed
\end{proof}

We now turn to the study of relations. Our first is the hardest and most
significant, the ``Conjugation Relation'', whose name is inspired by
the group theoretic relation $vu^v=uv$ (here $u^v$ denotes group
conjugation, $u^v=v^{-1}uv$). Consider the following equality:

\[ \input{figs/ConjugationRelation.pstex_t} \]

Easily, the rKBH on the very left is $\rho^+_{ux}(\rho^+_{vy}\rho^+_{wz}\sslash
tm^{vw}_v)\sslash hm^{xy}_x$ and the one on the very right is
$(\rho^+_{vx}\rho^+_{wz}\sslash tm^{vw}_v)\rho^+_{uy}\sslash hm^{xy}_x$,
and so
\begin{equation} \label{eq:ConjugationRelation}
  \rho^+_{ux}\rho^+_{vy}\rho^+_{wz}
    \sslash tm^{vw}_v\sslash hm^{xy}_x\sslash tha^{uz}
  = \rho^+_{vx}\rho^+_{wz}\rho^+_{uy}
    \sslash tm^{vw}_v\sslash hm^{xy}_x.
\end{equation}

\begin{definition} \label{def:Kbhz} Let $\Kbhz$ be the MMA freely
generated by symbols $\rho^\pm_{ux}\in\Kbhz(u;x)$, modulo
the following relations:
\begin{itemize}
\item Relabelling: $\rho^\pm_{ux}\sslash h\sigma^x_y\sslash
  t\sigma^u_v=\rho^\pm_{vy}$.
\item Cutting and puncturing: $\rho^\pm_{ux}\sslash h\eta^x=t\epsilon_u$
and $\rho^\pm_{ux}\sslash t\eta^u=h\epsilon_x$.
\item Inverses: $\rho^+_{ux}\rho^-_{vy}\sslash tm^{uv}_w\sslash
hm^{xy}_z=t\epsilon_w h\epsilon_z$.
\item Conjugation relations: for any $s_{1,2}\in\{\pm\}$,
\[
  \rho^{s_1}_{ux}\rho^{s_2}_{vy}\rho^{s_2}_{wz}
    \sslash tm^{vw}_v\sslash hm^{xy}_x\sslash tha^{uz}
  = \rho^{s_2}_{vx}\rho^{s_2}_{wz}\rho^{s_1}_{uy}
    \sslash tm^{vw}_v\sslash hm^{xy}_x.
\]
\item Tail-commutativity: on any inputs, $tm^{uv}_w=tm^{vu}_w$.
\item Framing independence:
\hfill$\displaystyle \rho^\pm_{ux}\sslash tha^{ux} = \rho^\pm_{ux}.$
\hfill\refstepcounter{equation}(\theequation)\label{eq:FI}
\end{itemize}
\end{definition}

The following proposition, whose proof we leave as an exercise, says that
$\Kbhz$ is a pretty good approximation to $\Kbh$:

\begin{proposition} \label{prop:pidelta} The obvious maps
$\pi\colon\Kbhz\!\to\!\Kbh$ and $\delta\colon\wT\!\to\!\Kbhz$ are well
defined. \qed
\end{proposition}

\begin{conjecture} \label{conj:Kbhz}
The projection $\pi\colon\Kbhz\to\Kbh$ is an isomorphism.
\end{conjecture}

We expect that there should be a Reidemeister-style combinatorial
calculus of ribbon knots in $\bbR^4$. The above conjecture is that the
definition of $\Kbhz$ is such a calculus. We expect that given any such
calculus, the proof of the conjecture should be easy. In particular,
the above conjecture is equivalent to the statement that the stated
relations in the definition of $\wT$ generate the relations in the
kernel of Satoh's Tube map $\delta_0$ (see Section~\ref{subsec:topdelta}),
and this is equivalent to the conjecture whose proof was attempted
at~\cite{Winter:RibbonEmbeddings}. Though I understood by private
communication with B.~Winter that~\cite{Winter:RibbonEmbeddings} is
presently flawed.

In the absence of a combinatorial description of $\Kbh$, we replace it by
$\Kbhz$ throughout the rest of this paper. Hence we construct invariants of
elements of $\Kbhz$ instead of invariants of genuine rKBHs. Yet note that
the map $\delta\colon\wT\to\Kbhz$ is well-defined, so our invariants are 
always good enough to yield invariants of tangles and virtual tangles.

\subsection{Example: The Fundamental Invariant} \label{sec:Pi}
The ``Fundamental Invariant'' $\pi$ of Section~\ref{subsec:pi} is
defined in a direct manner on $\Kbh$ and does not need to suffer from
the difficulties of the previous section.  Yet it can also serve as
an example for our approach for defining invariants on $\Kbhz$ using
generators and relations.

\begin{definition} \label{def:PiMMA}
Let $\Pi(T;H)$ denote the set of all triples
$(G;\,m;\,l)$ of a group $G$ along with functions $m\in G^T$ and $l\in
G^H$, regarded modulo group isomorphisms with their obvious action on $m$
and $l$%
\footnote{
  I ignore set-theoretic difficulties. If you insist, you may
  restrict to countable groups or to finitely presented groups.
}.
Define MMA operations $(\ast, t\sigma^u_v, h\sigma^x_y, t\eta^u, h\eta^x,
tm^{uv}_w, hm^{xy}_z, tha^{ux})$ on $\Pi=\{\Pi(T;H)\}$ and units
$t\epsilon_u$ and $h\epsilon_x$ as follows:
\begin{itemize}
\item $\ast$ is the operation of taking the free product $G_1\ast G_2$ of
groups and concatenating the lists of heads and tails:
\[
  (G_1;\,m_1;\,l_1) \ast (G_2;\,m_2;\,l_2)
  := (G_1\ast G_2;\,m_1\cup m_2;\,l_1\cup l_2).
\]
\item $t\sigma^a_b$ / $h\sigma^a_b$ relabels an element labelled $a$
  to be labelled $b$.
\item $t\eta^u$ / $h\eta^x$ removes the element labelled $u$ / $x$.
\item $tm^{uv}_w$ ``combines'' $u$ and $v$ to make $w$. Precisely, it
replaces the input group $G$ with $G'=G/\langle m_u=m_v\rangle$,
removes the tail labels
$u$ and $v$, and introduces a new tail, the element $m_u=m_v$ of $G'$
and labels it $w$:
\[ tm^{uv}_w(G;\,m;\,l) :=
  (G/\langle m_u=m_v\rangle;\, (m\remove\{u,v\})\cup(w\to m_u);\, l).
\]
\item $hm^{xy}_z$ replaces two elements in $l$ by their product:
\[
  hm^{xy}_z(G;\,m;\,l) := (G, m, (l\remove\{x,y\})\cup(z\to l_xl_y).
\]
\item The best way to understand the action of $tha^{ux}$ is as ``the thing
that makes the fundamental invariant $\pi$ a homomorphism, given the
geometric interpretation of $tha^{ux}$ on $\Kbh$ in
Section~\ref{subsec:MMAOperations}''. In formulae, this becomes
\[ tha^{ux}(G;\,m;\,l) :=
  (G\ast\langle\alpha\rangle/\langle m_u=l_x\alpha l_x^{-1}\rangle;\,
    (m\remove u)\cup(u\to\alpha), l),
\]
where $\alpha$ is some new element that is added to $G$.
\item $t\epsilon_u=(\langle \alpha\rangle;\,(u\to \alpha);\,())$ and
$h\epsilon_x=(1;\,();\,(x\to 1))$.
\end{itemize}
\end{definition}

We state the following without its easy topological proof:

\begin{proposition} $\pi\colon\Kbh\to\Pi$ is a homomorphism of MMAs. \qed
\end{proposition}

A consequence is that $\pi$ can be computed on any rKBH starting from its
values on the generators of $\Kbh$ as listed in Section~\ref{subsec:pi} and
then using the operations of Definition~\ref{def:PiMMA}.

\begin{comment} The fundamental groups of ribbon 2-knots
are ``labelled oriented tree'' (LOT) groups in the sense of
Howie~\cite{Howie:RibbonDiscComplements, Howie:HigherRibbonKnots}. Howie's
definition has an obvious extension to labelled oriented forests, yielding
a class of groups that may be called ``LOF groups''. One may show that the
the fundamental groups of complements of rKBHs are always LOF groups. One
may also show that the subset $\Pi^\text{LOF}$ of $\Pi$ in which the
group component $G$ is an LOF group is a sub-MMA of $\Pi$. Therefore
$\pi\colon\Kbh\to\Pi^\text{LOF}$ is also a homomorphism of MMAs; I expect
it to be an isomorphism or very close to an isomorphism. Thus much
of the rest of this paper can be read as a ``theory of homomorphic (in
the MMA sense) invariants of LOF groups''. I don't know how much it may
extend to a similar theory of homomorphic invariants of bigger classes
of groups.
\end{comment}

\section{The Free Lie Invariant} \label{sec:trees}

In this section we construct $\zeta_0$, the ``tree'' part to our main
tree-and-wheel valued invariant $\zeta$, by following the scheme of
Section~\ref{subsec:genrels}. Yet before we succeed, it is useful to
aim a bit higher and fail, and thus appreciate that even $\zeta_0$
is not entirely trivial.

\subsection{A Free Group Failure} If the balloon part of an rKBH $K$ is
unknotted, the fundamental group $\pi_1(K^c)$ of its complement is the
free group generated by the meridians $(m_u)_{u\in T}$. The hoops of $K$
are then elements in that group and hence they can be written as words
$(w_x)_{x\in H}$ in the $m_u$'s and their inverses. Perhaps we can make
and MMA $\calW$ out of lists $(w_x)$ of free words in letters $m_u^{\pm 1}$
and use it to define a homomorphic invariant $W\colon\Kbh\to\calW$? All we
need, it seems, is to trace how MMA operations on $K$ affect the
corresponding list $(w_x)$ of words.

The beginning is promising. $\ast$ acts on pairs of lists of words by
taking the union of those lists. $hm^{xy}_z$ acts on a list of words
by replacing $w_x$ and $w_y$ by their concatenation, now labelled $z$.
$tm^{pq}_r$ acts on $\bar{w}=(w_x)$ by replacing every occurrence of
the letter $m_p$ and every occurrence of the letter $m_q$ in $\bar{w}$
by a single new letter, $m_r$.

The problem is with $tha^{ux}$. Imitating the topology, $tha^{ux}$ should
act on $\bar{w}=(w_y)$ by replacing every occurrence of $m_u$ in $\bar{w}$
with $w_x\alpha w_x^{-1}$, where $\alpha$ is a new letter, destined to
replace $m_u$. But $w_x$ may also contain instances of $m_u$, so after the
replacement $m_u\mapsto \alpha^{w_x}$ is performed, it should be performed
again to get rid of the $m_u$'s that appear in the ``conjugator'' $w_x$.
But new $m_u$'s are then created, and the replacement should be carried out
yet again\ldots The process clearly doesn't stop, and our attempt failed.

Yet not all is lost. The later and later replacements occur within
conjugators of conjugators, deeper and deeper into the lower central series
of the free groups involved. Thus if we replace free groups by some
completion thereof in which deep members of the lower central series are
``small'', the process becomes convergent. This is essentially what will be
done in the next section.

\subsection{A Free Lie Algebra Success} \label{subsec:FLSuccess}
Given a set $T$, let $\FL(T)$ denote the graded completion of the free
Lie algebra on the generators in $T$ (sometimes we will write ``$\FL$''
for ``$\FL(T)$ for some set $T$''). We define a meta-monoid-action $M_0$
as follows. For any finite set $T$ of ``tail labels'' and any finite
set $H$ or ``head labels'', we let
\[ M_0(T;H):=\FL(T)^H \]
be the set of $H$-labelled arrays of elements of $\FL(T)$. On
$M_0:=\{M_0(T;H)\}$ we define operations as follows, starting from the
trivial and culminating with the most interesting, $tha^{ax}$. All of our
definitions are directly motivated by the ``failure'' of the previous
section; in establishing the correspondence between the definitions below
and the ones above, one should interpret $\lambda=(\lambda_x)\in M_0(T;H)$
as ``a list of logarithms of a list of words $(w_x)$''.

\begin{itemize}

\item $h\sigma^x_y$ is simply $\sigma^x_y$ as explained in the conventions
section, Section~\ref{subsec:Conventions}.

\item $t\sigma^u_v$ is induced by the map $\FL(T)\to \FL((T\remove u)\cup\{v\})$
in which the generator $u$ is mapped to the generator $v$.

\item $t\eta$ acts by setting one of the tail variables to $0$, and $h\eta$
acts by dropping an array element. Thus for $\lambda\in M_0(T;H)$,
\[ \lambda\sslash t\eta^u=\lambda\sslash(u\mapsto 0)
  \qquad\text{and}\qquad
  \lambda\sslash h\eta^x = \eta\remove x.
\]

\item If $\lambda_1\in M_0(T_1;H_1)$ and $\lambda_2\in M_0(T_2;H_2)$ (and,
of course, $T_1\cap T_2=\emptyset=H_1\cap H_2$), then
\[ \lambda_1\ast\lambda_2 :=
  (\lambda_1\sslash\iota_1)\cup(\lambda_2\sslash\iota_2)
\]
where $\iota_i$ are the natural embeddings $\iota_i\colon
\FL(T_i)\hookrightarrow \FL(T_1\cup T_2)$, for $i=1,2$.

\item If $\lambda\in M_0(T;H)$ then 
\[ \lambda\sslash tm^{uv}_w := \lambda \sslash (u,v\mapsto w), \]
where $(u,v\mapsto w)$ denotes the morphism $\FL(T)\to
\FL(T\remove\{u,v\}\cup\{w\})$ defined by mapping the generators $u$ and $v$
to the generator $w$.

\item If $\lambda\in M_0(T;H)$ then 
\[ \lambda\sslash hm^{xy}_z :=
  \lambda\remove\{x,y\}\cup(z\to\bch(\lambda_x,\lambda_y)),
\]
where $\bch$ stands for the Baker-Campbell-Hausdorff formula:
\[ \bch(a,b) := \log(e^ae^b) = a+b+\frac12[a,b]+\dots. \]

\item If $\lambda\in M_0(T;H)$ then 
\begin{equation} \label{eq:CRC}
  \lambda\sslash tha^{ux} :=
  \lambda\sslash (C_u^{-\lambda_x})^{-1} = \lambda\sslash RC_u^{\lambda_x}
\end{equation}
In the above formula $C_u^{-\lambda_x}$ denotes the automorphism of
$\FL(T)$ defined by mapping the generator $u$ to its ``conjugate''
$e^{-\lambda_x}ue^{\lambda_x}$. More precisely, $u$ is mapped to
$e^{-\ad\lambda_x}(u)$, where $\ad$ denotes the adjoint action, and
$e^{\ad}$ is taken in the formal
sense. Thus
\begin{equation} \label{eq:RCDef}
  C_u^{-\lambda_x}\colon u \mapsto 
  e^{-\ad\lambda_x}(u) = u - [\lambda_x,u]
    + \frac12[\lambda_x,[\lambda_x,u]]-\dots.
\end{equation}
Also in Equation~\eqref{eq:CRC},
$RC_u^{\lambda_x}:=(C_u^{-\lambda_x})^{-1}$ denotes the inverse of the
automorphism $C_u^{-\lambda_x}$.
\item $t\epsilon_u=()$ and $h\epsilon_x=(x\to 0)$.
\end{itemize}

\begin{warning} When $\gamma\in\FL$, the inverse of $C_u^{-\gamma}$
may {\em not} be $C_u^{\gamma}$. If $\gamma$ does not contain the
generator $u$, then indeed $C_u^{-\gamma}\sslash C_u^{\gamma}=I$. But
in general applying $C_u^{-\gamma}$ creates many ``new'' $u$'s, within
the $\gamma$'s that appear in the right hand side of~\eqref{eq:RCDef},
and the ``new'' $u$'s are then conjugated by $C_u^{\gamma}$ instead of
being left in place. Yet $C_u^{-\gamma}$ is invertible, so we simply
name its inverse $RC_u^{\gamma}$.
\end{warning}

The name ``$RC$'' stands either for ``Reverse Conjugation'', or for
``Repeated Conjugation''. The rationale for the latter naming is that if
$\alpha\in\FL(T)$ and $\baru$ is a name for a new ``temporary'' free-Lie
generator, then $RC_u^{\gamma}(\alpha)$ is the result of applying the
transformation $u \mapsto e^{\ad\gamma}(\baru)$ repeatedly to $\alpha$
until it stabilizes (at any fixed degree this will happen after a finite
number of iterations), followed by the eventual renaming $\baru\mapsto u$.

\parpic[r]{\raisebox{-11mm}{$\xymatrix@C=-0.52in{
  \FL(T) \ar@<-2pt>[rr]_{RC^{\gamma}_u} \ar[dr]^(0.4)\phi &
    & \FL(T) \ar@<-2pt>[ll]_{C^{-\gamma}_u}
      \ar[dl]_(0.4){\bar\phi}^(0.25){u\mapsto\baru} \\
  & \FL(T\cup\{\baru\})\left/\left({\begin{array}{c}
      \baru=e^{-\ad\gamma}u \\
      \text{\scriptsize and / or} \\
      u=e^{\ad\gamma}\baru
    \end{array}}\right)\right.
}$}}
\begin{comment} \label{com:CRC} Some further insight into
$RC_u^{\gamma}$ can be obtained by studying the triangle on the
right. The space at the bottom of the triangle is the quotient of
the free Lie algebra on $T\cup\{\baru\}$ (where $\baru$ is a new
``temporary'' generator) by either of the two relations shown there;
these two relations are of course equivalent. The map $\phi$ is induced
from the obvious inclusion of $\FL(T)$ into $\FL(T\cup\{\baru\})$,
and in the presence of the relation $\baru=e^{-\ad\gamma}u$, it is
clearly an isomorphism. The map $\bar\phi$ is likewise induced from the
renaming $u\mapsto\baru$. It too is an isomorphism, but slightly less
trivially --- indeed, using the relation $u=e^{\ad\gamma}\baru$ {\em
repeatedly}, any element in $\FL(T\cup\{\baru\})$ can be written in form
that does not include $u$, and hence is in the image of $\bar\phi$. It
is clear that $C^{-\gamma}_u=\bar\phi\sslash\phi^{-1}$. Hence
$RC^{\gamma}_u=\phi\sslash\bar\phi^{-1}$, and as $\bar\phi^{-1}$
is described in terms of repeated applications of the relation
$u=e^{\ad\gamma}\baru$, it is clear that $RC^{\gamma}_u$ indeed
involves ``repeated conjugation'' as asserted in the previous paragraph.
\end{comment}

\begin{warning} Equation~\eqref{eq:CRC} does {\em not} say that
$tha^{ux}=RC^{\lambda_x}_u$ as abstract operations, only that they are
equal when evaluated on $\lambda$.  In general it is not the case that
$\mu\sslash tha^{ux}=\mu\sslash RC^{\lambda_x}_u$ for arbitrary $\mu$
--- the latter equality is only guaranteed if $\mu_x=\lambda_x$.

As another example of the difference, the operations $hm^{xy}_z$ and
$tha^{ux}$ do not commute --- in fact, the composition $hm^{xy}_z \sslash
tha^{ux}$ does not even make sense, for by the time $tha^{ux}$ is evaluated
its
input does not have an entry labelled $x$. Yet the commutativity
\begin{equation} \label{eq:hmcc}
  \lambda \sslash hm^{xy}_z \sslash RC^{\lambda_x}_u
  = \lambda \sslash RC^{\lambda_x}_u\sslash hm^{xy}_z
\end{equation}
makes perfect sense and holds true, for the operation $hm^{xy}_z$ only
involves the heads / roots of trees, while $RC^{\lambda_x}_u$ only involves
their tails / leafs.
\end{warning}

\begin{theorem} \label{thm:MMAZero}
$M_0$, with the operations defined above, is a meta-monoid-action (MMA).
\end{theorem}

\begin{proof} Most MMA axioms are trivial to verify. The most
important ones are the ones in Equations~\eqref{eq:hassoc}
through~\eqref{eq:haction}. Of these, the meta-associativity
of $hm$ follows from the associativity of the $\bch$
formula, $\bch(\bch(\lambda_x,\lambda_y),\lambda_z) =
\bch(\lambda_x,\bch(\lambda_y,\lambda_z))$, the meta-associativity of $tm$
and is trivial, and it remains to prove that meta-actions commute
(Equation~\eqref{eq:thatha}; all other required commutativities are easy)
and the the meta-action axiom $t$ (Equation~\eqref{eq:taction}) and $h$
(Equation~\eqref{eq:haction}).

\noindent{\bf Meta-actions commute.} Expanding~\eqref{eq:thatha} using the
above definitions and denoting $\alpha:=\lambda_x$, $\beta=\lambda_y$,
$\alpha':=\alpha\sslash RC_v^\beta$, and $\beta':=\beta\sslash
RC_u^\alpha$, we see that we need to prove the identity
\begin{equation} \label{eq:RCuRCv}
  RC_u^\alpha\sslash RC_v^{\beta'} = RC_v^\beta\sslash RC_u^{\alpha'}.
\end{equation}

\parpic[r]{$\displaystyle
\xymatrix@R=0.5in@C=-0.3in{
  \FL(u,v)
    \ar[rr]^{RC_u^{\alpha}}_(0.2){\beta}_(0.8){\beta'}
    \ar[rd]^(0.4){(u,v)\to(u,v)}
    \ar[dd]_{RC_v^\beta}^(0.2){\alpha}^(0.8){\alpha'} &
  &
  \FL(u,v)
    \ar[ld]_(0.6){(u,v)\to(\baru,v)}
    \ar[dd]^{RC_v^{\beta'}} \\
  &
  \FL(u,\baru,v,\barv)\left/\left(\!\!{\begin{array}{c}
      u=e^{\ad \alpha}\baru \\
      v=e^{\ad \beta}\barv
    \end{array}}\!\!\right)\right. &
  \\
  \FL(u,v)
    \ar[rr]_{RC_u^{\alpha'}}
    \ar[ru]_(0.4){(u,v)\to(u,\barv)} &
  &
  \FL(u,v)
    \ar[lu]^(0.6){(u,v)\to(\baru,\barv)}
}
$}
Consider the commutative diagram on the right. In it $\FL(u,v)$ means
``the (completed) free Lie algebra with generators $u$ and $v$, and
some additional fixed collection of generators'', and likewise for
$\FL(u,\baru,v,\barv)$. The diagonal arrows are all substitution
homomorphisms as indicated, and they are all isomorphisms. We put
the elements $\alpha$ and $\beta$ in the upper-left space, and by
comparing with the diagram in Comment~\ref{com:CRC}, we see that the
upper horizontal map is $RC_u^\alpha$ and the left vertical map is
$RC_v^\beta$. Therefore $\beta'$ is the image of $\beta$ in the top left
space, and $\alpha'$ is the image of $\alpha$ in the bottom left space.
Therefore again using the diagram in Comment~\ref{com:CRC}, the right
vertical map is $RC_v^{\beta'}$ and the lower horizontal map is
$RC_u^{\alpha'}$, and~\eqref{eq:RCuRCv} follows from the commutativity of
the external square in the above diagram.

For use later, we record the fact that by reading all the horizontal and
vertical arrows backwards,
the above argument also proves the identity
\begin{equation} \label{eq:CuCv}
  C_u^{-\alpha\sslash RC_v^\beta}\sslash C_v^{-\beta}
  = C_v^{-\beta\sslash RC_u^\alpha}\sslash C_u^{-\alpha}.
\end{equation}

\noindent{\bf Meta-action axiom $t$.} Expanding~\eqref{eq:taction} 
and denoting $\gamma:=\lambda_x$, we need to prove the identity
\begin{equation} \label{eq:nts-t}
  t^{uv}_w\sslash RC_w^{\gamma\sslash t^{uv}_w}
  = RC_u^{\gamma} \sslash RC_v^{\gamma\sslash RC_u^{\gamma}}
    \sslash t^{uv}_w.
\end{equation}

\parpic[r]{$\displaystyle
\def\map#1#2{{\renewcommand{\arraystretch}{0.6}
  \begin{array}{c}
    \scriptstyle #1 \\
    \scriptstyle \downarrow \\
    \scriptstyle #2
  \end{array}
}}
\xymatrix@R=0.5in@C=-0.3in{
  \gamma\in\FL(u,v)
    \ar[r]^{RC_u^{\gamma}}
    \ar[rd]_{\phi_1}^(0.6){\hspace{-5mm}\map{u,v}{u,v}}
    \ar[ddd]_{\map{u,v}{w}} &
  \gamma_2\in\FL(u,v)
    \ar[r]^(0.56){RC_v^{\gamma_2}}
    \ar[d]_(0.4){\phi_2}^(0.4){\map{u,v}{\baru,v}} &
  \FL(u,v)
    \ar[ld]_{\phi_3}^(0.2){\hspace{-5mm}\map{u,v}{\baru,\barv}}
    \ar[ddd]^{\map{u,v}{w}} \\
  &
  \gamma\in\FL(u,\baru,v,\barv)\left/\left(\!\!{\begin{array}{c}
      u=e^{\ad \gamma}\baru \\
      v=e^{\ad \gamma}\barv
    \end{array}}\!\!\right)\right.
    \ar[d]_{{\map{u,v}{w}}}^{{\map{\baru,\barv}{\barw}}} &
  \\
  &
  \gamma_4\in\FL(w,\barw)\left/w=e^{\ad \gamma_4}\barw\right. &
  \\
  \gamma_4\in\FL(w)
    \ar[ur]^{\phi_4}_{w\to w}
    \ar[rr]^{RC_w^{\gamma_4}} &
  &
  \FL(w)
    \ar[ul]^{\phi_5}_{w\to\barw}
}
$}
Consider the diagram on the right. In it, the vertical and diagonal
arrows are all substitution homomorphisms as indicated. The horizontal
arrows are $RC$ maps as indicated. The element $\gamma$ lives in the
upper left corner of the diagram, but equally makes sense in the upper
of the central spaces.  We denote its image via $RC_u^{\gamma}$ by
$\gamma_2$, and think of it as an element of the middle space in the
top row. Likewise $\gamma_4:=\gamma\sslash t^{uv}_w$ lives in both the
bottom left space and the bottom of the two middle spaces.

It requires a minimal effort to show that the map at the very centre of
the diagram is well defined. The commutativity of the triangles in the
diagram follows from Comment~\ref{com:CRC}, and the commutativity of the
trapezoids is obvious. Hence the diagram is overall commutative. Reading
it from the top left to the bottom right along the left and the bottom
edges gives the left hand side of Equation~\eqref{eq:nts-t}, and along
the top and the right edges gives the right hand side.

\noindent{\bf Meta-action axiom $h$.} Expanding~\eqref{eq:haction}, we need
to prove
\[
  \lambda\sslash hm^{xy}_z
    \sslash RC^{\bch(\lambda_x,\lambda_y)}_u
  = \lambda\sslash RC^{\lambda_x}_u
    \sslash RC^{\lambda_y\sslash RC^{\lambda_x}_u}_u
    \sslash hm^{xy}_z.
\]
Using commutativities as in Equation~\eqref{eq:hmcc} and denoting
$\alpha=\lambda_x$ and $\beta=\lambda_y$ we can cancel the
$hm^{xy}_z$'s, and we are left with
\begin{equation} \label{eq:RCh}
  RC^{\bch(\alpha,\beta)}_u
  \overset{?}{=} RC^{\alpha}_u \sslash RC^{\beta'}_u,
  \qquad \text{where} \qquad
  \beta':=\beta\sslash RC^{\alpha}_u.
\end{equation}
This last equality follows from a careful inspection of the following
commutative diagram:
\begin{equation} \label{eq:RCProof}
\xymatrix@C=-0.75in{
  \FL(u) \ar[rr]^{RC^{\alpha}_u} \ar[dr]
    && \FL(u)
      \ar[rr]^{RC^{\beta'}_u}
      \ar[dl]^(0.4){u\to\baru} \ar[dr]_(0.4){u\to\baru}
    && \FL(u) \ar[dl]_{u\to\barbaru} \\
  & \FL(u,\baru)\left/\left(u=e^{\ad\alpha}\baru\right)\right. \ar[dr]
    && \FL(\baru,\barbaru)\left/\left(
        \baru=e^{\ad\beta'}\barbaru
      \right)\right. \ar[dl] \\
  && \FL(u,\baru,\barbaru)\left/{\begin{pmatrix}
        u=e^{\ad\alpha}\baru, \\
        \baru=e^{\ad\beta'}\barbaru
      \end{pmatrix}}\right.
}
\end{equation}

\parpic[r]{$\xymatrix@C=-0.35in{
  \FL(u) \ar@{-->}[rr] \ar[dr] &
    & \FL(u) \ar[dl]_{u\to\barbaru} \\
  & \FL(u,\barbaru)\left/\left(
      u=e^{\ad\bch(\alpha,\beta)}\barbaru
  \right)\right.
}$}
Indeed, by the definition of $RC_u^{\alpha}$ we have that
$\beta'=\beta$ modulo the relation $u=e^{\ad\alpha}\baru$. So in
the bottom space, $u=e^{\ad\alpha}\baru =
e^{\ad\alpha}e^{\ad\beta'}\barbaru =
e^{\ad\alpha}e^{\ad\beta}\barbaru =
e^{\bch(\ad\alpha,\ad\beta)}\barbaru =
e^{\ad\bch(\alpha,\beta)}\barbaru$. Hence if we concentrate
on the three corners of~\eqref{eq:RCProof}, we see the diagram
on the right, whose top row is both $RC^{\alpha}_u
\sslash RC^{\beta'}_u$ and the definition of
$RC^{\bch(\alpha,\beta)}_u$. \qed
\end{proof}

It remains to construct $\zeta_0\colon\Kbhz\to M_0$ by proclaiming its
values on the generators:
\[ \zeta_0(t\epsilon_u):=(),
  \qquad \zeta_0(h\epsilon_x):=(x\to 0),
  \qquad\text{and}\qquad
  \zeta_0(\rho^\pm_{ux}):=(x\to\pm u).
\]

\begin{proposition} \label{prop:zetazero}
$\zeta_0$ is well defined; namely, the values above satisfy the relations
in Definition~\ref{def:Kbhz}.
\end{proposition}

\begin{proof} We only verify the Conjugation
Relation~\eqref{eq:ConjugationRelation}, as all other relations are easy.
On the left we have
\begin{multline*}
  \rho^+_{ux}\rho^+_{vy}\rho^+_{wz}
    \xrightarrow{\zeta_0}
  (x\to u,\,y\to v,\,z\to w)
    \xrightarrow{tm^{vw}_v}
  (x\to u,\,y\to v,\,z\to v) \\
    \xrightarrow{hm^{xy}_x}
  (x\to\bch(u,v),\,z\to v)
    \xrightarrow{tha^{uz}}
  (x\to\bch(e^{\ad v}(u),v),\,z\to v),
\end{multline*}
while on the right it is
\[
  \rho^+_{vx}\rho^+_{wz}\rho^+_{uy}
    \xrightarrow{\zeta_0}
  (x\to v,\,y\to u,\,z\to w)
    \xrightarrow{tm^{vw}_v \sslash hm^{xy}_x}
  (x\to\bch(v,u),\,z\to v),
\]
and the equality follows because $\bch(e^{\ad
v}(u),v)=\log(e^ve^ue^{-v}\cdot e^v)=\bch(v,u)$. \qed
\end{proof}

As we shall see in Section~\ref{sec:ft}, $\zeta_0$ is related to
the tree part of the Kontsevitch integral. Thus by finite-type
folklore~\cite{Bar-Natan:OnVassiliev, HabeggerMasbaum:Milnor},
when evaluated on string links (i.e., pure tangles) $\zeta_0$
should be equivalent to the collection of all Milnor $\mu$
invariants~\cite{Milnor:LinkGroups}. No proof of this fact will be
provided here.

\section{The Wheel-Valued Spice and the Invariant $\zeta$} \label{sec:zeta}

This is perhaps the most important section of this paper. In it we construct
the wheels part of the full trees-and-wheels MMA $M$ and the full
tree-and-wheels invariant $\zeta\colon\Kbh\to M$.

\subsection{Cyclic words, $\diver_u$, and $J_u$} \label{subsec:divJ}

The target MMA, $M$, of the extended invariant $\zeta$ is an extension
of $M_0$ by ``wheels'', or equally well, by ``cyclic words'', and the
main difference between $M$ and $M_0$ is the addition of a wheel-valued
``spice'' term $J_u(\lambda_x)$ to the meta-action $tha^{ux}$. We first
need the ``infinitesimal version'' $\diver_u$
of $J_u$.

Recall that if $T$ is a set (normally, of tail labels), we denote by
$\FL(T)$ the graded completion of the free Lie algebra on the generators
in $T$. Similarly we denote by $\FA(T)$ the the graded completion of
the free associative algebra on the generators in $T$, and by $\CW(T)$
the graded completion of the vector space of cyclic words on $T$,
namely, $\CW(T):=\FA(T)/\{uw=wu\colon u\in T, w\in\FA(T)\}$. Note that
the last is a vector space quotient --- we mod out by the vector-space
span of $\{uw=wu\}$, and not by the ideal generated by that set. Hence
$\CW$ is not an algebra and not ``commutative''; merely, the words
in it are invariant under cyclic permutations of their letters. We
often call the elements of $\CW$ ``wheels''. Denote by $\tr$ the
projection $\tr\colon\FA\to\CW$ and by $\iota$ the standard inclusion
$\iota\colon\FL(T)\to\FA(T)$ ($\iota$ is defined to be the identity
on letters in $T$, and is then extended to the rest of $\FL$ using
$\iota([\lambda_1, \lambda_2]) := \iota(\lambda_1)\iota(\lambda_2)
- \iota(\lambda_2)\iota(\lambda_1)$). Note that operations defined by
``letter substitutions'' make sense on $\FA$ and on $\CW$. In particular,
the operation $RC_u^{\gamma}$ of Section~\ref{subsec:FLSuccess} makes
sense on $\FA$ and on $\CW$.

\parpic[r]{\input{figs/divu.pstex_t}}
The inclusion $\iota$ can be extended from ``trees'' (elements of $\FL$)
to ``wheels of trees'' (elements of $\CW(\FL)$). Given a letter $u\in
T$ and an element $\gamma\in\FL(T)$, we let $\diver_u\gamma$
be the sum of all ways of gluing the root of $\gamma$ to near
any one of the $u$-labelled leafs of $\gamma$; each such gluing
is a wheel of trees, and hence can be interpreted as an element of
$\CW(T)$. An example is on the right, and a formula-level definition
follows: we first define $\sigma_u\colon\FL(T)\to\FA(T)$ by setting
$\sigma_u(v):=\delta_{uv}$ for letters $v\in T$ and then setting
$\sigma_u([\lambda_1,\lambda_2]) := \iota(\lambda_1)\sigma_u(\lambda_2)
- \iota(\lambda_2)\sigma_u(\lambda_1)$, and then we set
$\diver_u(\gamma):=\tr(u\sigma_u(\gamma))$. An alternative definition
of a similar functional $\diver$ is in
\cite[Proposition~3.20]{AlekseevTorossian:KashiwaraVergne}, and some
further discussion is in \cite[Section 5.2]{WKO}.

Now given $u\in T$ and $\gamma\in\FL(T)$ define
\begin{equation} \label{eq:JDef}
  J_u(\gamma) := \int_0^1ds\,\diver_u\!\left(
    \gamma \sslash RC_u^{s\gamma}
  \right) \sslash C_u^{-s\gamma}.
\end{equation}
Note that at degree $d$, the integrand in the above formula is a degree
$d$ element of $\CW(T)$ with coefficients that are polynomials of degree
at most $d-1$ in $s$. Hence the above formula is entirely algebraic. The
following (difficult!) proposition contains all that we will need to
know about $J_u$.

\begin{proposition} \label{prop:JProperties} If
$\alpha,\beta,\gamma\in\FL$ then the following three equations hold:
\begin{equation} \label{eq:JhProperty}
  J_u(\bch(\alpha,\beta)) = J_u(\alpha)
    + J_u(\beta\sslash RC_u^{\alpha}) \sslash C_u^{-\alpha},
\end{equation}
\begin{equation} \label{eq:JuvProperty}
  J_u(\alpha) - J_u(\alpha\sslash RC_v^\beta)\sslash C_v^{-\beta}
  = J_v(\beta) - J_v(\beta\sslash RC_u^\alpha)\sslash C_u^{-\alpha}
\end{equation}
\begin{equation} \label{eq:JtProperty}
  J_w(\gamma\sslash tm^{uv}_w)
  = \left(
    J_u(\gamma) + J_v(\gamma\sslash RC_u^\gamma)\sslash C_u^{-\gamma}
  \right) \sslash tm^{uv}_w
\end{equation}
\end{proposition}

We postpone the proof of this proposition to
Section~\ref{subsec:JProperties}.

\begin{remark} \label{rem:JCharacterization}
$J_u$ can be characterized as the unique functional
$J_u\colon\FL(T)\to\CW(T)$ which satisfies Equation~\eqref{eq:JhProperty}
as well as the conditions $J_u(0)=0$ and
\begin{equation} \label{eq:dJ0}
  \left.\frac{d}{d\epsilon}J_u(\epsilon\gamma)\right|_{\epsilon=0}
  = \diver_u(\gamma),
\end{equation}
which in themselves are easy consequences of the definition of $J_u$,
Equation~\eqref{eq:JDef}. Indeed, taking $\alpha=s\gamma$ and
$\beta=\epsilon\gamma$ in Equation~\eqref{eq:JhProperty}, where $s$
and $\epsilon$ are scalars, we find that
\[ J_u((s+\epsilon)\gamma) = J_u(s\gamma)
    + J_u(\epsilon\gamma\sslash RC_u^{s\gamma}) \sslash C_u^{-s\gamma}. \]
Differentiating the above equation with respect to $\epsilon$ at
$\epsilon=0$ and using Equation~\eqref{eq:dJ0}, we find that
\[ \frac{d}{ds}J_u(s\gamma) =
  \diver_u(\gamma\sslash  RC_u^{s\gamma}) \sslash C_u^{-s\gamma}, \]
and integrating from $0$ to $1$ we get Equation~\eqref{eq:JDef}.
\end{remark}

Finally for this section, one may easily verify that the degree $1$ piece
of $\CW$ is preserved by the actions of $C_u^\gamma$ and $RC_u^\gamma$, and
hence it is possible to reduce modulo degree $1$. Namely, set
$\CWr(T):=\CW(T)/\text{deg 1} = \CW^{\,>1}(T)$, and all operations remain
well defined and satisfy the same identities.

\subsection{The MMA $M$} \label{subsec:MMAM}
Let $M$ be the collection $\{M(T;H)\}$, where
\[ M(T;H):=\FL(T)^H\times\CWr(T)=M_0(T;H)\times\CWr(T) \]
(I really mean $\times$, not $\otimes$). The collection $M$ has MMA
operations as follows:

\begin{itemize}
\item $t\sigma^u_v$, $t\eta^u$, and $tm^{uv}_w$ are defined by the same
  formulae as in Section~\ref{subsec:FLSuccess}. Note that these formulae
  make sense on $\CW$ and on $\CWr$ just as they do on $\FL$.
\item $h\sigma^x_y$, $h\eta^x$, and $hm^{xy}_z$ are extended to act as the
  identity on the $\CWr(T)$ factor of $M(T;H)$.
\item If $\mu_i=(\lambda_i;\omega_i)\in M(T_i;H_i)$ for $i=1,2$ (and, of
  course, $T_1\cap T_2=\emptyset=H_1\cap H_2$), set
  \[ \mu_1\ast\mu_2:=
    (\lambda_1\ast\lambda_2;\ \iota_1(\omega_1)+\iota_2(\omega_2)),
  \]
  where $\iota_i$ are the obvious inclusions
  $\iota_i\colon\CWr(T_i)\to\CWr(T_1\cup T_2)$.
\item The only truly new definition is that of $tha^{ux}$:
  \[ (\lambda;\omega)\sslash tha^{ux} :=
    (\lambda;\ \omega+J_u(\lambda_x)) \sslash RC_u^{\lambda_x}.
  \]
  Thus the ``new'' $tha^{ux}$ is just the ``old'' $tha^{ux}$, with an added
  term of $J_u(\lambda_x)$.
\item $t\epsilon_u:=(();\ 0)$ and $h\epsilon_x:=((x\to 0);\ 0)$.
\end{itemize}

\begin{theorem} \label{thm:zeta} $M$, with the operations defined above,
is a meta-monoid-action (MMA). Furthermore, if $\zeta\colon\Kbhz\to M$ is
defined on the generators in the same way as $\zeta_0$, except extended
by $0$ to the $\CWr$ factor, $\zeta(\rho^\pm_{ux}):=((x\to\pm u);\ 0)$,
then it is well-defined; namely, the values above
satisfy the relations in Definition~\ref{def:Kbhz}.
\end{theorem}

\begin{proof} Given Theorem~\ref{thm:MMAZero} and
Proposition~\ref{prop:zetazero}, the only non-obvious checks remaining are
the ``wheel parts'' of the main equations defining and MMA, Equations
\eqref{eq:hassoc}--\eqref{eq:haction}, the Conjugation
Relation~\eqref{eq:ConjugationRelation}, and the FI
relation~\eqref{eq:FI}. As the only interesting wheels-creation occurs
with the operation $tha$, \eqref{eq:hassoc} and \eqref{eq:tassoc} are
easy. As easily $J_u(v)=0$ if $u\neq v$, no wheels are created by the
$tha$ action within the proof of Proposition~\ref{prop:zetazero}, so that
proof still holds. We are left with \eqref{eq:thatha}--\eqref{eq:haction}
and \eqref{eq:ConjugationRelation}--\eqref{eq:FI}.

Let us start with the wheels part of Equation~\eqref{eq:thatha}. If
$\mu=((x\to\alpha,y\to\beta,\ldots);\omega)\in M$, then
\[ \mu\sslash tha^{ux}
  = ((
    x\to\alpha\sslash RC_u^{\alpha},
    y\to\beta\sslash RC_u^{\alpha},
    \ldots);
    (\omega+J_u(\alpha))\sslash RC_u^{\alpha})
\]
and hence the wheels-only part of $\mu\sslash tha^{ux}\sslash tha^{vy}$ is
\[ \omega\sslash RC_u^{\alpha}
    \sslash RC_v^{\beta\sslash RC_u^{\alpha}}
  +J_u(\alpha)\sslash RC_u^{\alpha}
    \sslash RC_v^{\beta\sslash RC_u^{\alpha}}
  +J_v(\beta\sslash RC_u^{\alpha})
    \sslash RC_v^{\beta\sslash RC_u^{\alpha}}
\]
\[ = \left[
    \omega+J_u(\alpha)
      + J_v(\beta\sslash RC_u^{\alpha})\sslash C_u^{-\alpha}
    \right]
    \sslash RC_u^{\alpha} \sslash RC_v^{\beta\sslash RC_u^{\alpha}}.
\]
In a similar manner, the wheels-only part of $\mu\sslash tha^{vy}\sslash
tha^{ux}$ is
\[ \left[
    \omega+J_v(\beta)
      + J_u(\alpha\sslash RC_v^{\beta})\sslash C_v^{-\beta}
    \right] 
    \sslash RC_v^{\beta} \sslash RC_u^{\beta\sslash RC_v^{\beta}}.
\]
Using Equation~\eqref{eq:RCuRCv} the operators outside the square brackets
in the above two formulae are the same, and so we only need to verify that
\[
  \omega+J_u(\alpha) + J_v(\beta\sslash RC_u^{\alpha})\sslash C_u^{-\alpha}
  = \omega+J_v(\beta) + J_u(\alpha\sslash RC_v^{\beta})\sslash C_v^{-\beta}.
\]
But this is Equation~\eqref{eq:JuvProperty}. In a similar manner, the
wheels parts of Equations~\eqref{eq:taction} and~\eqref{eq:haction}
reduce to Equations~\eqref{eq:JtProperty} and~\eqref{eq:JhProperty},
respectively. One may also verify that no wheels appear within
Equation~\eqref{eq:ConjugationRelation}, and that wheels appear in
Equation~\eqref{eq:FI} only in degree $1$, which is eliminated in $\CWr$.
\qed
\end{proof}

Thus we have a tree-and-wheel valued invariant $\zeta$ defined on $\Kbhz$,
and thus $\delta\sslash\zeta$ is a tree-and-wheel valued invariant of
tangles and w-tangles.

As we shall see in Section~\ref{sec:ft}, the wheels part $\omega$ of $\zeta$
is related to the wheels part of the Kontsevitch integral. Thus by
finite-type folklore (e.g.~\cite{Kricker:Kontsevich}), the Abelianization of
$\omega$ (obtained by declaring all the letters in
$\CW(T)$ to be commuting) should be closely related to the multi-variable
Alexander polynomial. More on that in Section~\ref{sec:uA}. I don't know
what the bigger (non-commutative) part of $\omega$ measures.

\section{Some Computational Examples} \label{sec:Computations}
Part of the reason I am happy about the invariant $\zeta$ is that it
is relatively easily computable. Cyclic words are easy to implement,
and using the Lyndon basis (e.g.~\cite[Chapter~5]{Reutenauer:FreeLie}),
free Lie algebras are easy too. Hence I include here a demo-run of a
rough implementation, written in {\em Mathematica}. The full source
files are available at~\web{}.

\subsection{The Program}

First we load the package {\tt FreeLie.m}, which contains a collection of
programs to manipulate series in completed free Lie algebras and series of
cyclic words. We tell {\tt FreeLie.m} to show series by default only up to
degree 3, and that if two (infinite) series are compared, they are to be
compared by default only up to degree 5:

\mathinclude{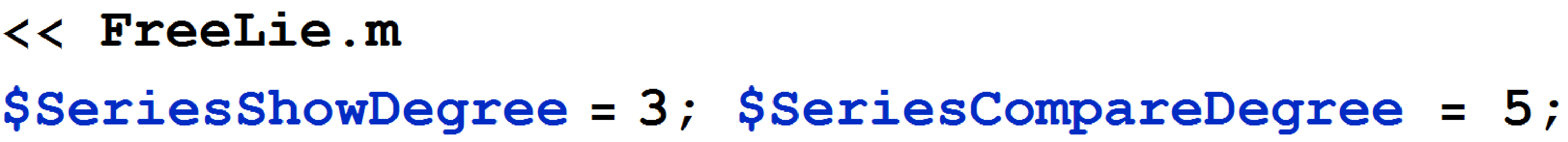}

Merely as a test of {\tt FreeLie.m}, we tell it to set {\tt t1} to be
$\bch(u,v)$. The computer's response is to print that series to degree 3:

\dialoginclude{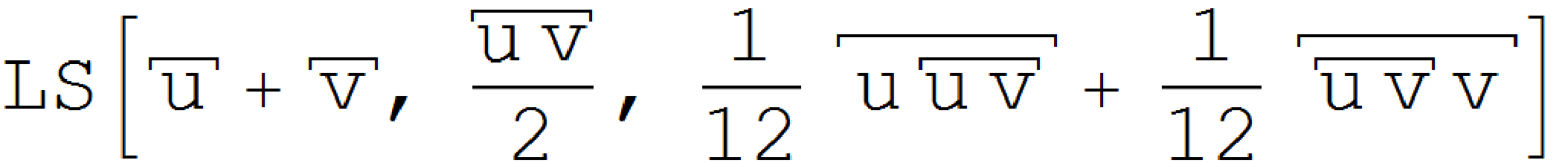}

Note that by default Lie series are printed in ``top bracket form'', which
means that brackets are printed above their arguments, rather than around
them. Hence $\overbracket[0.5pt][1pt]{u\overbracket[0.5pt][1pt]{uv}}$
means $[u,[u,v]]$. This practise is especially advantageous when it is
used on highly-nested expressions, when it becomes difficult for the
eye to match left brackets with the their corresponding right brackets.

Note also that that {\tt FreeLie.m} utilizes {\em lazy evaluation}, meaning
that when a Lie series (or a series of cyclic words) is defined, its
definition is stored but no computation take place until it is printed or
until its value (at a certain degree) is explicitly requested. Hence {\tt
t1} is a reference to the entire Lie series $\bch(u,v)$, and not merely to
the degrees 1--3 parts of that series, which are printed above. Hence when
we request the value of {\tt t1} to degree 6, the computer complies:

\dialoginclude{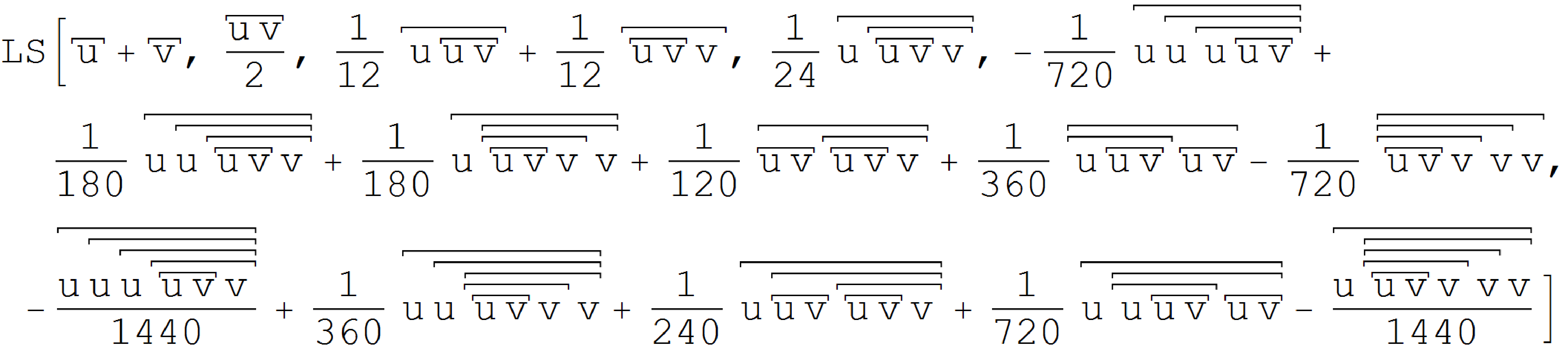}

The package {\tt FreeLie.m} know about various free Lie algebra operations,
but not about our specific circumstances. Hence we have to make some
further definitions. The first few are set-theoretic in nature. We define
the ``domain'' of a function stored as a list of {\it key$\to$value} pairs
to be the set of ``first elements'' of these pairs; meaning, the set of
keys. We define what it means to remove a key (and its corresponding value),
and likewise for a list of keys. We define what it means for two functions
to be equal (their domains must be equal, and for every key $\#$, we are to
have $\#\sslash f_1=\#\sslash f_2$). We also define how to apply a Lie
morphism {\tt mor} to a function (apply it to each value), and how to
compare $(\lambda,\omega)$ pairs (in $\FL(T)^H\times\CWr(T)$):

\mathinclude{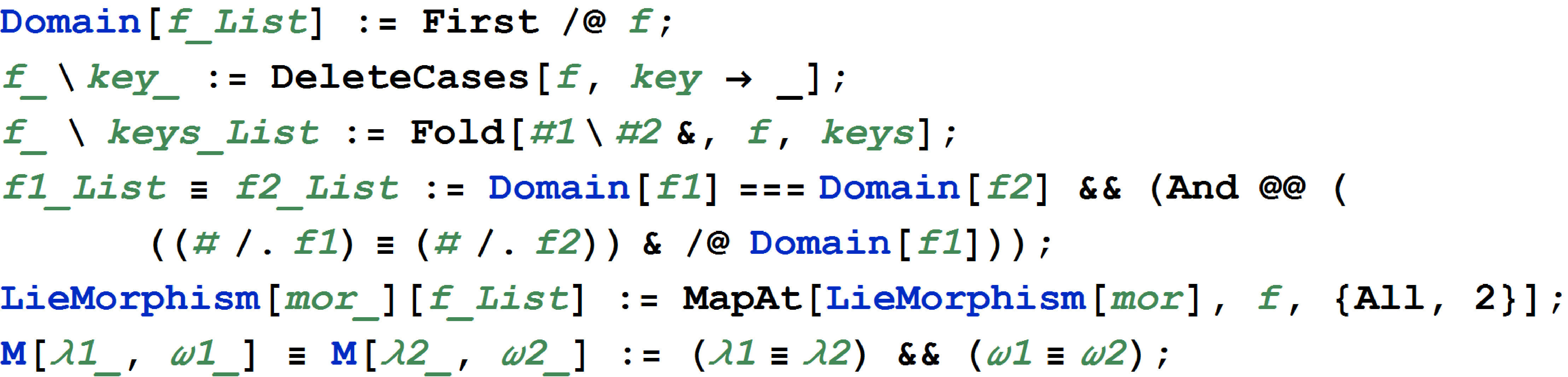}

Next we enter some free-Lie definitions that are not a part of {\tt
FreeLie.m}. Namely we define $RC_{u,\baru}^{\gamma}(s)$ to be the result of
``stable application'' of the morphism $u\to e^{\ad(\gamma)}(\baru)$ to
$s$ (namely, apply the morphism repeatedly until things stop changing; at
any fixed degree this happens after a finite number of iterations). We
define $RC_u^{\gamma}$ to be $RC_{u,\baru}^{\gamma}\sslash(\baru\to
u)$. Finally, we define $J$ as in Equation~\eqref{eq:JDef}:

\mathinclude{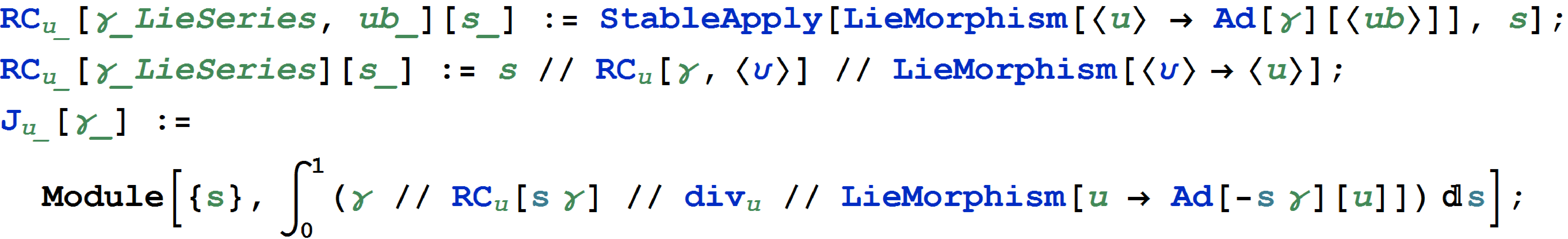}

Mostly to introduce our notation for cyclic words, let us compute
$J_v(\bch(u,v))$ to degree 4. Note that when a series of wheels is printed
out here, its degree 1 piece is greyed out to honour the fact that it
``does not count'' within $\zeta$:

\dialoginclude{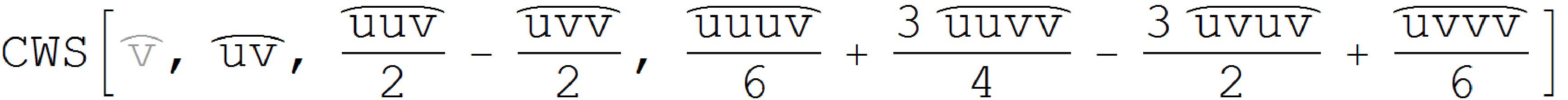}

Next is a series of definitions that implement the definitions of $\ast$,
$tm$, $hm$, and $tha$ following Sections~\ref{subsec:FLSuccess}
and~\ref{subsec:MMAM}:

\mathinclude{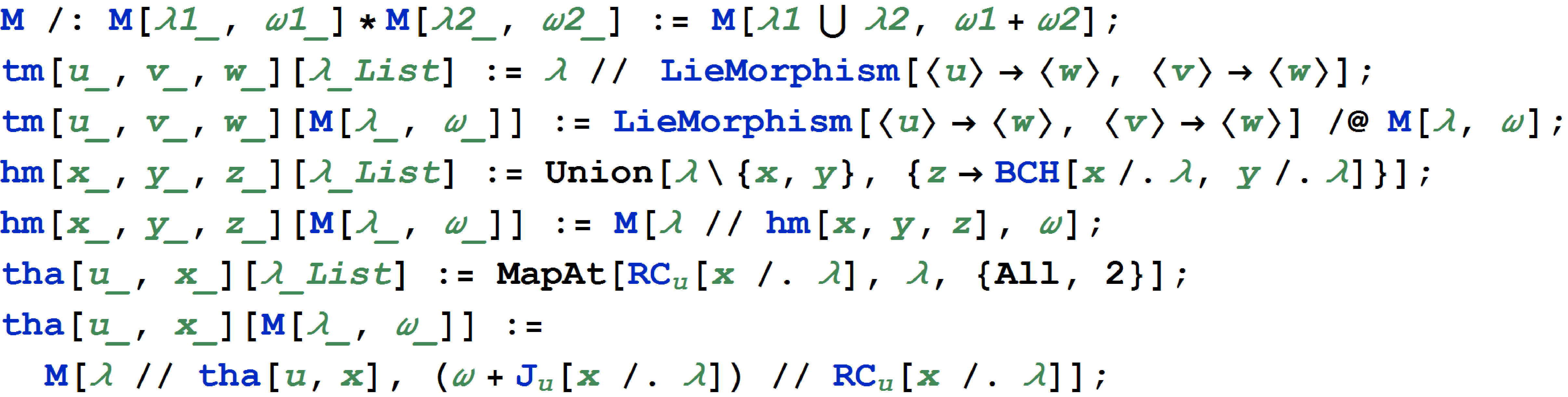}

Next we set the values of $\zeta(t\epsilon_x)$ and $\zeta(\rho^{\pm}_{ux})$,
which we simply denote $t\epsilon_x$ and $\rho^{\pm}_{ux}$:

\mathinclude{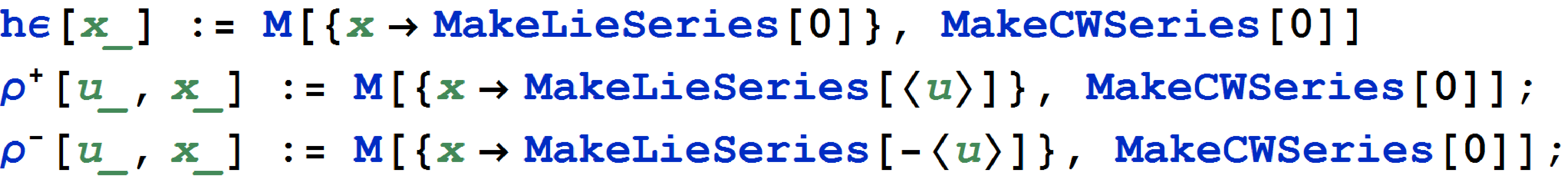}

The final bit of definitions have to do with 3-dimensional tangles. We set
$R^+$ to be the value of $\zeta(\delta(\overcrossing))$ as in the proof of
Theorem~\ref{thm:Generators}, likewise for $R^-$, and we define {\tt dm}
following Equation~\eqref{eq:dm}:

\mathinclude{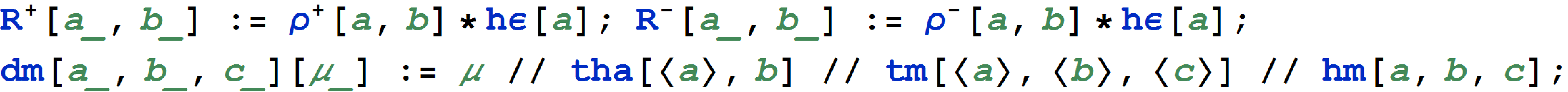}

\subsection{Testing Properties and Relations} It is always good to test
both the program and the math by verifying that the operations we have
implemented satisfy the relations predicted by the mathematics. As
a first example, we verify the meta-associativity of $tm$. Hence in
line 1 below we set {\tt t1} to be the element $t_1=((x\to u+v+w,
y\to[u,v]+[v,w]);\,uvw)$ of $M(u,v,w;x,y)$. In line 2 we compute
$t_1\sslash tm^{uv}_u$, in line 3 we compute $t_2:=t_1\sslash
tm^{uv}_u\sslash tm^{uw}_u$ and store its value in {\tt t2}, in line 4 we
compute $t_1\sslash tm^{vw}_v$, in line 5 we compute $t_3:=t_1\sslash
tm^{vw}_v\sslash tm^{uv}_u$ and store its value in {\tt t3}, and then in
line 6 we test if $t_2$ is equal to $t_3$. The computer thinks the answer
is ``{\tt True}'', at least to the degree tested:

\dialoginclude{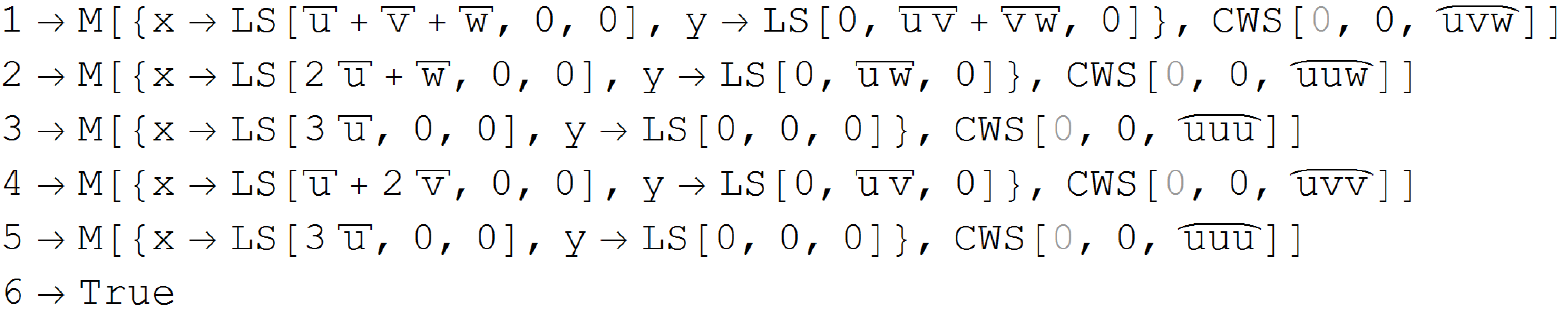}

The corresponding test for the meta-associativity of $hm$ is a bit
harder, yet produces the same result. Note that we have declared {\tt
\$SeriesCompareDegree} to be higher than {\tt \$SeriesShowDegree}, so the
``{\tt True}'' output below means a bit more than the visual comparison of
lines 3 and 5:

\dialoginclude{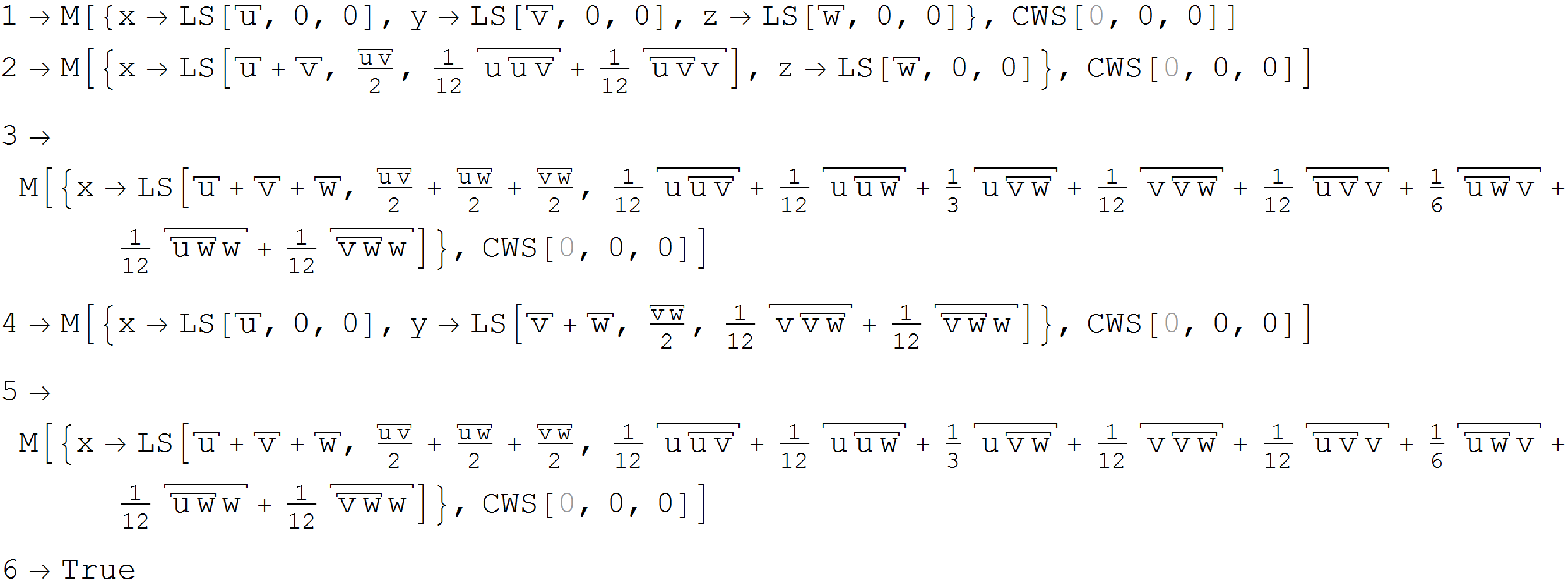}

We next test the meta-action axiom $t$ on $((x\to u+[u,t],y\to
u+[u,t]);\, uu+tuv)$ and the meta-action axiom $h$ on $((x\to u+[u,v],y\to
v+[u,v]);\,uu+uvv)$:

\dialoginclude{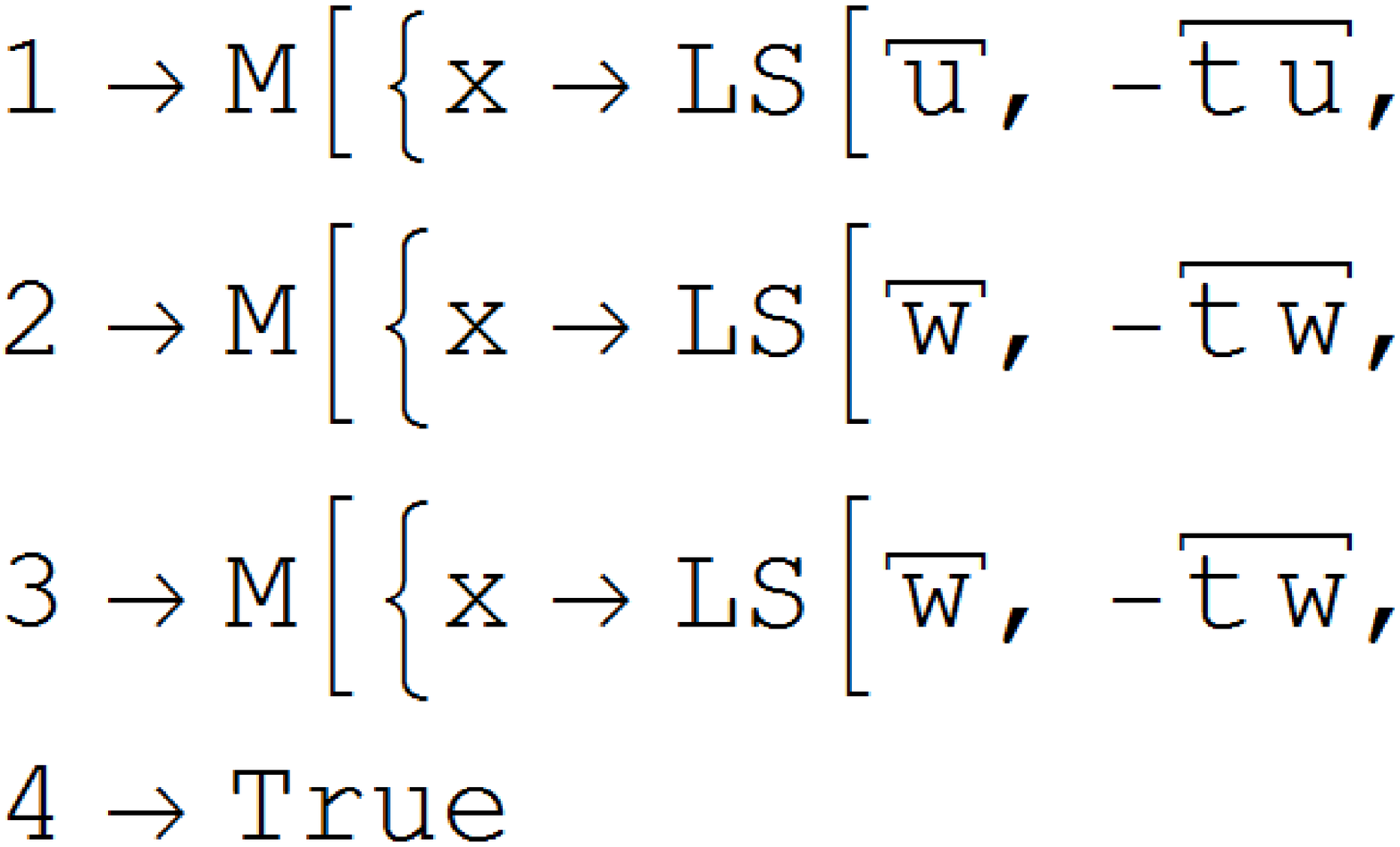}

\dialoginclude{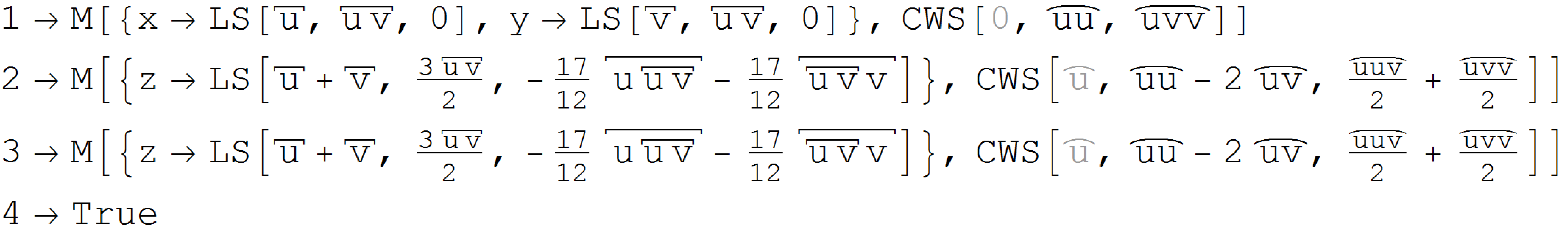}

And finally for this testing section, we test the Conjugation Relation of
Equation~\eqref{eq:ConjugationRelation}:

\dialoginclude{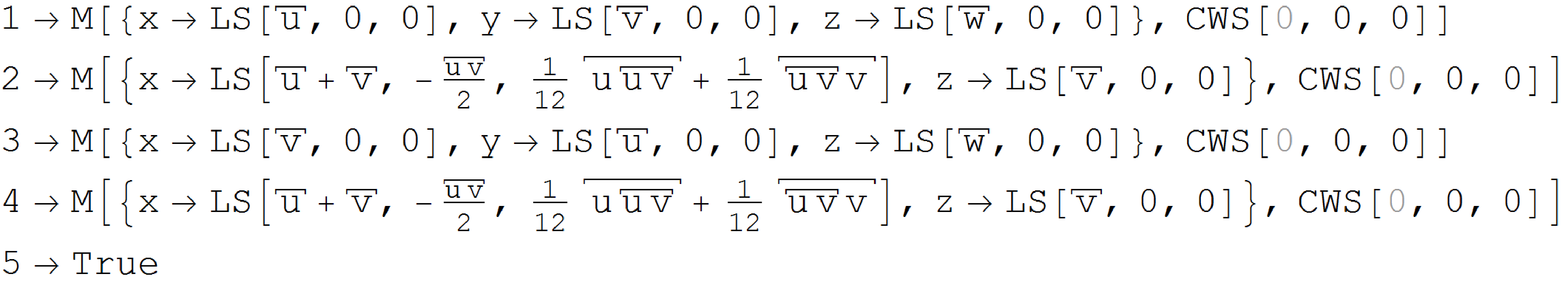}

\subsection{Demo Run 1 --- the Knot $8_{17}$} \label{subsec:Demo1}
We are ready for a more substantial computation --- the invariant of
the knot $8_{17}$. We draw $8_{17}$ in the plane, with all but the
neighbourhoods of the crossings dashed-out. We thus get a tangle $T_1$
which is the disjoint union of 8 individual crossings (4 positive and
4 negative). We number the 16 strands that appear in these 8 crossings
in the order of their eventual appearance within $8_{17}$, as seen below.

\parpic[r]{\input{figs/817.pstex_t}}
The 8-crossing tangle $T_1$  we just got has a rather boring $\zeta$
invariant, a disjoint merge of 8 $\rho^{\pm}$'s. We store it in {\tt
$\mu$1}. Note that we used numerals as labels, and hence in the expression
below top-bracketed numerals should be interpreted as symbols and not
as integers. Note also that the program automatically
converts two-digit numerical labels into alphabetical symbols, when
these appear within Lie elements. Hence in the output below, ``{\tt a}''
is ``10'', ``{\tt c}'' is ``12'', ``{\tt e}'' is ``14'', and ``{\tt g}''
is ``16'':

\dialoginclude{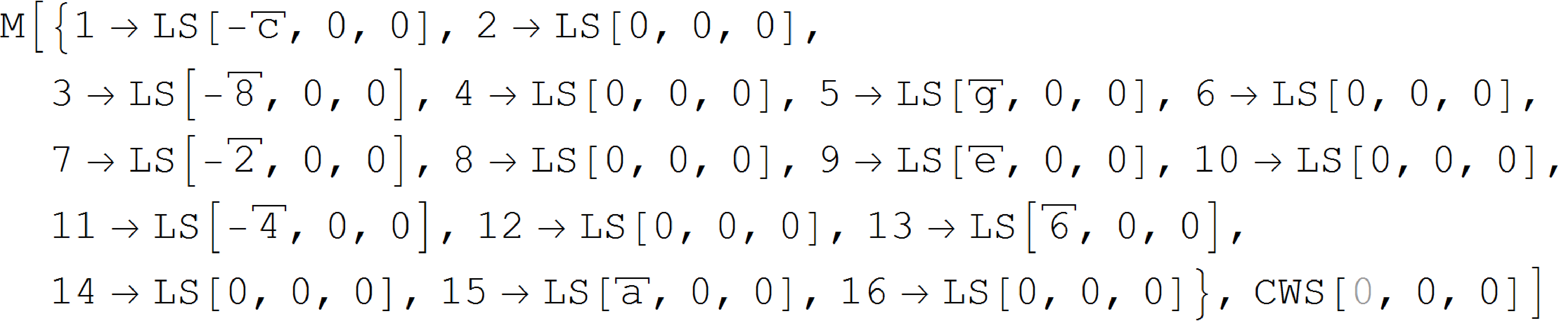}

Next is the key part of the computation. We ``sew'' together the strands of
$T_1$ in order by first sewing 1 and 2 and naming the result 1, then sewing
1 and 3 and naming the result 1 once more, and on until everything is sewn
together to a single strand named 1. This is done by applying $dm^{1k}_1$
repeatedly to {\tt $\mu$1}, for $k=2,\ldots,16$, each time storing the
result back again in {\tt $\mu$1}. Finally, we only wish to print the
wheels part of the output, and this we do to degree 6:

\dialoginclude{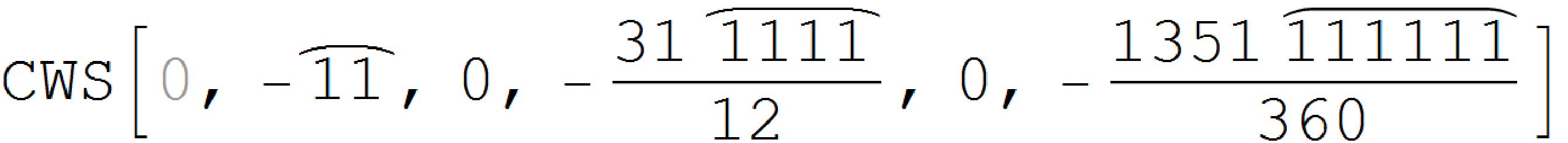}

Let $A(X)$ be the Alexander polynomial of $8_{17}$. Namely,
$A(X)=-X^{-3}+4X^{-2}-8X^{-1}+11-8X+4X^2-X^3$. For comparison with the
above computation, we print the series expansion of $\log A(e^x)$, also to
degree 6:

\dialoginclude{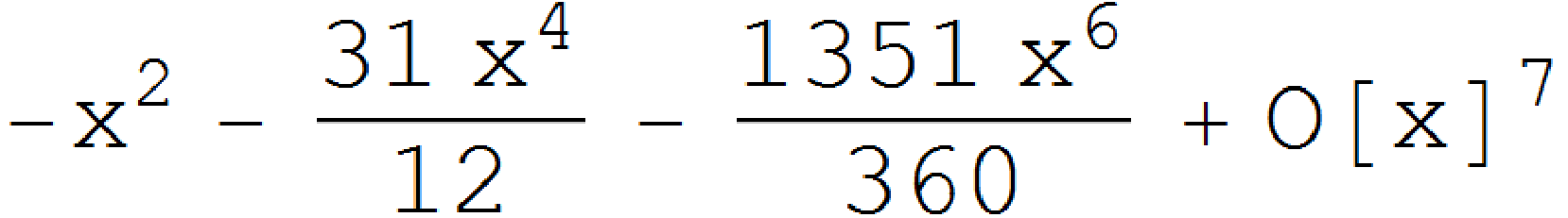}

\subsection{Demo Run 2 --- the Borromean Tangle} In a similar manner we
compute the invariant of the {\it rgb}-coloured Borromean tangle, shown below.

\parpic[r]{\imagetop{\input{figs/BorromeanTangle.pstex_t}}}
We label the edges near the crossings as shown, using the labels
$\{r,1,2,3\}$ for the $r$ component, $\{g,4,5,6\}$ for the $g$ component,
and $\{b,7,8,9\}$ for the $b$ component. We let {\tt $\mu$2} store the
invariant of the disjoint union of 6 independent crossings labelled as in
the Borromean tangle, we concatenate the numerically-labelled strands into
their corresponding letter-labelled strands, and we then print {\tt $\mu$2},
which now contains the invariant we seek:

\dialoginclude{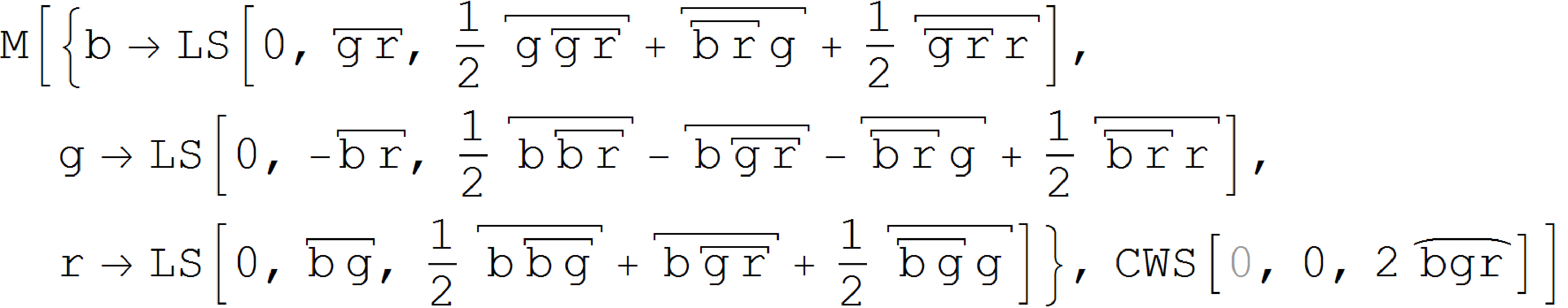}

We then print the $r$-head part of the tree part of the invariant to degree
5 (the $g$-head and $b$-head parts can be computed in a similar way, or
deduced from the cyclic symmetry of $r$, $g$, and $b$), and the wheels part
to the same degree:

\dialoginclude{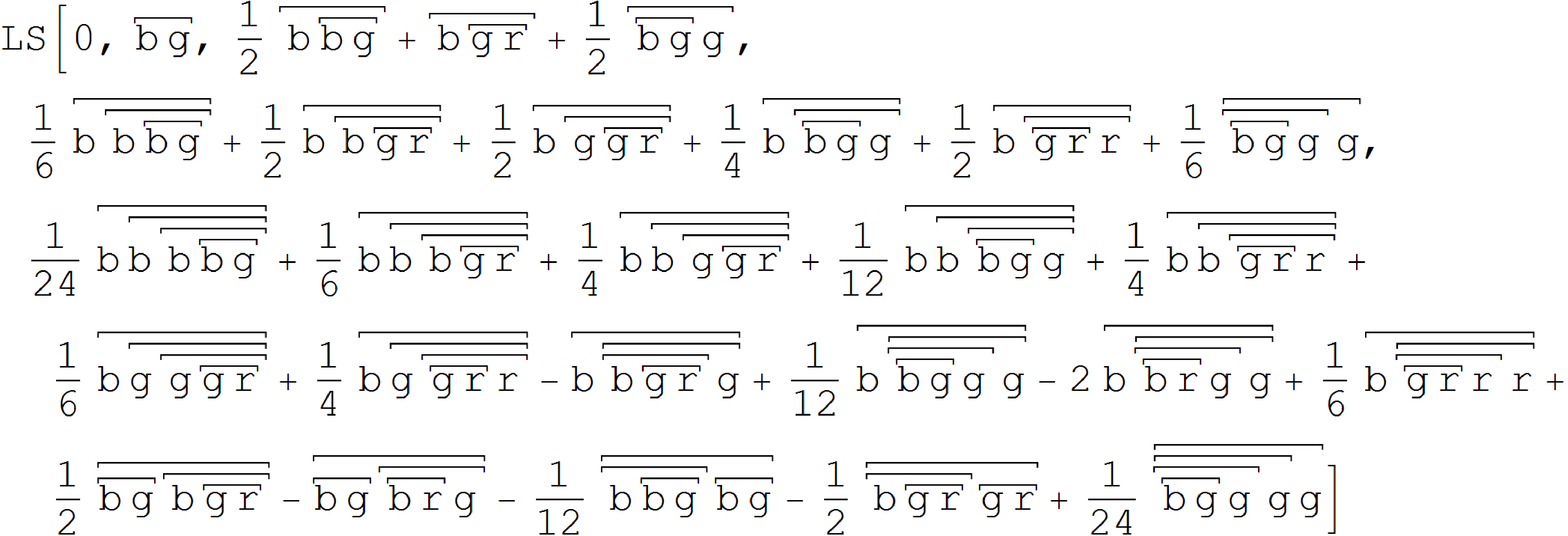}

\dialoginclude{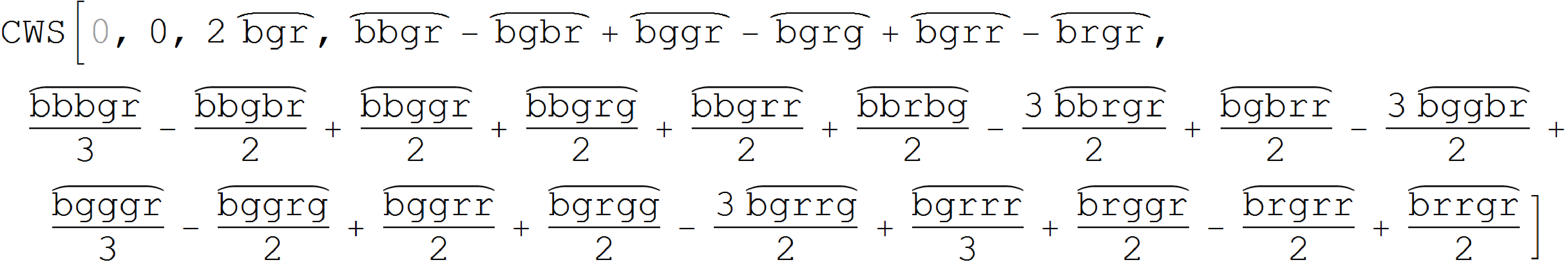}

A more graphically-pleasing presentation of the same values, with the
degree raised to 6, appears in Figure~\ref{fig:BorromeanInvariant}.

\begin{figure}
\[ \includegraphics[width=6in]{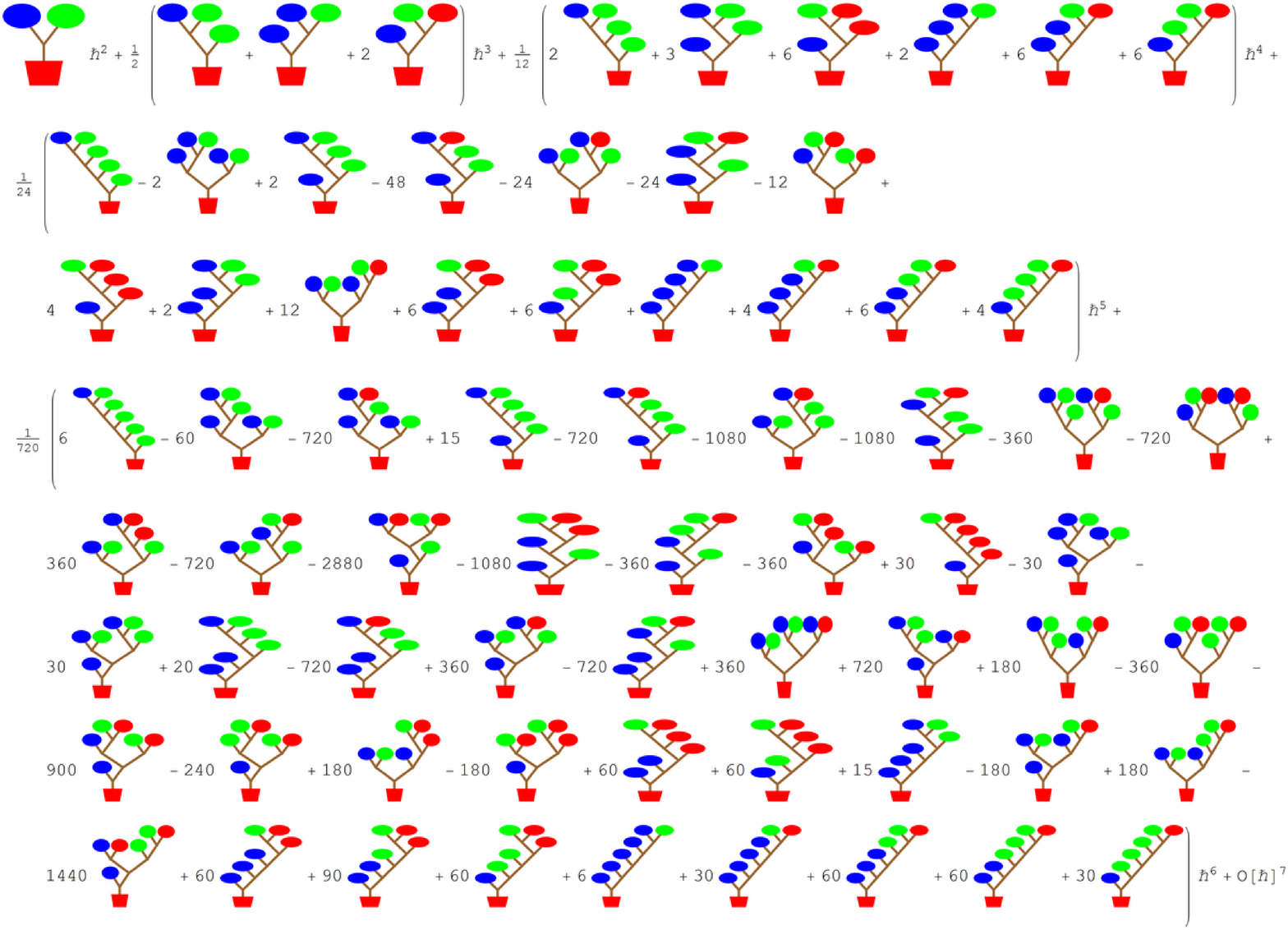} \]
\[ \includegraphics[width=6in]{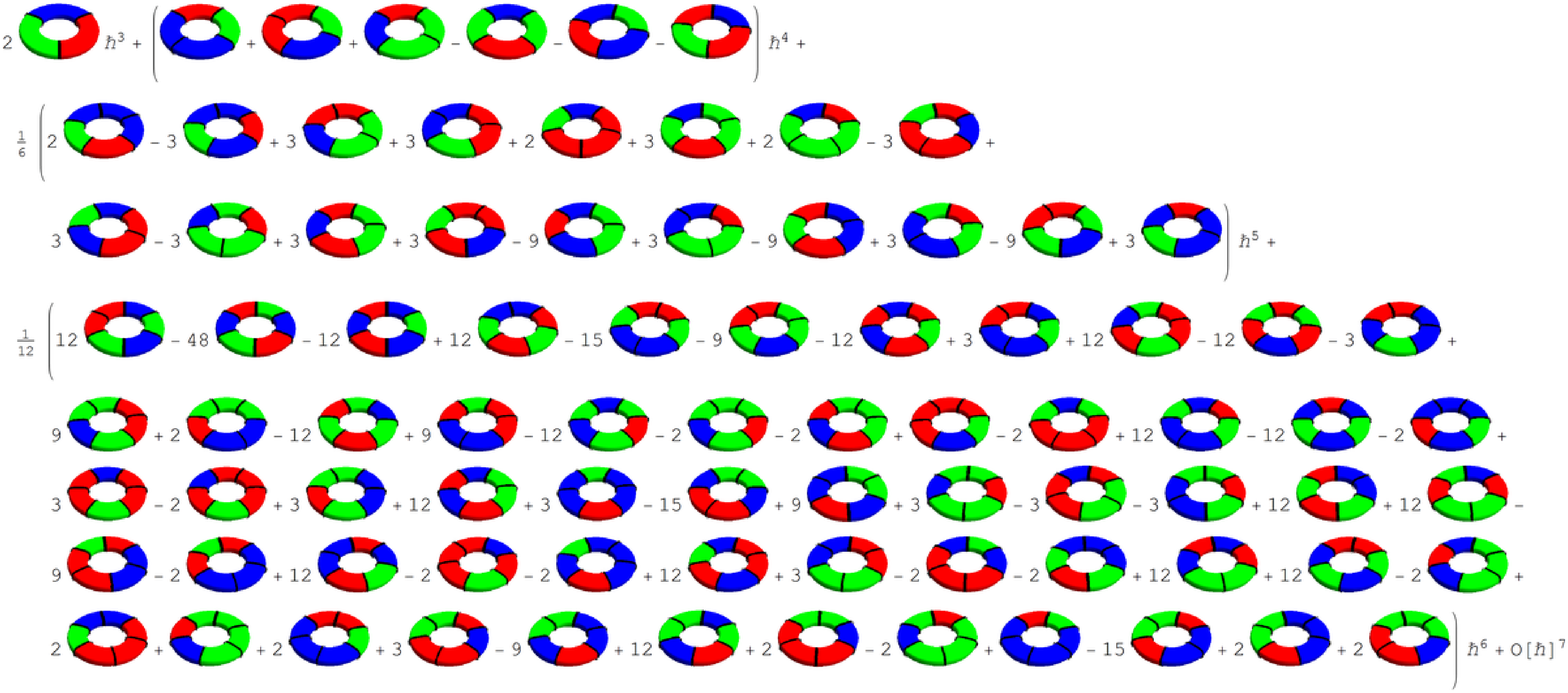} \]
\caption{
  The redhead part of the tree part, and the wheels part,
  of the invariant of the Borromean tangle, to degree 6.
} \label{fig:BorromeanInvariant}
\end{figure}

\section{Sketch of The Relation with Finite Type Invariants}
\label{sec:ft}

One way to view the invariant $\zeta$ of Section~\ref{sec:zeta} is as
a mysterious extension of the reasonably natural invariant $\zeta_0$
of Section~\ref{sec:trees}. Another is as a solution to a universal
problem --- as we shall see in this section, $\zeta$ is a universal
finite type invariant of objects in $\Kbhz$. Given that $\Kbhz$ is
closely related to $\wT$ (w-tangles), and given that much was already
said on finite type invariants of w-tangles in~\cite{WKO}, this section
will be merely a sketch, difficult to understand without reading much
of sections 1--5 of~\cite{WKO}, as well as the parts of section 6 that
concern with caps.

Over all, defining $\zeta$ using the language of Sections~\ref{sec:trees}
and~\ref{sec:zeta} is about as difficult as using finite type invariants.
Yet computing it using the language of Sections~\ref{sec:trees}
and~\ref{sec:zeta} is much easier while proving invariance is
significantly harder.

\subsection{A circuit algebra description of $\Kbhz$} \label{subsec:CA}
A w-tangle represents a collection of ribbon-knotted tubes in $\bbR^4$. It
follows from Theorem~\ref{thm:surjective} that every rKBH can be obtained
from a w-tangle by capping some of its tubes and ``puncturing'' the
rest, where ``puncturing'' a tube means ``replacing it with its spine,
a strand that runs along it''. Using thick red lines to denote tubes,
red bullets to denote caps, and dotted blue lines to denote punctured
tubes, we find that
\[ \imagetop{\input{figs/CA.pstex_t}}
  \qquad\imagetop{\input{figs/R1.pstex_t}}\hspace{-1in}
\]
Note that punctured tubes (meanings strands or ``hoops'') can only go under
capped tubes (balloons), and that while it is allowed to slide tubes ``over''
caps, it is not allowed to slide them ``under'' caps. Further explanations
and the meaning of ``$\CA$'' are in~\cite{WKO}. The ``red bullet''
subscript on the right hand side indicates that we restrict our attention
to the subspace in which all red strands are eventually capped. We leave it
to the reader to interpret the operations $hm$, $tha$, and $tm$ is this
language ($tm$ is non-obvious!).

\subsection{Arrow diagrams for $\Kbhz$} \label{subsec:ArrowDiagrams}
As in~\cite{WKO}, one we finite type invariants of elements on $\Kbhz$
bi considering iterated differences of crossings and non-crossings
(``virtual crossings''), and then again as in~\cite{WKO}, we find that
the arrow-diagram space $\Abh(T;H)$ corresponding to these invariants
may be described schematically as follows:
\[ \input{figs/Abh.pstex_t}. \]
In the above, arrow tails may land only on the red ``tail'' strands, but
arrow heads may land on either kind of strand. The ``Relations'' are the TC
and $\aft$ relations of~\cite[Section~2.3]{WKO}, the CP relation
of~\cite[Section~6.2]{WKO}, and the relation $D_L=D_R=0$, which corresponds
to the R1 relation ($D_L$ and $D_R$ are defined in~\cite[Section~3]{WKO}).

The operation $hm$ acts on $\Abh$ by concatenating two head stands. The
operation $tha$ acts by duplicating a head strand (with the usual
summation over all possible ways of reconnecting arrow-heads as
in~\cite[Section~2.5.1.6]{WKO}), changing the colour of one of the
duplicates to red, and then concatenating it to the beginning of some
tail strand.

We note that modulo the relations, one may eliminate all arrow-heads from
all tail strands. For diagrams in which there are no arrow-heads on tail
strands, the operation $tm$ is defined by merging together two tail
strands. The TC relation implies that arrow-tails on the resulting
tail-strand can be order in any desired way.

As in~\cite[Section~3.5]{WKO}, $\Abh$ has an alternative model in which
internal ``2-in 1-out'' trivalent vertices are allowed, and in which we
also impose the $\aAS$, $\aSTU$ and $\aIHX$ relations (ibid.).

\subsection{The algebra structure on $\Abh$ and its primitives}
\label{subsec:Primitives}
For any fixed finite sets $T$ and $H$, the space $\Abh(T;H)$ is a
co-commutative bi-algebra. Its product defined using the disjoint union
followed by the $tm$ operation on all tail strands and the $hm$ operation
on all head strands, and its co-product is the ``sum of all splittings''
as in~\cite[Section~3.2]{WKO}. Thus by Milnor-Moore, $\Abh(T;H)$ is
the universal enveloping algebra of its set of primitives $\Pbh$. The
latter is the set of connected diagrams in $\Abh$ (modulo relations),
and those, as in~\cite[Section~5.2]{WKO}, are the trees and the degree
$>1$ wheels. (Though note that even if $T=H=\{1,\ldots,n\}$, the algebra
structure on $\Abh(T;H)$ is different from the algebra structure on the
space $\calA^w(\uparrow_n)$ of ibid.). Identifying trees with $\FL(T)$
and wheels with $\CWr(T)$, we find that
\[ \Pbh(T;H)\cong\FL(T)^H\times\CWr(T)=M(T;H). \]

\begin{theorem} By taking logarithms (using formal power series and the
algebra structure of $\Abh$), $\Pbh(T;H)$ inherits the structure
of an MMA from the group-like elements of $\Abh$. Furthermore,
$\Pbh(T;H)$ and $M(T;H)$ are isomorphic as MMAs.
\end{theorem}

\noindent{\em Sketch of the proof.} Once it is established
that $\Pbh(T;H)$ is an MMA, that $tm$ and $hm$ act in the same
way as on $M$ and that the tree part of the action of $tha$
is given using the $RC$ operation, it follows that the wheels
part of the action of $tha$ is given by some functional $J'$ which
necessarily satisfies Equation~\eqref{eq:JhProperty}.  But according
to Remark~\ref{rem:JCharacterization}, Equation~\eqref{eq:JhProperty}
and a few auxiliary conditions determine $J$ uniquely. These conditions
are easily verified for $J'$, and hence $J'=J$. This concludes the proof.

Note that the above theorem and the fact that $\Pbh(T;H)$ is an MMA
provide an alternative proof of Proposition~\ref{prop:JProperties} which
bypasses the hard computations of Section~\ref{subsec:JProperties}. In
fact, personally I first knew that $J$ exists and satisfies
Proposition~\ref{prop:JProperties} using the reasoning of
this section, and only then I observed using the reasoning of
Remark~\ref{rem:JCharacterization} that $J$ must be given by the formula in
Equation~\eqref{eq:JDef}.

\subsection{The homomorphic expansion $\Zbh$} \label{subsec:Zbh}
As in~\cite[Sections~3.4 and~5.1]{WKO}, there is a homomorphic expansion
(a universal finite type invariant with good composition properties)
$\Zbh\colon\Kbhz\to\Abh$ defined by mapping crossings to exponentials
of arrows. It is easily verified that $\Zbh$ is a morphism of MMAs,
and therefore it is determined by its values on the generators
$\rho^{\pm}$ of $\Kbhz$, which are single crossings in the language of
Section~\ref{subsec:CA}. Taking logarithms we find that $\log\Zbh=\zeta$
on the generators and hence always, and hence $\zeta$ is the logarithm
of a universal finite type invariant of elements of $\Kbhz$.

\section{The Relation with the BF Topological Quantum Field Theory}
\label{sec:BF}

\subsection{Tensorial Interpretation} \label{subsec:TensorialInterpretation}
Given a Lie algebra $\frakg$, any element of $\FL(T)$ can be
interpreted as a function taking $|T|$ inputs in $\frakg$
and producing a single output in $\frakg$. Hence, putting
aside issues of completion and convergence, there is a map
$\tau_1\colon\FL(T)\to\Fun(\frakg^T\to\frakg)$, where in general,
$\Fun(X\to Y)$ denotes the space of functions from $X$ to $Y$.
To deal with completions more precisely, we pick a formal parameter
$\hbar$, multiply the degree $k$ part of $\tau_1$ by $\hbar^k$,
and get a perfectly good $\tau = \tau_\frakg\colon\FL(T) \to
\Fun(\frakg^T\to\frakg\llbracket\hbar\rrbracket)$, where in general,
$V\llbracket\hbar\rrbracket := \bbQ\llbracket\hbar\rrbracket\otimes V$
for any vector space $V$. The map $\tau$ obviously extends to
$\tau\colon\FL(T)^H\to\Fun(\frakg^T\to\frakg^H\llbracket\hbar\rrbracket)$.

Similarly, if also $\frakg$ is finite dimensional, then by taking traces
in the adjoint representation we get a map $\tau = \tau_\frakg\colon\CW(T) \to
\Fun(\frakg^T\to\bbQ\llbracket\hbar\rrbracket)$. Multiplying this $\tau$
with the $\tau$ from the previous paragraph we get $\tau = \tau_\frakg
\colon M(T;H) \to \Fun(\frakg^T\to\frakg^H\llbracket\hbar\rrbracket)$.
Exponentiating, we get
\[ e^\tau\colon M(T;H) \to
  \Fun(\frakg^T\to\calU(\frakg)^{\otimes H}\llbracket\hbar\rrbracket).
\]

\subsection{$\zeta$ and BF Theory} Fix a finite dimensional Lie algebra
$\frakg$. In~\cite{CattaneoRossi:WilsonSurfaces} (see especially section
4), Cattaneo and Rossi discuss the BF quantum field theory with fields
$A\in\Omega^1(\bbR^4,\frakg)$ and $B\in\Omega^2(\bbR^4,\frakg^\ast)$
and construct an observable ``$U(A,B,\Xi)$'' for each ``long''
$\bbR^2$ in $\bbR^4$; meaning, for each 2-sphere in $S^4$ with a
prescribed behaviour at $\infty$. We interpret these as observables
defined on our ``balloons''.  The Cattaneo-Rossi observables are
functions of a variable $\Xi\in\frakg$, and they can be interpreted
as power series in a formal parameter $\hbar$. Further, given the
connection-field $A$, one may always consider its formal holonomy
along a closed path (a ``hoop'') and interpret it as an element in
$\calU(\frakg)\llbracket\hbar\rrbracket$. Multiplying these hoop
observables and also the Cattaneo-Rossi balloon observables, we get an
observable $\calO_\gamma$ for any KBH $\gamma$, taking values in
$\Fun(\frakg^T\to\calU(\frakg)^{\otimes H}\llbracket\hbar\rrbracket)$.

\begin{conjecture} \label{conj:BF} If $\gamma$ is an rKBH, then 
$\langle\calO_\gamma\rangle_{\text{BF}}=e^\tau(\zeta(\gamma))$.
\end{conjecture}

Of course, some interpretation work is required before
Conjecture~\ref{conj:BF} even becomes a well-posed mathematical statement.

We note that the Cattaneo-Rossi observable does not depend on the ribbon
property of the KBH $\gamma$. I hesitate to speculate whether this is
an indication that the work presented in this paper can be extended to
non-ribbon knots or an indication that somewhere within the rigorous
mathematical analysis of BF theory an obstruction will arise that will
force one to restrict to ribbon knots (yet I speculate that one of these
possibilities holds true).

Most likely the work of Watanabe~\cite{Watanabe:CSI} is a proof of
Conjecture~\ref{conj:BF} for the case of a single balloon and no hoops, and
very likely it contains all key ideas necessary for a complete proof of
Conjecture~\ref{conj:BF}.

\section{The Simplest Non-Commutative Reduction and an Ultimate Alexander
Invariant} \label{sec:uA}

\subsection{Informal} \label{subsec:Informal}

Let us start with some informal words. All the fundamental operations
within the definition of $M$, namely $[.,.]$, $C_u^\gamma$, $RC_u^\gamma$
and $\diver_u$, act by modifying trees and wheels near their extremities
--- their ``tails'' and their ``heads'' (for wheels, all extremities
are ``tails'').  Thus all operations will remain well-defined and will
continue to satisfy the MMA properties if we extend or reduce trees and
wheels by objects or relations that are confined to their ``inner'' parts.

In this section we discuss the ``$\beta$-quotient of $M$'',
an extension/reduction of $M$ as discussed above, which
is even better-computable than $M$. As we have seen in
Section~\ref{sec:Computations}, objects in $M$, and in particular the
invariant $\zeta$, are machine-computable. Yet the dimensions of $\FL$
and of $\CW$ grow exponentially in the degree, and so does the complexity
of computations in $M$. Objects in the $\beta$-quotient are described in
terms of commutative power series, their dimensions grow polynomially
in the degree, and computations in the $\beta$-quotient are polynomial
time. In fact, the power series appearing with the $\beta$-quotient
can be ``summed'', and {\em non-perturbative} formulae can be given to
everything in sight.

Yet $\zeta^\beta$, meaning $\zeta$ reduced to the $\beta$-quotient, remains
strong enough to contain the (multi-variable) Alexander polynomial. I argue
that in fact, the formulae obtained for the Alexander polynomial within
this $\beta$-calculus are ``better'' than many standard formulae for the
Alexander polynomial.

More on the relationship between the $\beta$-calculus and the Alexander
polynomial (though nothing about its relationship with $M$ and $\zeta$),
is in~\cite{Bar-NatanSelmani:MetaMonoids}.

\parpic[r]{\input{figs/betaRelation.pstex_t}}
Still on the informal level, the $\beta$-quotient arises by allowing a
new type of a ``sink'' vertex $c$ and imposing the $\beta$-relation, shown
on the right, on both trees and wheels. One easily sees that under this
relation, trees can be shaved to single arcs union ``$c$-stubs'', wheels
become unions of $c$-stubs, and $c$-stubs ``commute with everything'':

\[ \input{figs/betaConsequences.pstex_t} \]

Hence $c$-stubs can be taken as generators for a commutative power series
ring $R$ (with one generator $c_u$ for each possible tail label $u$),
$\CW(T)$ becomes a copy of the ring $R$, elements of $\FL(T)$ becomes
column vectors whose entries are in $R$ and whose entries correspond
to the tail label in the remaining arc of a shaved tree, and elements
of $\FL(T)^H$ can be regarded as $T\times H$ matrices with entries
in $R$. Hence in the $\beta$-quotient the MMA $M$ reduces to an MMA
$\{\beta_0(T;H)\}$ whose elements are $T\times H$ matrices of power
series, with yet an additional power series to encode the wheels
part. We will introduce $\beta_0$ more formally below, and then note that
it can be simplified even further (with no further loss of information)
to an MMA $\beta$ whose entries and operations involve rational functions,
rather than power series.

\begin{remark} The $\beta$-relation arose from studying the (unique
non-commutative) 2-dimensional
Lie algebra $\frakg_2:=\FL(\xi_1,\xi_2)/([\xi_1,\xi_2]=\xi_2)$, as in
Section~\ref{subsec:TensorialInterpretation}. Loosely, within $\frakg_2$
the $\beta$-relation is a ``polynomial identity'' in a sense similar to the
``polynomial identities'' of the theory of
PI-rings~\cite{Rowen:PolynomialIdentities}. For a more direct
relationship between this Lie algebra and the Alexander polynomial, see
\web{chic2}.
\end{remark}

\subsection{Less informal} \label{subsec:LessInformal}
For a finite set $T$ let $R=R(T):=\bbQ\llbracket\{c_u\}_{u\in
T}\rrbracket$ denote the ring of power series with commuting
generators $c_u$ corresponding to the elements $u$ of $T$, and
let $L=L(T):=R\otimes\bbQ T$ be the the free $R$-module with
generators $T$. Turn $L$ into a Lie algebra over $R$ by declaring
that $[u,v]=c_uv-c_vu$ for any $u,v\in T$. Let $c\colon L\to R$ be the
$R$-linear extension of $u\mapsto c_u$; namely,
\begin{equation} \label{eq:cgamma}
  \gamma=\sum_u\gamma_uu\in L\mapsto c_\gamma:=\sum_u\gamma_uc_u\in R,
\end{equation}
where the $\gamma_u$'s are coefficients in $R$. Note that with this
definition we have $[\alpha,\beta]=c_\alpha\beta-c_\beta\alpha$ for any
$\alpha,\beta\in L$.  There are obvious surjections $\pi\colon\FL\to L$
and $\pi\colon\CW\to R$ (strictly speaking, the first of those maps has
a small cokernel yet becomes a surjection once the ground ring of its
domain space is extended to $R$).

The following Lemma-Definition may appear scary, yet its proof is nothing
more than high school level algebra, and the messy formulae within it
mostly get renormalized away by the end of this section. Hang on!

\begin{lemmadefinition} \label{lemdef:beta0}
The operations $C_u$, $RC_u$, $\bch$, $\diver_u$,
and $J_u$ descend from $\FL/\CW$ to $L/R$, and, for 
$\alpha,\beta,\gamma\in L$ (with $\gamma=\sum_v\gamma_vv$)
they are given by
\begin{eqnarray}
  v \sslash C_u^{-\gamma} & = & v \sslash RC_u^\gamma\ =\ v
    \qquad\text{for $u\neq v\in T$}, \label{eq:betaonv} \\
  \rho \sslash C_u^{-\gamma} & = & \rho \sslash RC_u^\gamma\ =\ \rho
    \qquad\text{for $\rho\in R$}, \label{eq:betaonrho} \\
  u\sslash C_u^{-\gamma} & = &
    e^{-c_\gamma}\left(
      u+c_u\frac{e^{c_\gamma}-1}{c_\gamma}\gamma
    \right) \label{eq:betaC0} \\
    &=& e^{-c_\gamma}\left(
      \left(1+c_u\gamma_u\frac{e^{c_\gamma}-1}{c_\gamma}\right)u
      +c_u\frac{e^{c_\gamma}-1}{c_\gamma}\sum_{v\neq u}\gamma_v v
    \right), \label{eq:betaC} \\
  u\sslash RC_u^\gamma & = &
    \left(1+c_u\gamma_u\frac{e^{c_\gamma}-1}{c_\gamma}\right)^{-1}
    \left(
      e^{c_\gamma}u
      -c_u\frac{e^{c_\gamma}-1}{c_\gamma}\sum_{v\neq u}\gamma_v v
    \right), \label{eq:betaRC} \\
  \bch(\alpha,\beta) & = &
    \frac{c_\alpha+c_\beta}{e^{c_\alpha+c_\beta}-1}\left(
      \frac{e^{c_\alpha}-1}{c_\alpha}\alpha
      + e^{c_\alpha}\frac{e^{c_\beta}-1}{c_\beta}\beta
    \right), \label{eq:betabch} \\
  \diver_u\gamma & = & c_u\gamma_u, \label{eq:betadiv} \\
  J_u(\gamma) & = & \log\left(
    1+\frac{e^{c_\gamma}-1}{c_\gamma}c_u\gamma_u
  \right). \label{eq:betaJ}
\end{eqnarray}
\end{lemmadefinition}

\begin{proof} (Sketch) Equation~\eqref{eq:betaonv} is obvious ---
$C_u$ or $RC_u$ conjugate or repeatedly conjugate $u$, but not $v$.
Equation~\eqref{eq:betaonrho} is the statement that $C_u$ and $RC_u$ are
$R$-linear, namely that they act on scalars as the identity. Informally
this is the fact that 1-wheels commute with everything, and formally it
follows from the fact that $\pi\colon\FL\to L$ is a well defined morphism
of Lie algebras.

To prove Equation~\eqref{eq:betaC0}, we need to compute $e^{-\ad\gamma}(u)$,
and it is enough to carry this computation out within the 2-dimensional
subspace of $L$ spanned by $u$ and by $\gamma$. Hence the computation is an
exercise in diagonalization --- one needs to diagonalize the $2\times 2$
matrix $\ad(-\gamma)$ in order to exponentiate it. Here are
some details: set $\delta = [-\gamma,u] = c_u\gamma-c_\gamma
u$. Then clearly $\ad(-\gamma)(\delta)=-c_\gamma\delta$, and
hence $e^{-\ad\gamma}(\delta)=e^{-c_\gamma}\delta$. Also note that
$\ad(-\gamma)(\gamma)=0$, and hence $e^{-\ad\gamma}(\gamma)=\gamma$. Thus
\[ u\sslash C_u^{-\gamma}
  = e^{-\ad\gamma}(u)
  = e^{-\ad\gamma}\left(
    -\frac{\delta}{c_\gamma}+\frac{c_u\gamma}{c_\gamma}
  \right)
  = -\frac{e^{-c_\gamma}\delta}{c_\gamma}+\frac{c_u\gamma}{c_\gamma}
  = e^{-c_\gamma}\left(
      u+c_u\frac{e^{c_\gamma}-1}{c_\gamma}\gamma
    \right).
\]
Equation~\eqref{eq:betaC} is simply~\eqref{eq:betaC0} rewritten using
$\gamma=\sum_v\gamma_vv$. To prove Equation~\eqref{eq:betaRC}, take its
right hand side and use Equations~\eqref{eq:betaC} and~\eqref{eq:betaonv}
to get $u$ back again, and hence our formula for $RC_u^{\gamma}$ indeed
inverts the formula already established for $C_u^{-\gamma}$.

Equation~\eqref{eq:betabch} amounts to writing the group law of a
2-dimensional Lie group in terms of its 2-dimensional Lie algebra,
$L_0:=\Span(\alpha,\beta)$, and this is again an exercise in $2\times
2$ matrix algebra, though a slightly harder one. We work in the
adjoint representation of $L_0$ and aim to compare the exponential
of the left hand side of Equation~\eqref{eq:betabch} with the
exponential of its right hand side. If $a$ and $b$ are scalars,
let $e(a,b)$ be the matrix representing $e^{\ad(a\alpha+b\beta)}$
on $L_0$ relative to the basis $(\alpha,\beta)$.  Then using
$[\alpha,\beta]=c_\alpha\beta-c_\beta\alpha$ we find that
$e(a,b)=\exp\begin{pmatrix} bc_\beta & -ac_\beta \\ -bc_\alpha &
ac_\alpha \end{pmatrix}$, and we need to show that $e(1,0)\cdot e(0,1)
= e\left( \frac{c_\alpha+c_\beta}{e^{c_\alpha+c_\beta}-1}
\frac{e^{c_\alpha}-1}{c_\alpha},
\frac{c_\alpha+c_\beta}{e^{c_\alpha+c_\beta}-1} e^{c_\alpha}
\frac{e^{c_\beta}-1}{c_\beta} \right)$. Lazy bums do it as follows:

\dialoginclude{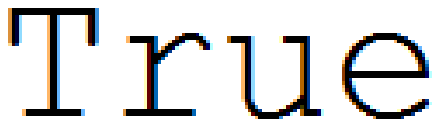}

Equation~\eqref{eq:betadiv} is the fact that $\diver_uu=c_u$, along with
the $R$-linearity of $\diver_u$.

For Equation~\eqref{eq:betaJ}, note that using
Equation~\eqref{eq:betaRC}, the coefficient of $u$ in
$\gamma\sslash RC_u^{s\gamma}$ is $\gamma_u e^{sc_\gamma}
\left(1+c_u\gamma_u\frac{e^{sc_\gamma}-1}{c_\gamma}\right)^{-1}$. Thus using
Equation~\eqref{eq:betadiv} and the fact that $C_u$ acts trivially on $R$,
\begin{multline*} J_u(\gamma)
  = \int_0^1ds\,\diver_u\!\left(\gamma \sslash RC_u^{s\gamma}\right)
    \sslash C_u^{-s\gamma}
  = \int_0^1ds\,
    \left(1+c_u\gamma_u\frac{e^{sc_\gamma}-1}{c_\gamma}\right)^{-1}
    c_u\gamma_u e^{sc_\gamma} \\
  = \left.\log\left(
      1+\frac{e^{sc_\gamma}-1}{c_\gamma}c_u\gamma_u
    \right)\right|_0^1
  = \log\left(1+\frac{e^{c_\gamma}-1}{c_\gamma}c_u\gamma_u\right).
  \qquad\qed
\end{multline*}

\end{proof}

\subsection{The reduced invariant $\zeta^{\beta_0}$.} \label{subsec:beta0}

We now let $\beta_0(T;H)$ be the $\beta$-reduced version of $M(T;H)$.
Namely, in parallel with Section~\ref{subsec:MMAM} we define
\[ \beta_0(T;H) := L(T)^H\times R^r(T) = R(T)^{T\times H}\times R^r(T). \]
In other words, elements of $\beta_0(T;H)$ are $T\times H$ matrices
$A=(A_{ux})$ of power series in the variables $\{c_u\}_{u\in T}$,
along with a single additional power series $\omega\in R^r$ ($R^r$ is $R$
modded out by its degree $1$ piece) corresponding to the
last factor above, which we write at the top left of $A$:
\[ \beta_0(u,v,\ldots;x,y,\ldots) = \left\{
  \left(\begin{array}{c|ccc}
    \omega & x & y & \cdots \\
    \hline
    u & A_{ux} & A_{uy} & \cdot \\
    v & A_{vx} & A_{vy} & \cdot \\
    \vdots & \cdot & \cdot & \ddots
  \end{array}\right)\colon \omega\in R^r(T),\,A_{\cdot\cdot}\in R(T)
  \right\}
\]

Continuing in parallel with Section~\ref{subsec:MMAM} and using the
formulae from Lemma-Definition~\ref{lemdef:beta0}, we turn
$\{\beta_0(T;H)\}$ into an MMA with operations defined as follows (on a
typical element of $\beta_0$, which is a decorated matrix $(A,\omega)$
as above):

\begin{itemize}

\item $t\sigma^u_v$ acts by renaming row $u$ to $v$ and sending the
variable $c_u$ to $c_v$ everywhere. $t\eta^u$ acts by removing row $u$
and sending $c_u$ to $0$. $tm^{uv}_w$ acts by adding row $u$ to row $v$
calling the result row $w$, and by sending $c_u$ and $c_v$ to $c_w$
everywhere.

\item $h\sigma^x_y$ and $h\eta^x$ are clear. To define $hm^{xy}_z$,
let $\alpha=(A_{ux})_{u\in T}$ and $\beta=(A_{uy})_{u\in T}$
denote the columns of $x$ and $y$ in $A$, let $c_\alpha:=\sum_{u\in
T}A_{ux}c_u$ and $c_\beta:=\sum_{u\in T}A_{uy}c_u$ in parallel with
Equation~\eqref{eq:cgamma}, and let $hm^{xy}_z$ act by removing the $x$-
and $y$-columns $\alpha$ and $\beta$ and introducing a new column, labelled
$z$, and containing $\frac{c_\alpha+c_\beta}{e^{c_\alpha+c_\beta}-1}\left(
\frac{e^{c_\alpha}-1}{c_\alpha}\alpha +
e^{c_\alpha}\frac{e^{c_\beta}-1}{c_\beta}\beta \right)$, as in
Equation~\eqref{eq:betabch}.

\parpic[r]{$\displaystyle\begin{array}{c|cccc}
  \omega & x & \cdot & y & \cdot \\
  \hline
  u & \gamma_u & \cdot & \alpha_u & \cdot \\
  \vdots & \gamma_{\text{rest}} & \cdot & \alpha_{\text{rest}} & \cdot
\end{array}$}
\item We now describe the action of $tha^{ux}$ on an input $(A,\omega)$ as
depicted on the right. Let $\gamma=\begin{pmatrix}\gamma_u \\
\gamma_{\text{rest}} \end{pmatrix}$ be the column of $x$, split into the
``row $u$'' part $\gamma_u$ and the rest, $\gamma_{\text{rest}}$. Let
$c_\gamma$ be $\sum_{v\in T}\gamma_vc_v$ as in Equation~\eqref{eq:cgamma}.
Then $tha^{ux}$ acts as follows:
\begin{itemize}
\item As dictated by Equation~\eqref{eq:betaJ}, $\omega$ is replaced by
$\omega+\log\left(1+\frac{e^{c_\gamma}-1}{c_\gamma}c_u\gamma_u\right)$.
\item As dictated by Equations~\eqref{eq:betaonv} and~\eqref{eq:betaRC},
every column $\alpha=\begin{pmatrix}\alpha_u \\ \alpha_{\text{rest}}
\end{pmatrix}$ in $A$ (including the column $\gamma$ itself) is replaced by 
\[ \left(1+c_u\gamma_u\frac{e^{c_\gamma}-1}{c_\gamma}\right)^{-1}
  \begin{pmatrix}
    e^{c_\gamma}\alpha_u \\
    \alpha_{\text{rest}}
      - c_u\frac{e^{c_\gamma}-1}{c_\gamma}(c\gamma)_{\text{rest}}
  \end{pmatrix},
\]
where $(c\gamma)_{\text{rest}}$ is the column whose row $v$ entry is
$c_v\gamma_v$, for any $v\neq u$.
\end{itemize}

\end{itemize}
\begin{itemize}

\item The ``merge'' operation $\ast$ is $\displaystyle
  \begin{array}{c|c}\omega_1&H_1\\\hline T_1&A_1\end{array}
  \,\ast\,
  \begin{array}{c|c}\omega_2&H_2\\\hline T_2&A_2\end{array}
  :=
  \begin{array}{c|cc}
    \omega_1+\omega_2 & H_1 & H_2 \\
    \hline
    T_1 & A_1 & 0 \\
    T_2 & 0 & A_2
  \end{array}$.

\item $t\epsilon_u=\begin{array}{c|c}0&\emptyset\\\hline
u&\emptyset\end{array}$ and $h\epsilon_x=\begin{array}{c|c}0&x\\\hline
\emptyset&\emptyset\end{array}$ (these values correspond to a
matrix with an empty set of columns and a matrix with an empty set of
rows, respectively).

\end{itemize}

We have concocted the definition of the MMA $\beta_0$ so that
the projection $\pi\colon M\to\beta_0$ would be a morphism of
MMAs. Hence to completely compute $\zeta^{\beta_0}:=\pi\circ\zeta$
on any rKBH (to all orders!), it is enough to note its values on the
generators. These are determined by the values in Theorem~\ref{thm:zeta}:
$\zeta^{\beta_0}(\rho^\pm_{ux})= \begin{array}{c|c}0&x\\\hline u&\pm
1\end{array}$.

\subsection{The ultimate Alexander invariant $\zeta^\beta$.}
\label{subsec:Ultimate}

Some repackaging is in order. Noting the ubiquity of factors of the form
$\frac{e^c-1}{c}$ in the previous section, it makes sense to multiply any
column $\alpha$ of the matrix $A$ by $\frac{e^{c_\alpha}-1}{c_\alpha}$.
Noting that row-$u$ entries (things like $\gamma_u$) often appear
multiplied by $c_u$, we multiply every row by its corresponding variable
$c_u$. Doing this and rewriting the formulae of the previous section in
the new variables, we find that the variables $c_u$ only appear within
exponentials of the form $e^{c_u}$. So we set $t_u:=e^{c_u}$ and rewrite
everything in terms of the $t_u$'s. Finally, the only formula that touches
$\omega$ is additive and has a $\log$ term. So we replace $\omega$ with
$e^\omega$. The result is ``$\beta$-calculus'', which was described in
detail in~\cite{Bar-NatanSelmani:MetaMonoids}. A summary version follows.
In these formulae, $\alpha$, $\beta$, $\gamma$, and $\delta$ denote
entries, rows, columns, or submatrices as appropriate, and whenever
$\alpha$ is a column, $\langle\alpha\rangle$ is the sum of is entries:

\[ \beta(T;H) =
  \left\{\left.\begin{array}{c|ccc}
    \omega & x & y & \cdots \\
    \hline
    u & \alpha_{ux} & \alpha_{uy} & \cdot \\
    v & \alpha_{vx} & \alpha_{vy} & \cdot \\
    \vdots & \cdot & \cdot & \cdot
  \end{array}\right|\parbox{2.4in}{
    $\omega$ and the $\alpha_{ux}$'s are rational functions in variables
    $t_u$, one for each $u\in T$. When all $t_u$'s are set to $1$, $\omega$
    is $1$ and every $\alpha_{ux}$ is $0$.
  }\right\},
\]
\[ tm^{uv}_w\colon
    \begin{array}{c|c}
      \omega & H \\
      \hline
      u & \alpha \\
      v & \beta \\
      T & \gamma
    \end{array}
    \mapsto
    \left(\begin{array}{c|c}
      \omega & H \\
      \hline
      w & \alpha+\beta \\
      T & \gamma
    \end{array}\right)\sslash(t_u,t_v\to t_w),
\]
\[ hm^{xy}_z\colon
    \begin{array}{c|ccc}
      \omega & x & y & H \\
      \hline
      T & \alpha & \beta & \gamma
    \end{array}
    \mapsto
    \begin{array}{c|cc}
      \omega & z & H \\
      \hline
      T & \alpha+\beta+\langle\alpha\rangle\beta & \gamma
    \end{array},
\]
\[ tha^{ux}\colon
    \begin{array}{c|cc}
      \omega & x & H \\
      \hline
      u & \alpha & \beta \\
      T & \gamma & \delta
    \end{array}
    \mapsto
    \begin{array}{c|cc}
      \omega(1+\alpha) & x & H \\
      \hline
      u & \alpha(1+\langle\gamma\rangle/(1+\alpha))
        & \beta(1+\langle\gamma\rangle/(1+\alpha)) \\
      T & \gamma/(1+\alpha) & \delta-\gamma\beta/(1+\alpha)
    \end{array},
\]
\[
  \begin{array}{c|c}\omega_1&H_1\\\hline T_1&A_1\end{array}
  \ast
  \begin{array}{c|c}\omega_2&H_2\\\hline T_2&A_2\end{array}
  :=
  \begin{array}{c|cc}
    \omega_1\omega_2 & H_1 & H_2 \\
    \hline
    T_1 & A_1 & 0 \\
    T_2 & 0 & A_2
  \end{array},
\]
\[ \zeta^\beta(t\epsilon_u)=
    \begin{array}{c|c}1&\emptyset\\\hline u&\emptyset\end{array},
  \qquad \zeta^\beta(h\epsilon_x)=
    \begin{array}{c|c}1&x\\\hline \emptyset&\emptyset\end{array},
  \qquad\text{and}\qquad
  \zeta^\beta(\rho^\pm_{ux})=
    \begin{array}{c|c}1&x\\\hline u&t_u^{\pm 1}-1\end{array}.
\]

\begin{theorem} \label{thm:Alexander} If $K$ is a u-knot regarded as a
1-component pure tangle by cutting it open, then the $\omega$ part of
$\zeta^\beta(\delta(K))$ is the Alexander polynomial of $K$.
\end{theorem}

I know of three winding paths that constitute a proof of the above theorem:
\begin{itemize}
\item Use the results of Section~\ref{sec:ft} here, of~\cite[Section~3.7]{WKO},
and of~\cite{Lee:AlexanderInvariant}.
\item Use the results of Section~\ref{sec:ft} here,
of~\cite[Section~3.9]{WKO}, and the known relation of the Alexander
polynomial with the wheels part of the Kontsevich integral
(e.g.~\cite{Kricker:Kontsevich}).
\item Use the results of~\cite{KirkLivingstonWang:Gassner}, where formulae
very similar to ours appear.
\end{itemize}
Yet to me, the strongest evidence that Theorem~\ref{thm:Alexander}
is true is that it was verified explicitly on very many knots --- see
the single example in Section~\ref{subsec:Demo1} here and many more
in~\cite{Bar-NatanSelmani:MetaMonoids}.

In several senses, $\zeta^\beta$ is an ``ultimate'' Alexander invariant:
\begin{itemize}

\item The formulae in this section may appear complicated, yet note that
if an rKBH consists of about $n$ balloons and hoops, its invariant is
described in terms of only $O(n^2)$ polynomials and each of the operations
$tm$, $hm$ and $tha$ involves only $O(n^2)$ operations on polynomials.

\item It is defined for tangles and has a prescribed behaviour under
tangle compositions (in fact, it is defined in terms of that prescribed
behaviour). This means that when $\zeta^\beta$ is computed on some large
knot with (say) $n$ crossings, the computation can be broken up into $n$
steps of complexity $O(n^2)$ at the end of each the quantity computed is
the invariant of some topological object (a tangle), or even into $3n$
steps at the end of each the quantity computed is the invariant of some
rKBH\footnote{A similar statement can be made for Alexander formulae based
on the Burau representation. Yet note that such formulae still end with
a computation of a determinant which may take $O(n^3)$ steps. Note also
that the presentation of knots as braid closures is typically inefficient
--- typically a braid with $O(n^2)$ crossings is necessary in order to
present a knot with just $n$ crossings.}.

\item $\zeta^\beta$ contains also the multivariable Alexander polynomial
and the Burau representation (overwhelmingly verified by experiment, not
written-up yet).

\item $\zeta^\beta$ has an easily prescribed behaviour under hoop- and
balloon-doubling, and $\zeta^\beta\circ\delta$ has an easily prescribed
behaviour under strand-doubling (not shown here).

\end{itemize}

\section{Odds and Ends} \label{sec:odds}

\subsection{Linking Numbers and Signs} \label{subsec:Hopf} If $x$ is
an oriented $S^1$ and $u$ is an oriented $S^2$ in an oriented $S^4$
(or $\bbR^4$) and the two are disjoint, their linking number $l_{ux}$
is defined as follows. Pick a ball $B$ whose oriented boundary is
$u$ (using the ``outward pointing normal'' convention for orienting
boundaries), and which intersects $x$ in finitely many transversal
intersection points $p_i$. At any of these intersection points $p_i$,
the concatenation of the orientation of $B$ at $p_i$ (thought of a basis
to the tangent space of $B$ at $p_i$) with the tangent to $x$ at $p_i$ is
a basis of the tangent space of $S^4$ at $p_i$, and as such it may either
be positively oriented or negatively oriented.  Define $\sigma(p_i)=+1$
in the former case and $\sigma(p_i)=-1$ in the latter case. Finally,
let $l_{ux}:=\sum_i\sigma(p_i)$. It is a standard fact that $l_{ux}$
is an isotopy invariant of $(u,x)$.

\begin{exercise} Verify that $l_{ux}(\rho^\pm_{ux})=\pm 1$, where
$\rho^+_{ux}$ and $\rho^-_{ux}$ are the positive and negative Hopf links as
in Example~\ref{exa:generators}. For the purpose of this exercise the plane
in which Figure~\ref{fig:Generators} is drawn is oriented counterclockwise,
the 3D space it represents has its third coordinate oriented ``up'' from
the plane of the paper, and $\bbR^4_{txyz}$ is oriented so that the $t$
coordinate is ``first''.
\end{exercise}

\parpic[r]{\includegraphics[width=1.5in]{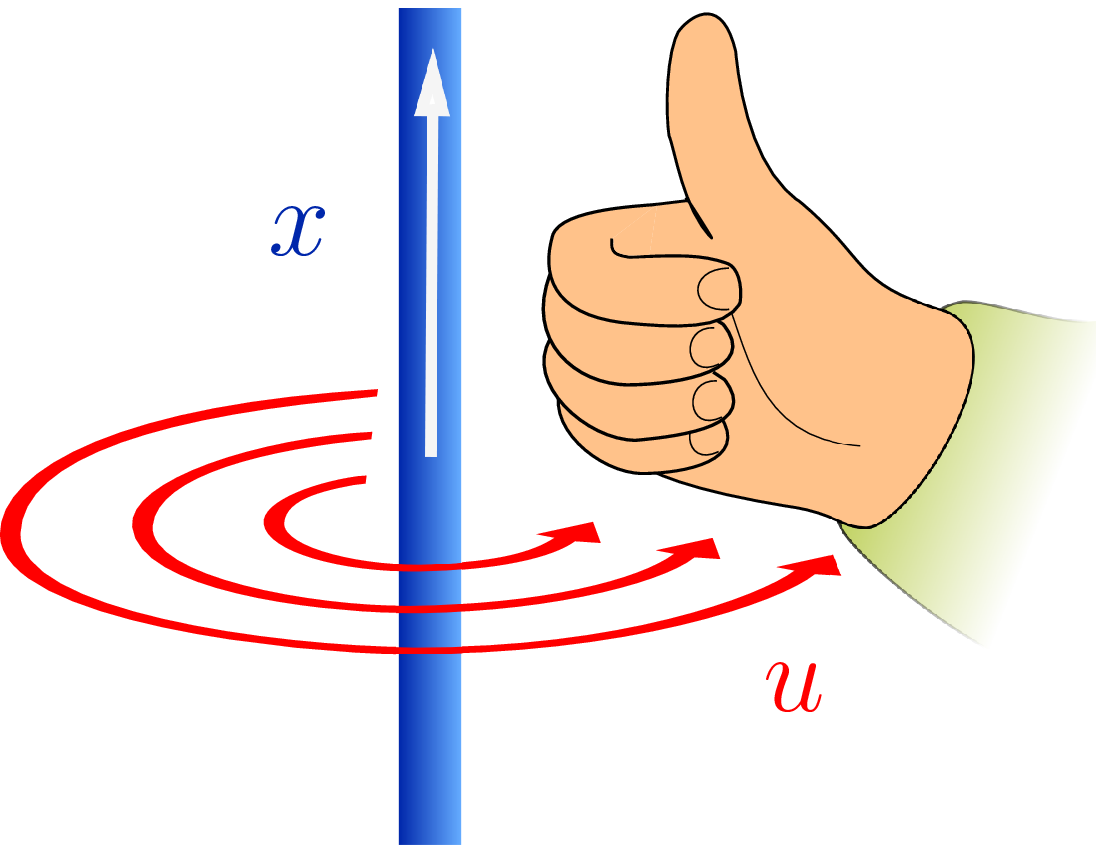}}
An efficient thumb rule for deciding the linking-number signs for
a balloon $u$ and a hoop $x$ presented using our standard notation as in
Section~\ref{subsec:kbh} is the ``right-hand rule'' of the figure on
the right, shown here without further
explanation.  The lovely figure is adopted from
[\href{http://en.wikipedia.org/wiki/Right-hand_rule}{Wikipedia:
Right-hand rule}].

\subsection{A topological construction of $\delta$}
\label{subsec:topdelta}

The map $\delta$ is a composition $\delta_0\sslash\Upsilon$ (``$\delta_0$
followed by $\Upsilon$'', aka $\Upsilon\circ\delta_0$. See ``notational
conventions'', Section~\ref{subsec:Conventions}.). Here $\delta_0$
is the standard ``tubing'' map $\delta_0$ (called ``Tube'' in
Satoh's~\cite{Satoh:RibbonTorusKnots}), though with the tubes decorated by
an additional arrowhead to retain orientation information.
The map $\Upsilon$ caps and
strings both ends of all tubes to $\infty$ and then uses, at the level
of embeddings, the fact that a pinched torus is homotopy equivalent to
a sphere wedge a circle:
\[ \input{figs/deltaZero.pstex_t}\qquad\input{figs/Upsilon.pstex_t} \]

It is worthwhile to give a completely ``topological'' definition of
the tubing map $\delta_0$, thus giving $\delta=\delta_0\sslash\Upsilon$
a topological interpretation. We must start with a topological
interpretation of v-tangles, and even before, with v-knots, also known as
virtual knots.

The standard topological interpretation of v-knots
(e.g.~\cite{Kuperberg:VirtualLink}) is that they are oriented knots
drawn\footnote{Here and below, ``drawn on $\Sigma$'' means ``embedded in
$\Sigma\times[-\epsilon,\epsilon]$''.} on an oriented surface $\Sigma$,
modulo ``stabilization'', which is the addition and/or removal of empty
handles (handles that do not intersect with the knot). We prefer an
equivalent, yet even more bare-bones approach. For us, a virtual knot is an
oriented knot $\gamma$ drawn on a ``virtual surface $\Sigma$ for
$\gamma$''. More precisely, $\Sigma$ is an oriented surface that may have
a boundary, $\gamma$ is drawn on $\Sigma$, and the pair $(\Sigma,\gamma)$
is taken modulo the following relations:
\begin{itemize}
\item Isotopies of $\gamma$ on $\Sigma$ (meaning, in
  $\Sigma\times[-\epsilon,\epsilon]$).
\item Tearing and puncturing parts of $\Sigma$ away from $\gamma$:
\end{itemize}
\[ \input{figs/TearingAndPuncturing.pstex_t} \]
(We call $\Sigma$ a ``virtual surface'' because tearing and puncturing
imply that we only care about it in the immediate vicinity of $\gamma$).

We can now define\footnote{Following a private discussion with Dylan
Thurston.} a map $\delta_0$, defined on v-knots and taking values
in ribbon tori in $\bbR^4$: given $(\Sigma,\gamma)$, embed $\Sigma$
arbitrarily in $\bbR^3_{xyz}\subset\bbR^4$. Note that the unit normal
bundle of $\Sigma$ in $\bbR^4$ is a trivial circle bundle and it has a
distinguished trivialization, constructed using its positive-$t$-direction
section and the orientation that gives each fibre a linking number
$+1$ with the base $\Sigma$.  We say that a normal vector to $\Sigma$
in $\bbR^4$ is ``near unit'' if its norm is between $1-\epsilon$ and
$1+\epsilon$. The near-unit normal bundle of $\Sigma$ has as fibre
an annulus that can be identified with $[-\epsilon,\epsilon]\times
S^1$ (identifying the radial direction $[1-\epsilon,1+\epsilon]$
with $[-\epsilon,\epsilon]$ in an orientation-preserving manner), and
hence the near-unit normal bundle of $\Sigma$ defines an embedding
of $\Sigma\times[-\epsilon,\epsilon]\times S^1$ into $\bbR^4$. On the
other hand, $\gamma$ is embedded in $\Sigma\times[-\epsilon,\epsilon]$ so
$\gamma\times S^1$ is embedded in $\Sigma\times[-\epsilon,\epsilon]\times
S^1$, and we can let $\delta_0(\Sigma,\gamma)$ be the composition
\[ \gamma\times S^1
  \hookrightarrow\Sigma\times[-\epsilon,\epsilon]\times S^1
  \hookrightarrow\bbR^4,
\]
which is a torus in $\bbR^4$, oriented using the given orientation of
$\gamma$ and the standard orientation of $S^1$.

We leave it to the reader to verify that $\delta_0(\Sigma,\gamma)$
is ribbon, that it is independent of the choices made within its
construction, that it is invariant under isotopies of $\gamma$ and under
tearing and puncturing, that it is also invariant under the ``overcrossing
commute'' relation of Figure~\ref{fig:OCUC}, and that it is equivalent
to Satoh's tubing map.

The map $\delta_0$ has straightforward generalizations to v-links,
v-tangles, framed-v-links, v-knotted-graphs, etc.

\subsection{Monoids, Meta-Monoids, Monoid-Actions and Meta-Monoid-Actions} 
\label{subsec:meta-naming}
How do we think about meta-monoid-actions? Why that name? Let us start with
ordinary monoids.

\subsubsection{Monoids} \label{subsubsec:monoids}
A monoid\footnote{A monoid is a group sans inverses. You lose nothing if
you think ``group'' whenever the discussion below states ``monoid''.}
$G$ gives rise to a slew of spaces and maps between them: the spaces would
be the spaces of sequences $G^n=\{(g_1,\dots,g_n)\colon g_i\in G\}$,
and the maps will be the maps ``that can be written using the monoid
structure'' --- they will include, for example, the map $m^{ij}_i\colon
G^n\to G^{n-1}$ defined as ``store the product $g_ig_j$ as entry number
$i$ in $G^{n-1}$ while erasing the original entries number $i$ and $j$
and re-numbering all other entries as appropriate''. In addition, there
is also an obvious binary ``concatenation'' map $\ast\colon G^n\times
G^m\to G^{n+m}$ and a special element $\epsilon\in G^1$ (the monoid unit).

Equivalently but switching from ``numbered registers'' to ``named
registers'', a monoid $G$ automatically gives rise to another slew of
spaces and operations. The spaces are $G^X=\{f\colon X\to G\}=\{(x\to
g_x)_{x\in X}\}$ of functions from a finite set $X$ to $G$, or as we
prefer to say it, of $X$-indexed sequences of elements in $G$, or how
computer scientists may say it, of associative arrays of elements of $G$
with keys in $X$. The maps between such spaces would now be the obvious
``register multiplication maps'' $m^{xy}_z\colon G^{X\cup\{x,y\}}\to
G^{X\cup\{z\}}$ (defined whenever $x,y,z\not\in X$ and $x\neq y$),
and also the obvious ``delete a register'' map $\eta^x\colon G^X\to
G^{X\remove x}$, the obvious ``rename a register'' map $\sigma^x_y\colon
G^{X\cup\{x\}}\to G^{X\cup\{y\}}$, and an obvious $\ast\colon G^X\times
G^Y\to G^{X\cup Y}$, defined whenever $X$ and $Y$ are disjoint. Also,
there are special elements, ``units'', $\epsilon_x\in G^{\{x\}}$.

This collection of spaces and maps between them (and the units) satisfies
some properties. Let us highlight and briefly discuss two of those:

\begin{enumerate}
\item The ``associativity property'': For any $\Omega\in G^X$,
\begin{equation} \label{eq:assoc} \Omega\sslash m^{xy}_x\sslash m^{xz}_x
  = \Omega\sslash m^{yz}_y\sslash m^{xy}_x.
\end{equation}
This property is an immediate consequence of the associativity axiom of
monoid theory. Note that it is a ``linear property'' --- its subject,
$\Omega$, appears just once on each side of the equality. Similar linear
properties include
$\Omega\sslash\sigma^x_y\sslash\sigma^y_z=\Omega\sslash\sigma^x_z$,
$\Omega\sslash m^{xy}_z\sslash\sigma^z_u=\Omega\sslash m^{xy}_u$, etc., and
there are also ``multi-linear'' properties like
$(\Omega_1\ast\Omega_2)\ast\Omega_3=\Omega_1\ast(\Omega_2\ast\Omega_3)$,
which are ``linear'' in each of their inputs.

\item If $\Omega\in G^{\{x,y\}}$, then 
\begin{equation} \label{eq:quadratic}
  \Omega=(\Omega\sslash\eta^y)\ast(\Omega\sslash\eta^x)
\end{equation}
(indeed, if $\Omega=(x\to g_x,\,y\to g_y)$, then
$\Omega\sslash\eta^y=(x\to g_x)$ and $\Omega\sslash\eta^x=(y\to g_y)$
and so the right hand side is $(x\to g_x)\ast(y\to g_y)$, which is
$\Omega$ back again), so an element of $G^{\{x,y\}}$ can be factored as an
element of $G^{\{x\}}$ times an element of $G^{\{y\}}$. Note that $\Omega$
appears twice in the right hand side of this property, so this property is
``quadratic''. In order to write this property one must be able to
``make two copies of $\Omega$''.
\end{enumerate}

\subsubsection{Meta-Monoids} \label{subsubsec:MetaMonoids}

\begin{definition} A meta-monoid is a collection $(G_X, m^{xy}_z, \sigma^x_z,
\eta^x, \ast)$ of sets $G_X$, one for each finite set $X$ ``of labels'',
and maps between them $m^{xy}_z$, $\sigma^x_z$, $\eta^x$, $\ast$ with the
same domains and ranges as above, and special elements $\epsilon_x\in
G_{\{x\}}$, and with the same {\bf linear and
multi-linear} properties as above.
\end{definition}

Very crucially, we do not insist on the non-linear
property~\eqref{eq:quadratic} of above, and so we may not have the
factorization $G_{\{x,y\}}=G_{\{x\}}\times G_{\{y\}}$, and in general
it need not be the case that $G_X=G^X$ for some monoid $G$. (Though of
course, the case $G_X=G^X$ is an example of a meta-monoid, which perhaps
may be called a ``classical meta-monoid'').

Thus a meta-monoid is like a monoid in that it has sets $G_X$ of
``multi-elements'' on which almost-ordinary monoid theoretic operations
are defined. Yet the multi-elements in $G_X$ need not simply be lists
of elements as in $G^X$, and instead they may be somehow ``entangled''. A
relatively simple example of a meta-monoid which isn't a monoid is
$H^{\otimes X}$ where $H$ is a Hopf algebra\footnote{Or merely an
algebra.}. This simple example is similar to ``quantum entanglement''. But
a meta-monoid is not limited to the kind of entanglement that appears in
tensor powers. Indeed many of the examples within the main text of this
paper aren't tensor powers and their ``entanglement'' is closer to that
of the theory of tangles. This especially applied to the meta-monoid $\wT$
of Section~\ref{subsec:WisMG}.

\subsubsection{Monoid-Actions} A monoid-action\footnote{Think
``group-action''.} of a monoid $G_1$ on another monoid $G_2$ is a
single algebraic structure MA consisting of two sets $G_1$ (``heads'')
and $G_2$ (``tails''), a binary operation defined on $G_1$, a binary
operation defined on $G_2$, and a mixed operation $G_1\times G_2\to G_2$
(denoted $(x,u)\mapsto u^x)$ which satisfy some well known axioms,
of which the most interesting are the associativities of the first
two binary operations and the two action axioms $(uv)^x=u^xv^x$ and
$u^{(xy)}=(u^x)^y$.

As in the case of individual monoids, a monoid-action MA gives
rise to a slew of spaces and maps between them. The spaces
are $\operatorname{MA}(T;H):=G_2^T\times G_1^H$, defined
whenever $T$ and $H$ are finite sets of ``tail labels'' and
``head labels''. The main operations\footnote{There are also
$\ast$, $t\eta^u$, $h\eta^x$, $t\sigma^u_v$, and $h\sigma^x_y$ and units
$t\epsilon_u$ and $h\epsilon_x$  as before.} are $tm^{uv}_w\colon
\operatorname{MA}(T\cup\{u,v\};H)\to\operatorname{MA}(T\cup\{w\};H)$
defined using the multiplication in $G_2$ (assuming $u,v,w\not\in T$
and $u\neq v$), $hm^{xy}_z\colon \operatorname{MA}(T;H\cup\{x,y\})
\to \operatorname{MA}(T;H\cup\{z\})$ (assuming $x,y\not\in H$
and $x\neq y$) defined using the multiplication in $G_1$, and
$tha^{ux}\colon\operatorname{MA}(T;H)\to\operatorname{MA}(T;H)$
(assuming $x\in H$ and $u\in T$) defined using the action of $G_1$ on
$G_2$. These operations have the following properties, corresponding
to the associativity of $G_1$ and $G_2$ and to the two action axioms of
the previous paragraph:
\begin{equation}\label{eq:MainMMAProperties}\begin{split}
  hm^{xy}_x\sslash hm^{xz}_x=hm^{yz}_y\sslash hm^{xy}_x,
  \quad&\quad
  tm^{uv}_u\sslash tm^{uw}_u=tm^{vw}_v\sslash tm^{uv}_u, \\
  tm^{uv}_w\sslash tha^{wx}=tha^{ux}\sslash tha^{vx}\sslash tm^{uv}_w,
  \quad&\quad
  hm^{xy}_z\sslash tha^{uz}=tha^{ux}\sslash tha^{uy}\sslash hm^{xy}_z.
\end{split}\end{equation}
There are also routine properties involving also $\ast$, $\eta$'s and
$\sigma$'s as before.

\subsubsection{Meta-Monoid-Actions} \label{sssec:MMA} Finally,
a meta-monoid-action is to a monoid-action like a meta-monoid is to
a monoid. Thus it is a collection
\[ (M(T;H), tm^{uv}_w, hm^{xy}_z, tha^{ux}, t\sigma^u_w, h\sigma^x_y,
  t\eta^u, h\eta^x, \ast, t\epsilon_u, h\epsilon_x)
\]
of sets $M(T;H)$, one for each pair of finite sets $(T;H)$ of ``tail
labels'' and ``head labels'', and maps between them $tm^{uv}_w$,
$hm^{xy}_z$, $tha^{ux}$, $t\sigma^u_v$, $h\sigma^x_y$, $t\eta^u$,
$h\eta^x$, $\ast$, and units $t\epsilon_u$ and $h\epsilon_x$, with the
same domains and ranges as above and with the same {\bf linear and
multi-linear} properties as above; most importantly, the properties
in~\eqref{eq:MainMMAProperties}.

Thus a meta-monoid-action is like a monoid-action in that it has sets
$M(T;H)$ of ``multi-elements'' on which almost-ordinary monoid theoretic
operations are defined. Yet the multi-elements in $M(T;H)$ need not
simply be lists of elements as in $G_2^T\times G_1^H$, and instead they
may be somehow ``entangled''. 

\subsubsection{Meta-Groups / Meta-Hopf-Algebras} \label{ssec:MetaHopf}
Clearly, the prefix ``meta'' can be added to many other types of
algebraic structures, though sometimes a little care must be taken. To
define a ``meta-group'', for example, one may add to the definition
of a meta-monoid in Section~\ref{subsubsec:MetaMonoids} a further
collection of operations $S^x$, one for each $x\in X$, representing
``invert the (meta-)element in register $x$''. Except that the axiom
for an inverse, $g\cdot g^{-1}=\epsilon$, is ``quadratic'' in $g$
--- one must have two copies of $g$ in order to write the axiom,
and hence it cannot be written using $S^x$ and the operations in
Section~\ref{subsubsec:MetaMonoids}. Thus in order to define a
meta-group, we need to also include ``meta-co-product'' operations
$\Delta^x_{yz}\colon G_{X\cup\{x\}}\to G_{X\cup\{y,z\}}$. These operations
should satisfy some further axioms, much like within the definition
of a Hopf algebra. The major ones are: a meta-co-associativity, a
meta-compatibility with the meta-multiplication, and a meta-inverse
axiom $\Omega \sslash \Delta^x_{yz} \sslash S^y \sslash m^{yz}_x =
(\Omega\sslash\eta^x)\ast\epsilon_x$.

A strict analogy with groups would suggest another axiom: a
meta-co-commutativity of $\Delta$, namely $\Delta^x_{yz}=\Delta^x_{zy}$.
Yet experience shows that it is better to sometimes not insist on
meta-co-commutativity. Perhaps the name ``meta-group'' should be used when
meta-co-commutativity is assumed, and ``meta-Hopf-algebra'' when it isn't.

Similarly one may extend ``meta-monoid-actions'' to ``meta-group-actions''
and/or ``meta-Hopf-actions'', in which new operations $t\Delta$ and
$h\Delta$ are introduced, with appropriate axioms.

Note that $\vT$ and $\wT$ have a meta-co-product, defined using ``strand
doubling''. It is not meta-co-commutative.

Note also that $\Kbh$ and $\Kbhz$ have operations $h\Delta$ and $t\Delta$,
defined using ``hoop doubling'' and ``balloon doubling''. The former is
meta-co-commutative while the latter is not.

Note also that $M$ and $M_0$ have have an operation $h\Delta^x_{yz}$
defined by cloning one Lie-word, and an operation $t\Delta^u_{vw}$ defined
using the substitution $u\to v+w$. Both of these operations are
meta-co-commutative.

Thus $\zeta_0$ and $\zeta$ cannot be homomorphic with respect to $t\Delta$.
The discussion of trivalent vertices in~\cite[Section~6]{WKO} can be
interpreted as an analysis of the failure of $\zeta$ to be homomorphic with
respect to $t\Delta$, but this will not be attempted in this paper.

\subsection{Some Differentials and the Proof of
  Proposition~\ref{prop:JProperties}} \label{subsec:JProperties}

We prove Proposition~\ref{prop:JProperties}, namely
Equations~\eqref{eq:JhProperty} through~\eqref{eq:JtProperty}, by
verifying that each of these equations holds at one point, and then
by differentiating each side of each equation and showing that the
derivatives are equal. While routine, this argument appears complicated
because the spaces involved are infinite dimensional and the operations
involved are non-commutative. In fact, even the well-known derivative
of the exponential function, which appears in the definition of $C_u$
which appears in the definitions of $RC_u$ and of $J_u$, may surprise
readers who are used to the commutative case $de^x=e^xdx$.

Recall that $\FA$ denotes the graded completion of the free
associative algebra on some alphabet $T$, and that the exponential
map $\exp\colon\FL\to\FA$ defined by $\gamma\mapsto\exp(\gamma) =
e^\gamma:=\sum_{k=0}^\infty\frac{\gamma^k}{k!}$ makes sense in this
completion.

\begin{lemma} \label{lem:dexp}
If $\delta\gamma$ denotes an infinitesimal variation of $\gamma$, then
the infinitesimal variation $\delta e^\gamma$ of $e^\gamma$ is given
as follows:
\begin{equation} \label{eq:dexp}
 \delta e^\gamma= e^\gamma\cdot\left(
    \delta\gamma\sslash\frac{1-e^{-\ad\gamma}}{\ad\gamma}
  \right)
  = \left(\delta\gamma\sslash\frac{e^{\ad\gamma}-1}{\ad\gamma}\right)
    \cdot e^\gamma.
\end{equation}
\end{lemma}

Above expressions such as $\frac{e^{\ad\gamma}-1}{\ad\gamma}$
are interpreted via their power series
expansions, $\frac{e^{\ad\gamma}-1}{\ad\gamma} = 1 +
\frac12\ad\gamma + \frac16(\ad\gamma)^2 + \dots$, and hence
$\delta\gamma\sslash\frac{e^{\ad\gamma}-1}{\ad\gamma} = \delta\gamma +
\frac12[\gamma,\delta\gamma] + \frac16[\gamma,[\gamma,\delta\gamma]]
+ \dots$. Also, the precise meaning of~\eqref{eq:dexp}
is that for any $\delta\gamma\in\FL$, the derivative
$\delta e^\gamma := \lim_{\epsilon\to 0} \frac{1}{\epsilon}
\left(e^{\gamma+\epsilon\delta\gamma}-e^\gamma\right)$ is given by the
right-hand-side of that equation. Equivalently, $\delta e^\gamma$ is
the term proportional to $\delta\gamma$ in $e^{\gamma+\delta\gamma}$,
where during calculations we may assume that ``$\delta\gamma$ is
an infinitesimal'', meaning that anything quadratic or higher in
$\delta\gamma$ can be regarded as equal to $0$.

Lemma~\ref{lem:dexp} is rather standard
(e.g.~\cite[Section~1.5]{DuistermaatKolk:LieGroups},
\cite[Section~7]{Meinrenken:LieNotes}). Here's a tweet:

\vskip 3mm

\noindent{\em Proof of Lemma~\ref{lem:dexp}.} With
an infinitesimal $\delta\gamma$, consider $F(s) :=
e^{-s\gamma}e^{s(\gamma+\delta\gamma)} - 1$. Then $F(0)=0$ and
$\frac{d}{ds}F(s) = e^{-s\gamma}(-\gamma)e^{s(\gamma+\delta\gamma)}
+ e^{-s\gamma}(\gamma+\delta\gamma)e^{s(\gamma+\delta\gamma)}
= e^{-s\gamma}\delta\gamma e^{s(\gamma+\delta\gamma)} =
e^{-s\gamma}\delta\gamma e^{s\gamma} = \delta\gamma\sslash
e^{-s\ad\gamma}$. So $e^{-\gamma}\delta\gamma = F(1) = \int_0^1
ds\frac{d}{ds}F(s) = \delta\gamma\sslash \int_0^1 ds\, e^{-s\ad\gamma} =
\delta\gamma\sslash\frac{1-e^{-\ad\gamma}}{\ad\gamma}$. The second part
of~\eqref{eq:dexp} is proven in a similar manner, starting with $G(s)
:= e^{s(\gamma+\delta\gamma)}e^{-s\gamma} - 1$. \qed

\begin{lemma} \label{lem:dbch}
If $\gamma=\bch(\alpha,\beta)$ and $\delta\alpha$, $\delta\beta$, and
$\delta\gamma$ are infinitesimals related by
$\gamma+\delta\gamma=\bch(\alpha+\delta\alpha, \beta+\delta\beta)$, then
\begin{equation} \label{eq:dbch}
  \delta\gamma\sslash\frac{1-e^{-\ad\gamma}}{\ad\gamma}
  = \left(\delta\alpha\sslash\frac{1-e^{-\ad\alpha}}{\ad\alpha}\sslash
    e^{-\ad\beta}\right)
  + \left(\delta\beta\sslash\frac{1-e^{-\ad\beta}}{\ad\beta}\right)
\end{equation}
\end{lemma}

\begin{proof} Use Leibniz' law on $e^\gamma = e^\alpha e^\beta$ to get
$\delta e^\gamma = (\delta e^\alpha)e^\beta + e^\alpha(\delta e^\beta)$.
Now use Lemma~\ref{lem:dexp} three times to get
\[
  e^\gamma\left(\gamma\sslash\frac{1-e^{-\ad\gamma}}{\ad\gamma}\right)
  = e^\alpha\left(
      \delta\alpha\sslash\frac{1-e^{-\ad\alpha}}{\ad\alpha}
    \right)e^\beta
  + e^\alpha e^\beta\left(
      \delta\beta\sslash\frac{1-e^{-\ad\beta}}{\ad\beta}
    \right),
\]
conjugate the $e^\beta$ in the first summand to the other side of the
parenthesis, and cancel $e^\gamma = e^\alpha e^\beta$ from both sides of the
resulting equation. \qed
\end{proof}

Recall that $C_u^\gamma$ and $RC_u^\gamma$ are automorphisms of
$\FL$. We wish to study their variations $\delta C_u^\gamma$ and
$\delta RC_u^\gamma$ with respect to $\gamma$ (these variations are
``infinitesimal'' automorphisms of $\FL$). We need a definition and a
property first.

\begin{definition} \label{def:adu} For $u\in T$ and $\gamma\in\FL(T)$ let
$\ad_u\{\gamma\}=\ad_u^\gamma\colon\FL(T)\to\FL(T)$ denote the derivation
of $\FL(T)$ defined by its action of the generators as follows:
\[ v\sslash\ad_u\{\gamma\} = v\sslash\ad_u^\gamma
  := \begin{cases}
    [\gamma,u] & v=u \\
    0 & \text{otherwise.}
  \end{cases}
\]
\end{definition}

\begin{property} \label{ppty:adu} $\ad_u$ is the infinitesimal version
of both $C_u$ and $RC_u$. Namely, if $\delta\gamma$ is an infinitesimal,
then $C_u^{\delta\gamma} = RC_u^{\delta\gamma} = 1+\ad_u\{\delta\gamma\}$.
\end{property}

We omit the easy proof of this property and move on to $\delta
C_u^\gamma$ and $\delta RC_u^\gamma$:

\begin{lemma} \label{lem:dC}
\hfill $\displaystyle \delta C_u^\gamma
  = \ad_u\left\{
    \delta\gamma\sslash\frac{e^{\ad \gamma}-1}{\ad\gamma}\sslash RC_u^{-\gamma}
  \right\} \sslash C_u^\gamma
$\hfill\null
\newline and
\hfill$\displaystyle \delta RC_u^\gamma
  = RC_u^\gamma \sslash\ad_u\left\{
    \delta\gamma\sslash\frac{1-e^{-\ad\gamma}}{\ad\gamma}\sslash RC_u^\gamma
  \right\}
$.\hfill\null
\end{lemma}

\begin{proof} Substitute $\alpha$ and $\delta\beta$ into
Equation~\eqref{eq:RCh} and get $RC_u^{\bch(\alpha,\delta\beta)} =
RC_u^\alpha \sslash RC_u^{\delta\beta\sslash RC_u^\alpha}$, and hence using
Property~\ref{ppty:adu} for the infinitesimal $\delta\beta\sslash
RC_u^\alpha$ and Lemma~\ref{lem:dbch} with $\delta\alpha=\beta=0$ on
$\bch(\alpha,\delta\beta)$,
\[ RC_u^{\alpha+(\delta\beta\sslash\frac{\ad\alpha}{1-e^{-\ad\alpha}})}
  = RC_u^\alpha
    + RC_u^\alpha\sslash\ad_u\{\delta\beta\sslash RC_u^\alpha\}.
\]
Now replacing $\alpha\to\gamma$ and
$\delta\beta\to\delta\gamma\sslash\frac{1-e^{-\ad\gamma}}{\ad\gamma}$, we
get the equation for $\delta RC_u^\gamma$. The equation for $\delta
C_U^\gamma$ now follows by taking the variation of $C_u^\gamma\sslash
RC_u^{-\gamma}=Id$. \qed
\end{proof}

Our next task is to compute $\delta J_u(\gamma)$. Yet before we can
do that, we need to know one of the two properties of $\diver_u$ that
matter for us (besides its linearity):

\begin{proposition} For any $u,v\in T$ and any $\alpha,\beta\in\FL$ and
with $\delta_{uv}$ denoting the Kronecker delta function, the following
``cocycle condition'' holds:
(compare with~\cite[Proposition~3.20]{AlekseevTorossian:KashiwaraVergne})
\begin{equation} \label{eq:divproph}
  \underbrace{(\diver_u\alpha)\sslash\ad_v^\beta}_A
  - \underbrace{(\diver_v\beta)\sslash\ad_u^\alpha}_B
  = \underbrace{\delta_{uv}\diver_u [\alpha,\beta]}_C
  + \underbrace{\diver_u(\alpha\sslash\ad_v^\beta)}_D
  - \underbrace{\diver_v(\beta\sslash\ad_u^\alpha)}_E.
\end{equation}
\end{proposition}

\parpic[r]{\begin{picture}(0,0)%
\includegraphics{figs/gates.pstex}%
\end{picture}%
%
%  pstex_opts: -m 1.0 
%
\setlength{\unitlength}{3947sp}%
\begingroup\makeatletter\ifx\SetFigFont\undefined%
\gdef\SetFigFont#1#2#3#4#5{%
  \reset@font\fontsize{#1}{#2pt}%
  \fontfamily{#3}\fontseries{#4}\fontshape{#5}%
  \selectfont}%
\fi\endgroup%
\begin{picture}(1524,1224)(-11,-523)
\put(1201, 89){\makebox(0,0)[b]{\smash{{\SetFigFont{12}{14.4}{\rmdefault}{\mddefault}{\updefault}{\color[rgb]{0,0,0}$\beta$}%
}}}}
\put(301, 89){\makebox(0,0)[b]{\smash{{\SetFigFont{12}{14.4}{\rmdefault}{\mddefault}{\updefault}{\color[rgb]{0,0,0}$\alpha$}%
}}}}
\end{picture}%
}
\begin{proof} Start with the case where $u=v$. We draw each contribution
to each of the terms above and note that all of these contributions
cancel, but we must first explain our drawing conventions. We draw
$\alpha$ and $\beta$ as the ``logic gates'' appearing on the right. Each
is really a linear combination, but~\eqref{eq:divproph} is bilinear so
this doesn't matter. Each is really a tree, but the proof does not use
this so we don't display this. Each may have many tail-legs labelled by
other elements of $T$, but we care only about the legs labelled $u=v$
and so we display only those, and without real loss of generality,
we draw it as if $\alpha$ and $\beta$ each have exactly 3 such tails.

\parpic[r]{\input{figs/gateops.pstex_t}}
Objects such as $\diver_u\alpha$ and $\alpha\sslash\ad_u^\beta$ are
obtained from $\alpha$ and $\beta$ by connecting the head of one near its
own tails, or near the other's tails, in all possible ways. We draw just
one summand from each sum, yet we indicate the other possible summands
in each case by marking the other places where the relevant head could
go with filled circles ($\bullet$) or empty circles ($\circ$) (the
filling of the circles has no algebraic meaning; it is there only to
separate summations in cases where two summations appear in the same
formula). I hope the pictures on the right explain this better than
the words.

\parpic[r]{\input{figs/TermA.pstex_t}}
We illustrate our next convention with the pictorial
representation of term $A$ of Equation~\eqref{eq:divproph},
$(\diver_u\alpha)\sslash\ad_u^\beta$, shown on the right. Namely, when the
two relevant summations dictate that two heads may fall on the same arc,
we split the sum into the generic part, $A_1$ on the right, in which the
two heads do not fall on the same arc, and the exceptional part, $A_2$
on the right, in which the two heads do indeed fall on the same arc. The
last convention is that $\bullet$ indicates the first summation, and
$\circ$ the second. Hence in $A_1$, the $\alpha$ head may fall in $3$
places, and after that, the $\beta$ head may only fall on one of the
remaining relevant tails, whereas in $A_1$ the $\alpha$ is again free, but
the $\beta$ head must fall on the same arc.

With all these conventions in place and with term $A$ as above, we depict
terms $B$-$E$:
\[ \input{figs/TermsBD.pstex_t} \]
\[ \input{figs/TermsCE.pstex_t} \]
Clearly, $A_1=D_1$, $B_1=E_1$, and $D_3=E_3$ (the last equality is the only
place in this paper that we need the cyclic property of cyclic words).
Also, by the Jacobi identity, $A_2-D_2=C_1$ and $E_2-B_2=C_2$. So
altogether, $A-B=C+D-E$.

The case where $u\neq v$ is similar, except we have to separate between
$u$ and $v$ tails, the terms analogous to $A_2$, $B_2$, $D_2$ and $E_2$
cannot occur, and $C=0$:
\[ \input{figs/TermsABDE.pstex_t} \]
Clearly, $A-B=D-E$. \qed

\end{proof}

For completeness and for use within the proof of
Equation~\eqref{eq:JtProperty}, here's the remaining property of $\diver$ we
need to know, presented without its easy proof:

\begin{proposition} \label{prop:divpropt} For any $\gamma\in\FL$,
\hfill$\displaystyle
  \gamma\sslash t^{uv}_w \sslash \diver_w
  = \gamma\sslash\diver_u\sslash t^{uv}_w
  + \gamma\sslash\diver_v\sslash t^{uv}_w.
$\qed
\end{proposition}

\begin{proposition} \label{prop:dJ}
$\displaystyle \delta J_u(\gamma)
  = \delta\gamma \sslash \frac{1-e^{-\ad\gamma}}{\ad\gamma}
    \sslash RC_u^{\gamma} \sslash \diver_u \sslash C_u^{-\gamma}
$.
\end{proposition}

\begin{proof} Let $I_s:=\gamma \sslash RC_u^{s\gamma} \sslash \diver_u
\sslash C_u^{-s\gamma}$ denote the integrand in the definition of
$J_u$. Then under $\gamma\to\gamma+\delta\gamma$, using 
Leibniz, the linearity of $\diver_u$, and both parts of
Lemma~\ref{lem:dC}, we have
\[ \begin{split} \delta I_s
  = \delta\gamma \sslash RC_u^{s\gamma} \sslash \diver_u \sslash
      C_u^{-s\gamma}
    &+ \gamma \sslash RC_u^{s\gamma} \sslash \ad_u\left\{
      \delta\gamma \sslash \frac{1-e^{-\ad s\gamma}}{\ad\gamma}
        \sslash RC_u^{s\gamma}
    \right\}\sslash \diver_u \sslash C_u^{-s\gamma} \\
    &- \gamma \sslash RC_u^{s\gamma} \sslash \diver_u
      \sslash \ad_u\left\{
        \delta\gamma \sslash \frac{1-e^{-\ad s\gamma}}{\ad\gamma}\sslash
        RC_u^{s\gamma}
        \right\} \sslash C_u^{-s\gamma}.
\end{split} \]
Taking the last two terms above as $D$ and $A$ of
Equation~\eqref{eq:divproph}, with $\alpha=\gamma \sslash
RC_u^{s\gamma}$ and $\beta=\delta\gamma \sslash \frac{1-e^{-\ad
s\gamma}}{\ad\gamma}\sslash RC_u^{s\gamma}$, and using
$[\alpha,\beta]=[\gamma, \delta\gamma \sslash \frac{1-e^{-\ad
s\gamma}}{\ad\gamma}] \sslash RC_u^{s\gamma} = \delta\gamma \sslash
(1-e^{-\ad s\gamma}) \sslash RC_u^{s\gamma}$, we get
\[ \begin{split} \delta I_s
  = \delta\gamma \sslash RC_u^{s\gamma} \sslash \diver_u \sslash
    C_u^{-s\gamma}
  &+ \delta\gamma \sslash \frac{1-e^{-\ad s\gamma}}{\ad\gamma}
    \sslash RC_u^{s\gamma} \sslash \ad_u\{\gamma \sslash RC_u^{s\gamma}\}
    \sslash \diver_u \sslash C_u^{-s\gamma} \\
  &- \delta\gamma \sslash \frac{1-e^{-\ad s\gamma}}{\ad\gamma}
    \sslash RC_u^{s\gamma} \sslash \diver_u
    \sslash \ad_u\{\gamma \sslash RC_u^{s\gamma}\}
    \sslash C_u^{-s\gamma} \\
  &- \delta\gamma \sslash (1-e^{-\ad s\gamma}) \sslash RC_u^{s\gamma}
    \sslash \diver_u \sslash C_u^{-s\gamma},
\end{split} \]
and so, by combining the first and the last terms above,
\[ \begin{split} \delta I_s
  =\, & \delta\gamma \sslash e^{-\ad s\gamma}
    \sslash RC_u^{s\gamma} \sslash \diver_u \sslash
    C_u^{-s\gamma} \\
  &+ \delta\gamma \sslash \frac{1-e^{-\ad s\gamma}}{\ad\gamma}
    \sslash RC_u^{s\gamma} \sslash \ad_u\{\gamma \sslash RC_u^{s\gamma}\}
    \sslash \diver_u \sslash C_u^{-s\gamma} \\
  &- \delta\gamma \sslash \frac{1-e^{-\ad s\gamma}}{\ad\gamma}
    \sslash RC_u^{s\gamma} \sslash \diver_u 
    \sslash \ad_u\{\gamma \sslash RC_u^{s\gamma}\}
    \sslash C_u^{-s\gamma},
\end{split} \]
and hence, once again using Lemma~\ref{lem:dC} to differentiate
$RC_u^{s\gamma}$ and $C_u^{-s\gamma}$ (except that things are now simpler
because $s\gamma$ and $\delta(s\gamma) = \frac{d}{ds}(s\gamma) = \gamma$
commute), we get
\[ \delta I_s = \frac{d}{ds}\left(
  \delta\gamma \sslash \frac{1-e^{-\ad s\gamma}}{\ad\gamma}
  \sslash RC_u^{s\gamma} \sslash \diver_u \sslash C_u^{-s\gamma} \right).
\]
Integrating with respect to the variable $s$ and using the fundamental
theorem of calculus, we are done. \qed
\end{proof}

\noindent{\em Proof of Equation~\eqref{eq:JhProperty}.} We fix $\alpha$
and show that Equation~\eqref{eq:JhProperty} holds for every $\beta$. For
this it is enough to show that Equation~\eqref{eq:JhProperty} holds
for $\beta=0$ (it trivially does), and that the derivatives of both
sides of Equation~\eqref{eq:JhProperty} in the radial direction are
equal, for any given $\beta$. Namely, it is enough to verify that the
variations of the two sides of Equation~\eqref{eq:JhProperty} under
$\beta\to\beta+\delta\beta$ are equal, where $\delta\beta$ is proportional
to $\beta$. Indeed, using the chain rule, Lemma~\ref{lem:dbch},
Proposition~\ref{prop:dJ}, the fact that $\beta$ commutes with
$\delta\beta$, and with $\gamma:=\bch(\alpha,\beta)$,
\[ \begin{split} \delta LHS
  &= \left(\delta\beta\sslash\frac{1-e^{-\ad\beta}}{\ad\beta}
    \sslash \frac{\ad\gamma}{1-e^{-\ad\gamma}}\right)
    \sslash \frac{1-e^{-\ad\gamma}}{\ad\gamma}
    \sslash RC_u^{\gamma} \sslash \diver_u \sslash C_u^{-\gamma} \\
  &= \delta\beta
    \sslash RC_u^{\gamma} \sslash \diver_u \sslash C_u^{-\gamma}.
\end{split} \]
Similarly, using Proposition~\ref{prop:dJ} and the fact that $\beta\sslash
RC_u^\alpha$ commutes with $\delta\beta\sslash RC_u^\alpha$,
\[ \delta RHS
  = \delta\beta\sslash RC_u^\alpha \sslash RC_u^{\beta\sslash RC_u^\alpha}
    \sslash \diver_u \sslash C_u^{-\beta\sslash RC_u^\alpha}
    \sslash C_u^{-\alpha}
  = \delta\beta
    \sslash RC_u^{\gamma} \sslash \diver_u \sslash C_u^{-\gamma},
\]
where in the last equality we have used Equation~\eqref{eq:RCh} to combine
the $RC$'s and its inverse to combine the $C$'s. \qed

\noindent{\em Proof of Equation~\eqref{eq:JuvProperty}.}
Equation~\eqref{eq:JuvProperty} clearly holds when $\alpha=0$, so as
before, it is enough to prove it after taking the radial derivative with
respect to $\alpha$. So we need (ouch!)
\begin{multline*}
  \alpha \sslash RC_u^\alpha \sslash \diver_u \sslash C_u^{-\alpha}
  - \alpha \sslash RC_v^\beta \sslash RC_u^{\alpha\sslash RC_v^\beta}
    \sslash \diver_u \sslash C_u^{-\alpha\sslash RC_v^\beta}
    \sslash C_v^{-\beta} \\
  = - \beta \sslash RC_u^\alpha \sslash \ad_u^{\alpha\sslash RC_u^\alpha}
    \sslash
    \frac{1-e^{-\ad(\beta\sslash RC_u^\alpha)}}{\ad(\beta\sslash RC_u^\alpha)}
    \sslash RC_v^{\beta\sslash RC_u^\alpha} \sslash \diver_v \sslash
    C_v^{-\beta\sslash RC_u^\alpha} \sslash C_u^{-\alpha} \\
  - \beta \sslash RC_u^\alpha \sslash J_v \sslash
    \ad_u^{-\alpha\sslash RC_u^\alpha} \sslash C_u^{-\alpha}.
\end{multline*}
This we simplify using~\eqref{eq:RCuRCv} and~\eqref{eq:CuCv}, cancel the
$C_u^{-\alpha}$ on the right, and get
\begin{multline*}
  \alpha \sslash RC_u^\alpha \sslash \diver_u
  - \alpha \sslash RC_u^\alpha \sslash RC_v^{\beta\sslash RC_u^\alpha}
    \sslash \diver_u \sslash C_v^{-\beta\sslash RC_u^\alpha} \\
  \overset{?}{=}
  - \beta \sslash RC_u^\alpha \sslash \ad_u^{\alpha\sslash RC_u^\alpha}
    \sslash 
    \frac{1-e^{-\ad(\beta\sslash RC_u^\alpha)}}{\ad(\beta\sslash RC_u^\alpha)}
    \sslash RC_v^{\beta\sslash RC_u^\alpha} \sslash \diver_v \sslash
    C_v^{-\beta\sslash RC_u^\alpha} \\
  - \beta \sslash RC_u^\alpha \sslash J_v \sslash
    \ad_u^{-\alpha\sslash RC_u^\alpha}.
\end{multline*}
We note that above $\alpha$ and $\beta$ only appear within the combinations
$\alpha\sslash RC_u^\alpha$ and $\beta\sslash RC_u^\alpha$, so we rename
$\alpha\sslash RC_u^\alpha\to\alpha$ and $\beta\sslash RC_u^\alpha\to\beta$:
\begin{multline} \label{eq:Jad}
  \alpha \sslash \diver_u
  - \alpha \sslash RC_v^\beta \sslash \diver_u \sslash C_v^{-\beta} \\
  \overset{?}{=}
  - \beta \sslash \ad_u^\alpha \sslash \frac{1-e^{-\ad(\beta)}}{\ad(\beta)}
    \sslash RC_v^\beta \sslash \diver_v \sslash C_v^{-\beta}
  - \beta \sslash J_v \sslash \ad_u^{-\alpha}.
\end{multline}

Equation~\eqref{eq:Jad} still contains a $J_v$ in it,
so in order to prove it, we have to differentiate once
again. So note that it holds at $\beta=0$, multiply by $-1$,
and take the radial variation with respect to $\beta$ (note that
$\frac{d}{ds}\left.\frac{1-e^{-\ad(s\beta)}}{\ad(s\beta)}\right|_{s=1}
= \frac{e^{-\ad(\beta)}(1+\ad(\beta)-e^{\ad(\beta)})}{\ad(\beta)}$):

\begin{equation} \label{eq:monster} \begin{split}
  \alpha \sslash RC_v^\beta
  & \sslash \ad_v^{\beta\sslash RC_v^\beta}
    \sslash \diver_u \sslash C_v^{-\beta}
  - \alpha \sslash RC_v^\beta \sslash \diver_u \sslash
    \ad_v^{\beta\sslash RC_v^\beta} \sslash C_v^{-\beta} \\
  \overset{?}{=} & \ 
  \beta \sslash \ad_u^\alpha \sslash \frac{1-e^{-\ad(\beta)}}{\ad(\beta)}
    \sslash RC_v^\beta \sslash \diver_v \sslash C_v^{-\beta} \\
  & + \beta \sslash \ad_u^\alpha
    \sslash \frac{e^{-\ad(\beta)}(1+\ad(\beta)-e^{\ad(\beta)})}{\ad(\beta)}
    \sslash RC_v^\beta \sslash \diver_v \sslash C_v^{-\beta} \\
  & + \beta \sslash \ad_u^\alpha \sslash \frac{1-e^{-\ad(\beta)}}{\ad(\beta)}
    \sslash RC_v^\beta \sslash \ad_v^{\beta\sslash RC_v^\beta}
    \sslash \diver_v \sslash C_v^{-\beta} \\
  & + \beta \sslash \ad_u^\alpha \sslash \frac{1-e^{-\ad(\beta)}}{\ad(\beta)}
    \sslash RC_v^\beta \sslash \diver_v
    \sslash \ad_v^{-\beta\sslash RC_v^\beta} \sslash C_v^{-\beta} \\
  & + \beta \sslash RC_v^\beta \sslash \diver_v \sslash CC_v^{-\beta}
    \sslash \ad_u^{-\alpha}
\end{split} \end{equation}
We massage three independent parts of the above desired equality at the
same time:
\begin{itemize}
\item The $\diver$ and the $\ad$ on the left hand
side make terms $D$ and $A$ of Equation~\eqref{eq:divproph},
with $\alpha \sslash RC_v^\beta\to\alpha$ and $\beta\sslash
RC_v^\beta\to\beta$. We replace them by terms $A$ and $E$.
\item We combine the first two terms of the right hand side using
$\frac{1-e^{-a}}{a}+\frac{e^{-a}(1+a-e^a)}{a}=e^{-a}$.
\item In Equation~\eqref{eq:CuCv},
$C_u^{-\alpha\sslash RC_v^\beta}\sslash C_v^{-\beta}
  = C_v^{-\beta\sslash RC_u^\alpha}\sslash C_u^{-\alpha}$,
take an infinitesimal $\alpha$ and use Property~\ref{ppty:adu} and
Lemma~\ref{lem:dC} to get
\begin{equation} \label{eq:InfinitesimalCuCv}
  \ad_u^{-\alpha\sslash RC_v^\beta}\sslash C_v^{-\beta}
  = \ad_v^{-\beta\sslash \ad_u^\alpha \sslash
    \frac{1-e^{-\ad(\beta)}}{\ad(\beta)} \sslash RC_v^\beta} \sslash
    C_v^{-\beta}
  + C_v^{-\beta}\sslash \ad_u^{-\alpha}.
\end{equation}
The last of that matches the last of~\eqref{eq:monster}, so we can
replace the last of~\eqref{eq:monster} with the start
of~\eqref{eq:InfinitesimalCuCv}.
\end{itemize}
All of this done, Equation~\eqref{eq:monster} becomes the lowest point of
this paper:
\[ \begin{split}
  \beta \sslash RC_v^\beta
  & \sslash \ad_u^{\alpha\sslash RC_v^\beta}
    \sslash \diver_v \sslash C_v^{-\beta}
  - \beta \sslash RC_v^\beta \sslash \diver_v \sslash
    \ad_u^{\alpha\sslash RC_v^\beta} \sslash C_v^{-\beta} \\
  \overset{?}{=} & \ 
  \beta \sslash \ad_u^\alpha \sslash e^{-\ad(\beta)}
    \sslash RC_v^\beta \sslash \diver_v \sslash C_v^{-\beta} \\
  & + \beta \sslash \ad_u^\alpha \sslash \frac{1-e^{-\ad(\beta)}}{\ad(\beta)}
    \sslash RC_v^\beta \sslash \ad_v^{\beta\sslash RC_v^\beta}
    \sslash \diver_v \sslash C_v^{-\beta} \\
  & + \beta \sslash \ad_u^\alpha \sslash
\frac{1-e^{-\ad(\beta)}}{\ad(\beta)}
    \sslash RC_v^\beta \sslash \diver_v
    \sslash \ad_v^{-\beta\sslash RC_v^\beta} \sslash C_v^{-\beta} \\
  & + \beta \sslash RC_v^\beta \sslash \diver_v \sslash
    \ad_u^{-\alpha\sslash RC_v^\beta}\sslash C_v^{-\beta} \\
  & - \beta \sslash RC_v^\beta \sslash \diver_v \sslash
    \ad_v^{-\beta\sslash \ad_u^\alpha \sslash
    \frac{1-e^{-\ad(\beta)}}{\ad(\beta)} \sslash RC_v^\beta} \sslash
    C_v^{-\beta}
\end{split} \]
Next we cancel the $C_v^{-\beta}$ at the right of every term, and a pair
of repeating terms to get
\begin{multline*}
  \beta \sslash RC_v^\beta \sslash \ad_u^{\alpha\sslash RC_v^\beta}
    \sslash \diver_v
  \overset{?}{=}
  \beta \sslash \ad_u^\alpha \sslash e^{-\ad(\beta)}
    \sslash RC_v^\beta \sslash \diver_v \\
  + \beta \sslash \ad_u^\alpha \sslash \frac{1-e^{-\ad(\beta)}}{\ad(\beta)}
    \sslash RC_v^\beta \sslash \ad_v^{\beta\sslash RC_v^\beta}
    \sslash \diver_v \\
  - \beta \sslash \ad_u^\alpha \sslash \frac{1-e^{-\ad(\beta)}}{\ad(\beta)}
    \sslash RC_v^\beta \sslash \diver_v
    \sslash \ad_v^{\beta\sslash RC_v^\beta} \\
  - \beta \sslash RC_v^\beta \sslash \diver_v \sslash 
    \ad_v^{-\beta\sslash \ad_u^\alpha \sslash
    \frac{1-e^{-\ad(\beta)}}{\ad(\beta)} \sslash RC_v^\beta}
\end{multline*}

The two middle terms above differ only in the order of $\ad_v$ and
$\diver_v$. So we apply Equation~\eqref{eq:divproph} again and get
\begin{multline*}
  \beta \sslash RC_v^\beta \sslash \ad_u^{\alpha\sslash RC_v^\beta}
    \sslash \diver_v
  \overset{?}{=}
  \beta \sslash \ad_u^\alpha \sslash e^{-\ad(\beta)}
    \sslash RC_v^\beta \sslash \diver_v \\
  + \beta \sslash RC_v^\beta \sslash \ad_v^{\beta\sslash\ad_u^\alpha
      \sslash\frac{1-e^{-\ad(\beta)}}{\ad(\beta)}\sslash RC_v^\beta}
    \sslash\diver_v
  - \beta \sslash RC_v^\beta \sslash \diver_v \sslash
    \ad_v^{\beta\sslash\ad_u^\alpha
      \sslash\frac{1-e^{-\ad(\beta)}}{\ad(\beta)}\sslash RC_v^\beta} \\
  + \left[
      \beta \sslash RC_v^\beta,
      \beta \sslash \ad_u^\alpha \sslash
        \frac{1-e^{-\ad(\beta)}}{\ad(\beta)}\sslash RC_v^\beta
    \right] \sslash \diver_v
  - \beta \sslash RC_v^\beta \sslash \diver_v \sslash
    \ad_v^{-\beta\sslash \ad_u^\alpha \sslash
    \frac{1-e^{-\ad(\beta)}}{\ad(\beta)} \sslash RC_v^\beta}
\end{multline*}
In the above, the two terms that do not end in $\diver_v$ cancel each
other. We then remove the $\diver_v$ at the end of all remaining terms,
thus making our quest only harder. Finally we note that $RC_v^\beta$ is a
Lie algebra morphism, so we can pull it out of the bracket in the
penultimate term, getting
\begin{multline*}
  \beta \sslash RC_v^\beta \sslash \ad_u^{\alpha\sslash RC_v^\beta}
  \overset{?}{=}
  \beta \sslash \ad_u^\alpha \sslash e^{-\ad(\beta)}
    \sslash RC_v^\beta \\
  + \beta \sslash RC_v^\beta \sslash \ad_v^{\beta\sslash\ad_u^\alpha
      \sslash\frac{1-e^{-\ad(\beta)}}{\ad(\beta)}\sslash RC_v^\beta}
  + \left[\beta,
      \beta \sslash \ad_u^\alpha \sslash \frac{1-e^{-\ad(\beta)}}{\ad(\beta)}
    \right] \sslash RC_v^\beta
\end{multline*}

The bracketing with $\beta$ in the last term above cancels the $\ad(\beta)$
denominator there, and then that term combines with the first term of the
right hand side to yield
\[
  \beta \sslash RC_v^\beta \sslash \ad_u^{\alpha\sslash RC_v^\beta}
  \overset{?}{=}
  \beta \sslash \ad_u^\alpha \sslash RC_v^\beta
  + \beta \sslash RC_v^\beta \sslash \ad_v^{\beta\sslash\ad_u^\alpha
      \sslash\frac{1-e^{-\ad(\beta)}}{\ad(\beta)}\sslash RC_v^\beta}
\]
We make our task harder again,
\[
  RC_v^\beta \sslash \ad_u^{\alpha\sslash RC_v^\beta}
  \overset{?}{=} \ad_u^\alpha \sslash RC_v^\beta 
  + RC_v^\beta \sslash \ad_v^{\beta\sslash\ad_u^\alpha
    \sslash\frac{1-e^{-\ad(\beta)}}{\ad(\beta)}\sslash RC_v^\beta}
\]
and then we both pre-compose and post-compose with the isomorphism
$C_v^{-\beta}$, getting
\[
  \ad_u^{\alpha\sslash RC_v^\beta} \sslash C_v^{-\beta}
  \overset{?}{=} C_v^{-\beta} \sslash \ad_u^\alpha
  + \ad_v^{\beta\sslash\ad_u^\alpha
    \sslash\frac{1-e^{-\ad(\beta)}}{\ad(\beta)}\sslash RC_v^\beta}
    \sslash C_v^{-\beta}
\]
The above is Equation~\eqref{eq:InfinitesimalCuCv}, with $\alpha$ replaced
by $-\alpha$, and hence it holds true. \qed

\noindent{\em Proof of Equation~\eqref{eq:JtProperty}.}
As before, the equation clearly holds at $\gamma=0$, so we take its radial
derivative. That of the left hand side is
\[ \gamma \sslash tm^{uv}_w \sslash RC_w^{\gamma\sslash tm^{uv}_w}
    \sslash \diver_w \sslash C_w^{-\gamma\sslash tm^{uv}_w}
\]
Using Equation~\eqref{eq:nts-t} and then Proposition~\ref{prop:divpropt},
this becomes
\[ \gamma \sslash RC_u^\gamma \sslash RC_v^{\gamma\sslash RC_u^\gamma}
    \sslash (\diver_u+\diver_v) \sslash tm^{uv}_w \sslash
    C_w^{-\gamma\sslash tm^{uv}_w}.
\]
Now using the reverse of Equation~\eqref{eq:nts-t}, proven by reading the
horizontal arrows within its proof backwards, this becomes
\[ \gamma \sslash RC_u^\gamma \sslash RC_v^{\gamma\sslash RC_u^\gamma}
    \sslash (\diver_u+\diver_v) \sslash C_v^{-\gamma\sslash RC_u^\gamma}
    \sslash C_u^{-\gamma} \sslash tm^{uv}_w.
\]

On the other hand, the radial variation of the right hand side
of~\eqref{eq:JtProperty} is
\begin{multline*}
  \gamma \sslash RC_u^\gamma \sslash \diver_u \sslash C_u^{-\gamma}
    \sslash tm^{uv}_w
  + \gamma \sslash RC_u^\gamma \sslash RC_v^{\gamma\sslash RC_u^\gamma}
    \sslash \diver_v \sslash C_v^{-\gamma\sslash RC_u^\gamma} \sslash
    C_u^{-\gamma} \sslash tm^{uv}_w \\
  + \gamma \sslash RC_u^\gamma \sslash \ad_u^{\gamma\sslash RC_u^\gamma}
    \sslash \frac{1-e^{-\ad(\gamma\sslash RC_u^\gamma)}}
      {\ad(\gamma\sslash RC_u^\gamma)}
    \sslash RC_v^{\gamma\sslash RC_u^\gamma} \sslash \diver_v \sslash
    C_v^{-\gamma\sslash RC_u^\gamma} \sslash C_u^{-\gamma} \sslash
    t^{uv}_w \\
  + \gamma \sslash RC_u^\gamma \sslash J_v \sslash
    \ad_u^{-\gamma\sslash RC_u^\gamma} \sslash C_u^{-\gamma} \sslash
    t^{uv}_w
\end{multline*}

Equating the last two formulae while eliminating the common term (the
second term in each) and removing all trailing $C_u^{-\gamma}\sslash
t^{uv}_w$'s (thus making the quest harder), we need to show that
\begin{multline*}
  \gamma \sslash RC_u^\gamma \sslash RC_v^{\gamma\sslash RC_u^\gamma}
    \sslash \diver_u \sslash C_v^{-\gamma\sslash RC_u^\gamma}
  = \gamma \sslash RC_u^\gamma \sslash \diver_u \\
  + \gamma \sslash RC_u^\gamma \sslash \ad_u^{\gamma\sslash RC_u^\gamma}
    \sslash \frac{1-e^{-\ad(\gamma\sslash RC_u^\gamma)}}
      {\ad(\gamma\sslash RC_u^\gamma)}
    \sslash RC_v^{\gamma\sslash RC_u^\gamma} \sslash \diver_v \sslash
    C_v^{-\gamma\sslash RC_u^\gamma} \\
  + \gamma \sslash RC_u^\gamma \sslash J_v \sslash
    \ad_u^{-\gamma\sslash RC_u^\gamma}.
\end{multline*}
Nicely enough, the above is Equation~\eqref{eq:Jad} with
$\alpha=\beta=\gamma\sslash RC_u^\gamma$. \qed

\subsection{Notational Conventions and Glossary} \label{subsec:Conventions}

For $n\in\bbN$ let $\underline{n}$ denote some fixed set with $n$
elements, say $\{1,2,\dots,n\}$.

Often within this paper we use postfix notation for
operator evaluations, so $f(x)$ may also be denoted $x\sslash
f$.  Even better, we use $f\sslash g$ for ``composition
done right'', meaning $f\sslash g=g\circ f$, meaning that if
$X\overset{f}{\longrightarrow}Y\overset{g}{\longrightarrow}Z$ then
$X\overset{f\sslash g}{\longrightarrow}Z$ rather than the uglier
(though equally correct) $X\overset{g\circ f}{\longrightarrow}Z$. We
hope that this notation will be adopted by others, to be used alongside
and eventually instead of $g\circ f$, much as we hope that $\tau$ will be
used alongside and eventually instead of the presently popular
$\pi:=\tau/2$. In \LaTeX, $\sslash=\verb$\sslash$\in\verb$stmaryrd.sty$$.

In the few paragraphs that follow, $X$ is an arbitrary set. Though within
this paper such $X$'s will usually be finite, and their elements will
thought of as ``labels''. Hence if $f\in G^X$ is a function $f\colon X\to
G$ where $G$ is some other set, we think of $f$ as a collection of elements
of $G$ labelled by the elements of $X$. We often write $f_x$ to denote
$f(x)$.

If $f\in G^X$ and $x\in X$, we let $f\remove x$ denote the restricted
function $f|_{X\remove x}$ in which $x$ is removed from the domain of
$f$. In other words, $f\remove x$ is ``the collection $f$, with the
element labelled $x$ removed''. We often neglect to state the condition
$x\in X$. Thus when writing $f\remove x$ we implicitly assume that
$x\in X$.

Likewise, we write $f\remove\{x,y\}$ for ``$f$ with $x$ and $y$ removed
from its domain'' and as before this includes the implicit assumption that
$\{x,y\}\subset X$.

If $f_1\colon X_1\to G$ and $f_2\colon X_2\to G$ and $X_1$ and $X_2$ are
disjoint, we denote by $f\cup g$ the obvious ``union function'' with domain
$X_1\cup X_2$ and range $G$. In fact, whenever we write $f\cup g$, we make
the implicit assumption that the domains of $f_1$ and $f_2$ are disjoint.

In the spirit of ``associative arrays'' as they appear in various computer
languages, we use the notation $(x\to a,\,y\to b,\dots)$ for
``inline function definition''. Thus $()$ is the empty function, and if
$f=(x\to a,\,y\to b)$, then the domain of $f$ is $\{x,y\}$ and
$f_x=a$ and $f_y=b$.

We denote by $\sigma^x_y$ the operation that renames the key $x$ in an
associative array to $y$. Namely, if $f\in G^X$, $x\not\in X$, and
$y\not\in X\remove x$, then
\[ \sigma^x_y f = (f\remove x)\cup(y\to f_x). \]

\noindent{\bf Glossary of Notations.} (Greek letters, then Latin, then
symbols)

\noindent
{\small \begin{multicols}{2}
\begin{list}{}{
  \renewcommand{\makelabel}[1]{#1\hfil}
}

\item[{$\alpha,\beta,\gamma$}] Free Lie series\hfill Sec.~\ref{sec:trees}
\item[{$\alpha,\beta,\gamma,\delta$}] Matrix parts\hfill
  Sec.~\ref{subsec:Ultimate}
\item[{$\beta$}] A repackaging of $\beta$\hfill Sec.~\ref{subsec:Ultimate}
\item[{$\beta_0$}] A reduction of $M$\hfill Sec.~\ref{subsec:beta0}
\item[{$\delta$}] A map $\uT/\vT/\wT\to\Kbh$\hfill Sec.~\ref{subsec:delta}
\item[{$\delta\alpha,\delta\beta,\delta\gamma$}] Infinitesimal free Lie
  series\hfill Sec.~\ref{subsec:JProperties}
\item[{$\epsilon_a$}] Units\hfill Sec.~\ref{subsec:WisMG}
\item[{$\Pi$}] The MMA ``of groups''\hfill Sec.~\ref{sec:Pi}
\item[{$\pi$}] The fundamental invariant\hfill Sec.~\ref{subsec:pi}
\item[{$\pi$}] The projection $\Kbhz\to\Kbh$\hfill Prop.~\ref{prop:pidelta}
\item[{$\rho^\pm_{ux}$}] $\pm$-Hopf links in 4D\hfill Ex.~\ref{exa:generators}
\item[{$\sigma^x_y$}] Re-labelling\hfill Sec.~\ref{subsec:Conventions}
\item[{$\tau$}] Tensorial interpretation map\hfill
  Sec.~\ref{subsec:TensorialInterpretation}
\item[{$\omega$}] The wheels part of $M$/$\zeta$\hfill Sec.\ref{sec:zeta}
\item[{$\omega$}] The scalar part in $\beta/\beta_0$\hfill
  Sec.~\ref{subsec:beta0}
\item[{$\Upsilon$}] Capping and sliding\hfill
  Sec.\ref{subsec:topdelta}
\item[{$\zeta$}] The main invariant\hfill Sec.~\ref{sec:zeta}
\item[{$\zeta_0$}] The tree-level invariant\hfill Sec.~\ref{sec:trees}
\item[{$\zeta^\beta$}] A $\beta$-valued invariant\hfill
  Sec.~\ref{subsec:Ultimate}
\item[{$\zeta^{\beta_0}$}] A $\beta_0$-valued invariant\hfill
  Sec.~\ref{subsec:beta0}

\item
\item[{$A$}] The matrix part in $\beta/\beta_0$\hfill Sec.~\ref{subsec:beta0}
\item[{$a,b,c$}] Strand labels\hfill Sec.~\ref{subsec:delta}
\item[{$\ad_u^\gamma,\,\ad_u\{\gamma\}$}] Derivations of $\FL$\hfill
  Def.~\ref{def:adu}
\item[{$\Abh$}] Space of arrow diagrams\hfill
  Sec.~\ref{subsec:ArrowDiagrams}
\item[{$\bch$}] Baker-Campbell-Hausdorff\hfill Sec.~\ref{subsec:FLSuccess}
\item[{$C_u^\gamma$}] Conjugating a generator\hfill Sec.~\ref{subsec:FLSuccess}
\item[{$\CA$}] Circuit algebra\hfill Sec.~\ref{subsec:CA}
\item[{$\CW$}] Cyclic words\hfill Sec.~\ref{subsec:divJ}
\item[{$\CW^r$}] $\CW$ mod degree 1\hfill Sec.~\ref{subsec:divJ}
\item[{$c$}] A ``sink'' vertex\hfill Sec.~\ref{subsec:Informal}
\item[{$c_u$}] A ``$c$-stub''\hfill Sec.~\ref{subsec:Informal}
\item[{$\diver_u$}] The ``divergence'' $\FL\to\CW$\hfill
  Sec.~\ref{subsec:divJ}
\item[{$dm^{ab}_c$}] Double/diagonal multiplication\hfill
  Sec.~\ref{subsec:WisMG}
\item[{$\FA$}] Free associative algebra\hfill Sec.~\ref{subsec:divJ}
\item[{$\FL$}] Free Lie algebra\hfill Sec.~\ref{subsec:FLSuccess}
\item[{$\Fun(X\to Y)$}] Functions $X\to Y$\hfill
  Sec.~\ref{subsec:TensorialInterpretation}
\item[{$H$}] Set of head / hoop labels\hfill Sec.~\ref{sec:objects}
\item[{$h\epsilon_x$}] Units\hfill
  Ex.~\ref{exa:generators}, Sec.~\ref{subsec:FLSuccess},\ref{subsec:MMAM}
\item[{$h\eta$}] Head delete\hfill
  Sec.~\ref{sec:Operations},\ref{subsec:FLSuccess},\ref{subsec:MMAM}
\item[{$hm^{xy}_z$}] Head multiply\hfill
  Sec.~\ref{sec:Operations},\ref{subsec:FLSuccess},\ref{subsec:MMAM}
\item[{$h\sigma^x_y$}] Head re-label\hfill
  Sec.~\ref{sec:Operations},\ref{subsec:FLSuccess},\ref{subsec:MMAM}
\item[{$J_u$}] The ``spice'' $\FL\to\CW$\hfill Sec.~\ref{subsec:divJ}
\item[{$\Kbh$}] All rKBHs\hfill Def.~\ref{def:KBH}
\item[{$\Kbhz$}] Conjectured version of $\Kbh$\hfill Sec.~\ref{subsec:genrels}
\item[{$l_{ux}$}] 4D linking numbers\hfill Sec.~\ref{subsec:Hopf}
\item[{$l_x$}] Longitudes\hfill Sec.~\ref{subsec:pi}
\item[{$M$}] The ``main'' MMA\hfill Sec.~\ref{subsec:MMAM}
\item[{$M_0$}] The MMA of trees\hfill Sec.~\ref{subsec:FLSuccess}
\item[{MMA}] Meta-monoid-action\hfill Def.~\ref{def:MMA}, Sec.~\ref{sssec:MMA}
\item[{$m_u$}] Meridians\hfill Sec.~\ref{subsec:pi}
\item[{$m^{ab}_c$}] Strand concatenation\hfill Sec~\ref{subsec:WisMG}
\item[{OC}] Overcrossings commute\hfill Fig.~\ref{fig:OCUC}
\item[{$\Pbh$}] Primitives of $\Abh$\hfill Sec.~\ref{subsec:Primitives}
\item[{$R$}] Ring of $c$-stubs\hfill Sec.~\ref{subsec:LessInformal}
\item[{$R^r$}] $R$ mod degree 1\hfill Sec.~\ref{subsec:beta0}
\item[{R1,R1',R2,R3}] Reidemeister moves\hfill
  Sec.~\ref{subsec:delta},~\ref{subsec:CA}
\item[{$RC_u^\gamma$}] Repeated $C_u^\gamma$ / reverse $C_u^{-\gamma}$\hfill
  Sec.~\ref{subsec:FLSuccess}
\item[{rKBH}] Ribbon knotted balloons\&hoops\hfill Def.~\ref{def:KBH}
\item[{$S$}] Set of strand labels\hfill Sec.~\ref{subsec:delta}
\item[{$T$}] Set of tail / balloon labels\hfill Sec.~\ref{sec:objects}
\item[{$t\epsilon^u$}] Units\hfill
  Ex.~\ref{exa:generators}, Sec.~\ref{subsec:FLSuccess},\ref{subsec:MMAM}
\item[{$tha^{ux}$}] Tail by head action\hfill
  Sec.~\ref{sec:Operations},\ref{subsec:FLSuccess},\ref{subsec:MMAM}
\item[{$t\eta^u$}] Tail delete\hfill
  Sec.~\ref{sec:Operations},\ref{subsec:FLSuccess},\ref{subsec:MMAM}
\item[{$tm^{uv}_w$}] Tail multiply\hfill
  Sec.~\ref{sec:Operations},\ref{subsec:FLSuccess},\ref{subsec:MMAM}
\item[{$t\sigma^x_y$}] Tail re-label\hfill
  Sec.~\ref{sec:Operations},\ref{subsec:FLSuccess},\ref{subsec:MMAM}
\item[{$t,x,y,z$}] Coordinates\hfill Sec.~\ref{sec:objects}
\item[{UC}] Undercrossings commute\hfill Fig.~\ref{fig:OCUC}
\item[{u-tangle}] A usual tangle\hfill Sec.~\ref{subsec:delta}
\item[{$\uT$}] All u-tangles\hfill Sec.~\ref{subsec:delta}
\item[{$u,v,w$}] Tail / balloon labels\hfill Sec.~\ref{sec:objects}
\item[{v-tangle}] A virtual tangle\hfill Sec.~\ref{subsec:vw}
\item[{$\vT$}] All v-tangles\hfill Sec.~\ref{subsec:vw}
\item[{w-tangle}] A virtual tangle mod OC\hfill Sec.~\ref{subsec:vw}
\item[{$\wT$}] All w-tangles\hfill Sec.~\ref{subsec:vw}
\item[{$x,y,z$}] Head / hoop labels\hfill Sec.~\ref{sec:objects}
\item[{$\Zbh$}] An $\Abh$-valued expansion\hfill Sec.~\ref{subsec:Zbh}

\item
\item[{$\ast$}] Merge
  operation\hfill Sec.~\ref{sec:Operations},\ref{subsec:FLSuccess},\ref{subsec:MMAM}
\item[{$\sslash$}] Composition done right\hfill Sec.~\ref{subsec:Conventions}
\item[{$x\sslash f$}] Postfix evaluation\hfill Sec.~\ref{subsec:Conventions}
\item[{$f\remove x$}] Entry removal\hfill Sec.~\ref{subsec:Conventions}
\item[{$x\to a$}] Inline function definition\hfill Sec.~\ref{subsec:Conventions}
\item[{$\overbracket[0.5pt][1pt]{uv}$}] ``Top bracket form''\hfill
  Sec.~\ref{sec:Computations}
\item[{$\overbrace[L1R]{uv}$}] A cyclic word\hfill Sec.~\ref{sec:Computations}

\end{list}
\end{multicols}}

\subsection{Acknowledgement} I wish to thank the people of the Knot at
Lunch seminar in Toronto (Oleg Chterental, Karene Chu, Zsuzsanna Dancso,
Peter Lee, Stephen Morgan, James Mracek, David Penneys, Peter Samuelson,
and Sam Selmani) for hearing this several times and for making comments
and suggestions. Further thanks to Zsuzsanna Dancso for pushing me to
write this, to Martin Bridson for telling me about LOT groups and to
Eckhard Meinrenken and Jim Stasheff for further comments. This work was
partially supported by NSERC grant RGPIN 262178, by ITGP (Interactions
of Low-Dimensional Topology and Geometry with Mathematical Physics),
an ESF RNP, and by the Newton Institute in Cambridge, UK.

\vfill

\noindent$\left.\parbox{5.2in}{
  Next two pages: The handout for a talk on this paper, given in Chicago
  in March 2013. A video recording of that talk is at \web{chic2}. Older
  versions of handout/talk/video are at \web{ham}, \web{ox}, and
  \web{tor}, and a slightly newer version is at \web{viet}.

  At end: A copy of \web{}.
}\quad\right\}\quad$\parbox{0.7in}{
  omitted in arXiv version
}

\end{document}